\documentclass[reqno]{amsart}

\usepackage{amssymb}
\usepackage{amsmath}
\usepackage{amsthm}
\usepackage[usenames]{color}
\usepackage{graphicx}

\usepackage{bbm}

\usepackage{todonotes}

\usepackage[colorlinks,linkcolor=blue,anchorcolor=red,citecolor=blue]{hyperref}

\usepackage[
margin=1in,
marginparwidth=0.8in,
marginparsep=0.1in
]{geometry}

\usepackage{mathtools}
\usepackage{amssymb}
\usepackage{amsthm, thmtools}
\usepackage{mathrsfs}

\declaretheorem[parent=section]{theorem}
\declaretheorem[sibling=theorem]{lemma}

\declaretheorem[parent=section, style=definition]{definition}
\declaretheorem[parent=section, style=remark]{remark}

\usepackage{marginnote}

\makeatletter
\@namedef{subjclassname@2020}{2020 Mathematics Subject Classification}
\makeatother

\title[Inelastic Boltzmann equation under shear heating]{Inelastic Boltzmann equation under shear heating}

\author[J. A. Carrillo]{Jos\'{e} A. Carrillo}
\address[JAC]{Mathematical Institute, University of Oxford, Oxford OX2 6GG, UK}
\email{carrillo@maths.ox.ac.uk}

\author[K. F. Chan]{Kam Fai Chan}
\address[KFC]{Department of Mathematics, The Chinese University of Hong Kong, Shatin, Hong Kong}
\email{akfchan@math.cuhk.edu.hk}

\author[R.-J. Duan]{Renjun Duan}
\address[RJD]{Department of Mathematics, The Chinese University of Hong Kong, Shatin, Hong Kong}
\email{rjduan@math.cuhk.edu.hk}

\author[Z.-G. Li]{Zongguang Li}
\address[ZGL]{Department of Applied Mathematics, The Hong Kong Polytechnic University, Hong Kong}
\email{zongguang.li@polyu.edu.hk}

\begin{document}

    \begin{abstract}
        In this paper, we study the spatially homogeneous inelastic Boltzmann equation for the angular cutoff pseudo-Maxwell molecules with an additional term of linear deformation. We establish the existence of non-Maxwellian self-similar profiles under the assumption of small deformation in the nearly elastic regime, and also obtain weak convergence to these self-similar profiles for global-in-time solutions with initial data that have finite mass and finite $p$-th order moment for any $2<p\leq 4$. Our results confirm the competition between shear heating and inelastic cooling that governs the  large time behavior of temperature. Specifically, temperature increases to infinity if shear heating dominates, decreases to zero if inelastic cooling prevails, and converges to a positive constant if the two effects are balanced. In the balanced scenario, the corresponding self-similar profile aligns with the steady solution.
    \end{abstract}

    \date{\today}


    \maketitle


    \thispagestyle{empty}

    \section{Introduction}
    The initial value problem {of} the spatially homogeneous inelastic Boltzmann equation under the effect of shear heating in \( \mathbb{R}^d \) modeled by the matrix \( A \in \mathbb{R}^{d \times d} \) with $d \geq 2$ is given by
    \begin{equation} \label{equ:inelasticBoltzmann}
        {\partial_{t}} f - { \nabla_{v} } \cdot \!(A v f) = Q_e(f, f), \quad f(0, v) = f_0(v).
    \end{equation}
    Here, the unknown function \( f = f(t, v) \geq 0 \) is the one-particle density distribution of rarefied gas particles with velocity \( v \in \mathbb{R}^d \) at time \( t \geq 0 \), and the initial data \( f(0, v) = f_0(v) \geq 0 \) is given. The inelastic Boltzmann collision operator \( Q_e \) acting only on {the} velocity variable involves a parameter \( {e_\mathrm{res}} \in (0, 1] \) called the coefficient of restitution, which we assume to be constant through the paper. In the weak form, $Q_e(f, f)$ is defined as
    \begin{align}\label{weakform}
        {\int_{\mathbb{R}^d}^{}} \psi(v) Q_e(f, f)(v) \, {\mathrm{d}{v}}
        &= \frac{1}{2} {\int_{\mathbb{R}^d}^{}} 
            {\int_{\mathbb{R}^d}^{}} 
                {\int_{\mathbb{S}^{d-1}}^{}} 
                    \mathrm{B}(v - v_\ast, \sigma) f(v) f(v_*) (\psi(v') + \psi(v_*') - \psi(v) - \psi(v_*))
                 \, {\mathrm{d}{\sigma}}
            {\mathrm{d}{v_*}}
        {\mathrm{d}{v}},
    \end{align}
    where the pre-collision velocity pair $(v, v_\ast)$ and the post-collision velocity pair $(v', v_\ast')$ satisfy
    \begin{equation*}
        \left\{\begin{aligned}
            v' &= \frac{v + v_*}{2} + \frac{1 - z}{2} (v - v_*) + \frac{z}{2} |v - v_*| \sigma, \\
            v_*' &= \frac{v + v_*}{2} - \frac{1 - z}{2} (v - v_*) - \frac{z}{2} |v - v_*| \sigma,
        \end{aligned}\right.
    \end{equation*}
    with $\sigma\in \mathbb{S}^{d-1}$ and 
    {\begin{align}\label{Defe}
    	z := \frac{1}{2} (1 + {e_\mathrm{res}}) \in (\frac{1}{2}, 1]. 
    \end{align}} When \( {e_\mathrm{res}} = 1 \), which corresponds to \( z = 1 \), the collision operator reverts back to the classical elastic case.

    The Boltzmann collision kernel \( \mathrm{B}(v - v_\ast, \sigma) \) in \eqref{weakform}  takes the form
    \begin{equation}\label{DefB}
        \mathrm{B}(v - v_\ast, \sigma) = \vert v-v_*\vert^{\gamma_0} b(\cos\theta),\  {-d}<\boldsymbol{\gamma}\leq1,
    \end{equation}
    where
    \begin{equation*}
        \cos \theta=\frac{v-v_*}{\vert v-v_*\vert}\cdot \sigma, \  0<\theta \leq\pi/2.
    \end{equation*}
    The cases ${-d}<\boldsymbol{\gamma}<0$, $\boldsymbol{\gamma}=0$ and $0<\boldsymbol{\gamma}\leq1$ are called soft potentials, pseudo-Maxwell molecules and hard potentials, respectively. Throughout the paper, we only focus on the case of pseudo-Maxwell molecules with $\boldsymbol{\gamma}=0$. The angular part of the collision kernel \( b \geq 0 \) is assumed to be a continuous function and satisfy the Grad cutoff assumption
    \begin{equation*}
        b_n = {\int_{\mathbb{S}^{d-1}}^{}} 
            b(\hat{e} \cdot \sigma) (\hat{e} \cdot \sigma)^n
         \, {\mathrm{d}{\sigma}} < \infty,
    \end{equation*}
    for all unit vectors \( \hat{e} \in \mathbb{S}^{d-1} \) and all natural numbers \( n \in \mathbb{N} \). We further assume that the kernel is non-degenerate in the sense that 
    \begin{align}\label{Defnondegenerate}
    {b_0 > b_2 > 0,\qquad b_0 > b_1 .}
    \end{align}
{We remark that the assumptions on Maxwell molecules and also on the homogeneous regime for granular media via the inelastic Boltzmann equation might be seen as an artificial ansatz, cf.~\cite{Villani06MGM} and \cite{ALT}. Moreover, the assumption that the restitution
coefficient is a fixed constant could be not seen as very physical as it may cause a problem to define the dynamics of the microscopic particle system (inelastic collapse), cf.~\cite{AH} and \cite{BP}.} 

    The shearing matrix \( A \in \mathbb{R}^{d \times d} \) is a constant real matrix with the norm \( \|A\| = \sup_{|v| = 1} |A v| \), which includes the special case of the simple uniform shear flow (USF) 
    {\begin{align}\label{DefUSF}
    A = \alpha E_{12},\qquad(E_{12})_{ij} = \delta_{i1} \delta_{2j} 
    \end{align}}
    and shearing parameter \( \alpha \geq 0 \){, where $\delta_{ij}$ denotes the Kronecker delta}.

    Through the paper we also assume that the initial value \( f_0 = f_0(v) \) satisfies
    \begin{equation*}
        {\int_{\mathbb{R}^d}^{}} f_0(v) \, {\mathrm{d}{v}} = 1 ,\quad
        {\int_{\mathbb{R}^d}^{}} v f_0(v) \, {\mathrm{d}{v}} = 0.
    \end{equation*}
    By the fact that \( 1, v \) are collision invariants of the inelastic collision operator and vanish on the shearing term after integration, it is well-known that the mass and momentum of the solution to \eqref{equ:inelasticBoltzmann} are conserved. However, the energy is not necessarily conserved since inelastic collision introduces a loss of energy, and the shearing term is not guaranteed to balance this effect. In fact, there is a competition between shear heating and inelastic cooling that governs the  large time behavior of temperature. To further study this phenomenon, one needs to define the natural cooling rate \( \zeta \geq 0 \) as
    \begin{equation}\label{Defzeta}
        \zeta
        = z (1 - z) (b_0 - b_1)
        = 2 z (1 - z) {\int_{\mathbb{S}^{d-1}}^{}} b(\cos\theta) \sin^2\frac{\theta}{2} \, {\mathrm{d}{\sigma}},
    \end{equation}
    such that
    \begin{equation*}
        -\zeta |v - v_*|^2
        = {\int_{\mathbb{S}^{d-1}}^{}} b(\cos\theta) (|v'|^2 + |v_*'|^2 - |v|^2 - |v_*|^2) \, {\mathrm{d}{\sigma}},
    \end{equation*}
    which implies, in the case of no-shearing \( A = 0 \), that
    \begin{align*}
        \frac{\mathrm{d}}{{\mathrm{d} t}} {\int_{\mathbb{R}^d}^{}} |v|^2 f(t, v) \, {\mathrm{d}{v}}
        &= \frac{1}{2} {\int_{\mathbb{R}^d}^{}} 
            {\int_{\mathbb{R}^d}^{}} 
                {\int_{\mathbb{S}^{d-1}}^{}} 
                    b(\cos\theta) f(t,v) f(t,v_*) (|v'|^2 + |v_*'|^2 - |v|^2 - |v_*|^2)
                 \, {\mathrm{d}{\sigma}}
            {\mathrm{d}{v}}
        {\mathrm{d}{v}} \\
        &= -\frac{1}{2} \zeta {\int_{\mathbb{R}^d}^{}} 
            {\int_{\mathbb{R}^d}^{}} 
                f(t,v) f(t,v_*) |v - v_*|^2
             \, {\mathrm{d}{v_*}}
        {\mathrm{d}{v}} \\
        &= -\zeta {\int_{\mathbb{R}^d}^{}} |v|^2 f(t, v) \, {\mathrm{d}{v}}.
    \end{align*}
    It is also necessary to denote the constant
    \begin{equation}\label{Defc11}
        c_{11}
        = \frac{b_0 - b_2}{d - 1}
        = \frac{1}{d - 1} {\int_{\mathbb{S}^{d-1}}^{}} b(\cos\theta) \sin^2\theta \, {\mathrm{d}{\sigma}}
        \geq 0, 
    \end{equation}
    such that, by a direct calculation,
    \begin{align}
        {\int_{\mathbb{S}^{d-1}}^{}} 
            b(\sigma \cdot \hat{e})
         \, {\mathrm{d}{\sigma}}
        &= b_0,\label{b0} \\
        {\int_{\mathbb{S}^{d-1}}^{}} 
            b(\sigma \cdot \hat{e}) \sigma
         \, {\mathrm{d}{\sigma}}
        &= b_1 \hat{e},\label{b1} \\
        {\int_{\mathbb{S}^{d-1}}^{}} 
            b(\sigma \cdot \hat{e}) {{\sigma}^{\otimes 2}}
         \, {\mathrm{d}{\sigma}}
        &= (b_0 - d c_{11}) {\hat{e}}^{\otimes 2} + c_{11} I,\label{bsigma2}
    \end{align}
    for any unit vector \( \hat{e} \in \mathbb{S}^{d-1} \), with the notation \( {{v}^{\otimes 2}} = v \otimes v \) and \( v \otimes w = v w^\mathsf{T} \), where $w^\mathsf{T}$ is the transpose of the column vector $w$.

    In the paper, we aim at studying the long-time behavior of solutions to \eqref{equ:inelasticBoltzmann} under the competition of the shear heating and the inelastic cooling. Such competitive phenomenon was observed and discussed {in \cite{Cercignani01SFG,Garzo19GGF}} and our goal is to make a rigorous mathematical analysis. Note that two effects have been extensively investigated by \cite{GarzoSantos03KTG, JamesEtAl17SSP, BobylevEtAl20SSA} and \cite{CarlenEtAl09SCT} in their individual frameworks. In particular, solutions in both situations behave self-similarly in large time. Thus, a unified framework is developed to study the competition of two effects on the large time behavior solutions. For the purpose, we shall consider the equation under self-similar scaling
    \begin{equation*}
        \tilde{f}(t, \tilde{v}) = e^{d \beta t} f(t, e^{\beta t} \tilde{v})
    \end{equation*}
    with some suitable self-similar parameter \( \beta \in \mathbb{R} \) to be determined. This gives
    \begin{equation} \label{equ:inelasticBoltzmannSelfSim}
        {\partial_{t}} \tilde{f} - { \nabla_{\tilde{v}} } \cdot \!((A + \beta I) \tilde{v} \tilde{f}) = Q_e(\tilde{f}, \tilde{f}).
    \end{equation}
    The resulting equation is very similar to \eqref{equ:inelasticBoltzmann} with \( A \) replaced by the altered shearing matrix
    \begin{align}\label{DefAbeta}
        A_\beta \coloneqq A + \beta I.\end{align} Hence, for simplicity, we shall drop the tilde in \eqref{equ:inelasticBoltzmannSelfSim} and still denote $\tilde{f},\tilde{v}$ by $f,v$. It is then natural to see the stationary profile \( G \) to \eqref{equ:inelasticBoltzmannSelfSim}, which satisfies
    \begin{equation} \label{equ:inelasticBoltzmannSelfSimStationary}
        -{ \nabla_{v} } \cdot \!(A_\beta v G) = Q_e(G, G),
    \end{equation}
    should describe the long-time behavior of solution to \eqref{equ:inelasticBoltzmannSelfSim} or equivalently \eqref{equ:inelasticBoltzmann}. Note that in the special case $\beta=0$, \eqref{equ:inelasticBoltzmannSelfSimStationary} just corresponds to the steady {state} of \eqref{equ:inelasticBoltzmann}.

    To study \eqref{equ:inelasticBoltzmann} in the frequency space, we take {the} Fourier transform \( \varphi(k) = (\mathcal{F} f)(k) = {\int_{\mathbb{R}^d}^{}} e^{-i k \cdot v} f(v) \, {\mathrm{d}{v}} \) in \eqref{equ:inelasticBoltzmann}. Recall $\boldsymbol{\gamma}=0$. We then {formally} obtain equation for $\varphi$ as
    \begin{equation} \label{equ:inelasticBobylev}
        {\partial_{t}} \varphi + A^\mathsf{T} k \cdot { \nabla_{k} } \varphi = \widehat{Q_e}(\varphi, \varphi).
    \end{equation}
    Here $A^\mathsf{T}$ is the transpose of $A$ and the collision operator $\widehat{Q_e}$ as in \cite{BobylevCercignani03SSA} can be split as
    \begin{equation}\label{DefhatQ}
        \widehat{Q_e}(\varphi, \varphi)(k)
        = {\int_{\mathbb{S}^{d-1}}^{}} 
            b\big(\sigma \cdot \frac{k}{|k|}\big) (\varphi(k^+) \varphi(k^-) - \varphi(0) \varphi(k))
         \, {\mathrm{d}{\sigma}}
        = \widehat{Q^+_e}(\varphi,\varphi)(k) - b_0 \varphi(k),
    \end{equation}
    for Maxwell molecules, where {the negative term is obtained in terms of the mass conservation,}
    \begin{equation*}
        \widehat{Q^+_e}(\varphi, \psi)(k)
        = {\int_{\mathbb{S}^{d-1}}^{}} b(\cos\vartheta) \varphi(k^+) \psi(k^-) \, {\mathrm{d}{\sigma}}
    \end{equation*}
    is the bilinear gain part operator, \( \cos\vartheta = \frac{k}{|k|} \cdot \sigma \), and
    \begin{align*}
        k^- &= \frac{z}{2} (k - |k| \sigma),\\
        k^+ &= k - k^- = (1 - \frac{z}{2}) k + \frac{z}{2} |k| \sigma,
    \end{align*}
    are the corresponding Fourier variables.

    The linearized gain part operator \( \mathcal{L}_e \), linearized around the constant function \( 1 \), is defined as
\begin{equation}\label{DefLe}
        \mathcal{L}_e\varphi
        = \widehat{Q_e^+}(\varphi, 1) + \widehat{Q_e^+}(1, \varphi)
        = {\int_{\mathbb{S}^{d-1}}^{}} b(\cos\vartheta) (\varphi(k^+) + \varphi(k^-))  \, {\mathrm{d}{\sigma}}.
    \end{equation}
Similar to how we obtain \eqref{equ:inelasticBobylev}, after applying Fourier transform to \eqref{equ:inelasticBoltzmannSelfSim}, one gets
    \begin{equation} \label{equ:inelasticBobylevSelfSim}
        {\partial_{t}} \varphi + A_\beta^\mathsf{T} k \cdot { \nabla_{k} } \varphi = \widehat{Q_e}(\varphi, \varphi),
    \end{equation}
    which has the similar form to \eqref{equ:inelasticBobylev} but with a replaced shearing matrix.

    Our motivation to study \eqref{equ:inelasticBoltzmann} also originates from the recent study of homo-energetic solutions of the elastic Boltzmann equation, as introduced by \cite{Truesdell56PFE} and \cite{Galkin58CSG}, where the second and the third order moments are computed. The behavior of these moments are later investigated in numerous works including \cite{Nikolskii63TDH, Nikolskii63SES, TruesdellMuncaster80FMK, Dufty84DSF, SantosGarzo95EMS, MontaneroEtAl96SBV, AcedoEtAl02DHV, BobylevEtAl04MIH}. The initial value problem, as shown by Cercignani in \cite{Cercignani89EHA}, has a global-in-time \( {\mathit{L}^{1}} \) solution for a class of cutoff hard or pseudo-Maxwell collision kernels and for suitable initial data (cf.~\cite{GarzoSantos03KTG}). More recently, \cite{JamesEtAl17SSP} considers the Radon measure-valued solution and existence of self-similar profiles, \cite{BobylevEtAl20SSA} considers the equation via Fourier transform to establish the self-similar asymptotics of weak solutions in large time, \cite{DuanLiu21BEU, DuanLiu22USF} obtains regularity and the structure of shear dependency of the solution, and \cite{Kepka21SSP, Kepka22LBH} generalizes the results to non-cutoff pseudo-Maxwellian and hard potential.

    On the other hand, the study of inelastic Boltzmann equation, which models the behavior of granular gases, is a quite related topic. In \cite{BobylevEtAl00SPK}, the Boltzmann equation for inelastic pseudo-Maxwell potential is derived, together with the existence and uniqueness of the solution with initial data in \( {\mathit{L}_{2}^{1}} \). The behavior of the solution is studied in \cite{BobylevCercignani03SSA, BobylevEtAl03PAP,BCT1,BCT2,BCL}, justifying the behavior conjectured in \cite{ErnstBrito02SSI}. In \cite{BolleyCarrillo07TTI}, convergence of solutions with pseudo-Maxwell molecule in 2-Wasserstein metric is tackled. A survey of the use of Wassertein and Toscani metrics in this framework was performed in \cite{CTsurvey}. In \cite{MischlerEtAl06CPI, MischlerMouhot06CPI}, global well-posedness of the initial value problem for a generic class of collision kernels is settled, and Haff's law and the self-similar structure for the solution is justified for hard sphere model{, which seems more difficult to treat}. Concerning the interaction between inelastic collision and shearing, an explicit stationary state for the second moments is given in \cite{Cercignani01SFG}, and the result is further discussed  {in \cite{Cercignani02BEA,Garzo19GGF}. See also a recent work \cite{DolmaireMieleNota} for the linear equation.} 

    In this paper, we shall follow the approaches of \cite{JamesEtAl17SSP} and \cite{BobylevEtAl20SSA} and consider \eqref{equ:inelasticBoltzmann} both from the physical space and from the frequency space. More precisely, we will show that the results in \cite{JamesEtAl17SSP, BobylevEtAl20SSA} for the elastic Boltzmann equation, taken as an evolution equation on the space of probability measures, generalize naturally to \eqref{equ:inelasticBoltzmann} with inelastic collisions. In particular, we obtain the following result concerning the existence of stationary profile for inelastic collision.

    \begin{theorem} \label{thm:mainResult-statRadonProfile}
        Assume that \( {e_\mathrm{res}} \in (0, 1] \) {for $e_\mathrm{res}$ given in \eqref{Defe}, and the collision kernel \eqref{DefB}} is cutoff pseudo-Maxwellian. For sufficiently small \( \epsilon > 0 \) ensuring small shear, \( \|A\| < \epsilon \), and small inelasticity, \( 1 - {e_\mathrm{res}} < \epsilon \), 
        there exists \( \beta \in \mathbb{R} \) such that with this self-similar parameter \( \beta \), \eqref{equ:inelasticBoltzmannSelfSimStationary} has a nonnegative Radon-measure solution \( G \in \mathcal{M}_+ \) with finite mass, zero momentum and finite energy. Furthermore, \( G \) has finite \( p \)-moment for some \( p > 2 \). 
    \end{theorem}

    Here we refer to Section \ref{sec:radonSol} for the definition of solutions. Theorem \ref{thm:mainResult-statRadonProfile} will be shown with a fixed point argument using a uniform-in-time \( p \)-moment bound on the solutions of \eqref{equ:inelasticBoltzmannSelfSim}. For technical reasons, instead of \( \mathbb{R}^d \) we will consider the Radon measures on \( \mathbb{R}_c^{d} \), the compatification of \( \mathbb{R}^d \). The selection of appropriate self-similar parameter \( \beta \) will be made by a perturbation argument. As noted in \cite{BobylevEtAl20SSA}, in the case of elastic collision, only a limited selection of \( \beta \) yields a meaningful solution of \eqref{equ:inelasticBoltzmannSelfSim}. We will see shortly that the same restriction applies {to} our model of inelastic interactions.

    We note that the correct choice of self-similar parameter \( \beta \in \mathbb{R} \) that gives the self-similar profile depends on both the shearing matrix \( A \) and the restitution coefficient \( {e_\mathrm{res}} \), and contrary to the case of elastic collision, it is not necessarily positive in the case of uniform shear flow. In fact, \( \beta \) is chosen such that the inelastic cooling effect is balanced by the shear heating effect that comes from the altered shearing matrix. If inelastic cooling is the dominant effect, it would be necessary to choose a negative \( \beta \), which corresponds to an anti-drift force instead of the friction force that corresponds to a positive \( \beta \) (cf.~\cite{Villani06MGM}).

    For the simple uniform shear flow, we can determine the sign of \( \beta \); its proof will be given in Section \ref{sec:discussionOnUSF}.

    \begin{theorem}\label{thm:mainResult-signOfUSFBeta}
        Assume \( {e_\mathrm{res}} \in (0, 1) \). Let \( A = \alpha E_{12} \) be the uniform shear flow matrix with shearing parameter \( \alpha \geq 0 \) and \( \beta \in \mathbb{R} \) be the self-similar parameter from Theorem \ref{thm:mainResult-statRadonProfile} in the {small inelasticity} regime, then there exists a computable strictly positive constant \( \alpha_0>0 \) depending only on \( {e_\mathrm{res}}, d, c_{11} \) such that $\alpha_0\to 0$ as ${e_\mathrm{res}}\to 1-$ and  the following further holds:
        \begin{enumerate}
            \item \( \beta > 0 \) if \( \alpha > \alpha_0 \);
            \item \( \beta < 0 \) if \( \alpha < \alpha_0 \);
            \item \( \beta = 0 \) if \( \alpha = \alpha_0 \).
        \end{enumerate}
        Therefore, \( \alpha_0 \) is a critical shearing parameter where the shear heating effect and the inelastic cooling effect are balanced; in particular, for the balance case $\alpha = \alpha_0$, the self-similar profile determined by \eqref{equ:inelasticBoltzmannSelfSimStationary} is reduced to the steady solution of \eqref{equ:inelasticBoltzmann}.
    \end{theorem}

    Note that for the Cauchy problem \eqref{equ:inelasticBoltzmann}, the usual fixed point approach can be adopted to prove the global existence of a class of measure-valued solutions even for any constant shearing matrix $A$ and any restitution coefficient \( {e_\mathrm{res}} \in (0, 1] \); see Theorem \ref{thm:existenceOfMildSol} later on. With this in hand, we focus on the large time convergence of solutions to those self-similar profiles obtained in Theorem \ref{thm:mainResult-statRadonProfile}. In fact, we have the following result.

    \begin{theorem} \label{thm:mainResult-convToProfile}
        Assume that \( {e_\mathrm{res}} \in (0, 1] \), and {the collision kernel \eqref{DefB}} is cutoff pseudo-Maxwellian. Then for all \( p \in (2, 4] \), there exists \( \epsilon > 0 \) such that if \( \|A\| < \epsilon \) and \( 1 - {e_\mathrm{res}} < \epsilon \), then the following holds:
        {let} \( f \in \mathit{C}^{0}([0, \infty), \mathcal{M}_+) \) be a measure-valued solution to \eqref{equ:inelasticBoltzmann} with initial data \( f_0 \in \mathcal{M}_+ \) being a nonnegative Radon measure with finite mass, zero momentum and finite \( p \)-moment.
        Then there exist \( \beta \in \mathbb{R} \), \( \lambda > 0 \), and a nonnegative probability measure \( G \in \mathcal{M}_+ \) with zero momentum and finite \( p \)-moment such that \( e^{d \beta t} f(t, e^{\beta t} v) \to \lambda^{-d} G(\lambda^{-1} v) \) on \( t \to \infty \) in the weak topology of \( \mathcal{M}_+ \), with a computable convergence rate in \( 2 \)-Toscani metric.
    \end{theorem}

    Here the Toscani metric is introduced in \eqref{def.pTd} later on, see \cite{CTsurvey} for a review of this topic. Theorem \ref{thm:mainResult-convToProfile} will be shown by studying the equation in the frequency space via Fourier transform. We will first show that the same results on Radon measures on \( \mathbb{R}_c^{d} \) have analogs on the characteristic functions of probability measures on \( \mathbb{R}^d \), with additional convergence results in Toscani metrics. The desired result then follows by reverting the Fourier transform with an appropriate scaling. {Moreover, regarding the stationary self-similar profile $G$, we can obtain its existence in both Theorem \ref{thm:existenceOfStatProfileWithSmallness} and Theorem \ref{thm:existenceOfBobStatProfileWithFiniteEnergy} using the different approaches. In the meantime, Theorem \ref{thm:existenceOfBobStatProfileWithFiniteEnergy} also gives the uniqueness of $G$ in an appropriate function space in the Fourier framework.}

    In the case of simple USF \eqref{DefUSF}, we note from Theorem \ref{thm:mainResult-signOfUSFBeta} that for a given ${e_\mathrm{res}}$ in the nearly elastic regime, if the shear parameter \( \alpha \) is exactly \( \alpha_0 \), then Theorem \ref{thm:mainResult-convToProfile} above gives the weak convergence of the solution to the initial value problem on \eqref{equ:inelasticBoltzmann} with appropriate initial data towards a non-degenerate stationary profile defined via \eqref{equ:inelasticBoltzmannSelfSimStationary} with self-similar parameter \( \beta = 0 \), which is just a stationary solution of \eqref{equ:inelasticBoltzmann}. Moreover, temperature of solutions increases to infinity for $\beta>0$ corresponding to the shear-heating dominated case and decreases to zero for $\beta<0$ corresponding to the inelastic-cooling dominated case, both {exponentially fast} as $t\to \infty$.

    We divide the rest of this paper into two parts. {In the first part of Section \ref{sec:radonSol},} we treat \eqref{equ:inelasticBoltzmann} in the framework of nonnegative Radon measures. We first establish {the existence and uniqueness} of global in time solutions as well as propagation of moment bounds. To obtain the existence of a non-degenerate stationary profile, we analyze the behavior of the second moment of Radon solutions in Section \ref{sec:secondMomentEqu}, together with the special case of USF in Section \ref{sec:discussionOnUSF}. The stationary Radon profile is to be obtained in section \ref{sec:statRadonProfile}. In the second part of Section \ref{sec:fourierSol}, we treat \eqref{equ:inelasticBobylev} in frequency space for analogous results on existence of both global in time solutions and self-similar profiles in Section \ref{sec:statFourierProfile}, in particular obtaining the weak convergence of global solutions to self-similar profiles.

    \section{Radon measure-valued solution} \label{sec:radonSol}

    In this part, we consider the mild solutions to \eqref{equ:inelasticBoltzmann} that are Radon measure-valued, see \cite[Chapter 7]{Folland}. We first give the formulation for mild solutions in the inelastic case, which is consistent with \cite{JamesEtAl20LTA} for the classical case. 
    Denote \( \mathcal{M}_+(\mathbb{R}_c^{d}) \) the space of nonnegative Radon measures on \( \mathbb{R}_c^{d} \), which is the one-point compactification of \( \mathbb{R}^d \) appending the point \( \infty \), equipped with {the total variation distance}
    \begin{align}\label{DefM}
    \left\|f\right\|_{\mathcal{M}} = \int_{\mathbb{R}_c^{d}}^{} |f|(\mathrm{d}{v}).
    \end{align}
    For \( s > 0 \), let the \( s \)-moment norm of \(f \in \mathcal{M}_+ \) be
    \begin{align}\label{Defs}
    \|f\|_{s} = \int_{\mathbb{R}_c^{d}}^{} (1 + |v|^s) |f|(\mathrm{d}{v}).
    \end{align}
    For simplicity of notation, we write \( f \mathrm{d}{v} \) as if \( f \) is a function. The equation \eqref{equ:inelasticBoltzmann} should now be understood in sense of measure. More precisely, we will focus on the following notion of solutions to \eqref{equ:inelasticBoltzmann}.

    \begin{definition}[Weak solution]
        We say that \( f \in \mathit{C}([0, \infty), \mathcal{M}_+) \) is a weak solution of \eqref{equ:inelasticBoltzmann} {with initial condition} \( f(0,v) = f_0(v) \in \mathcal{M}_+ \) if for any \( T > 0 \) and for any test function {\( \psi \in \mathit{C}^1([0, T], \mathit{C}^{1}(\mathbb{R}_c^{d})) \)},
        \begin{equation} \label{equ:inelasticBoltzmann_weakForm}
            \begin{aligned}
                &\int_{\mathbb{R}^d}{\psi(T,v) f(T, \mathrm{d}{v})} - \int_{\mathbb{R}^d}{\psi(0,v) f_0(\mathrm{d}{v})} \\
                =&-\int^T_0\mathrm{d}{t}\int_{\mathbb{R}^d}{
                    {\partial_{t}}\psi\, f(t, \mathrm{d}{v})
                }
                - \int^T_0dt
                \int_{\mathbb{R}^d}{
                    (A v \cdot { \nabla_{v} }\psi) f(t, \mathrm{d}{v})
                }
                \\
                &+ \frac{1}{2} \int^T_0 \mathrm{d}{t}
                \int_{\mathbb{R}^d}
                \int_{\mathbb{R}^d}
                \int_{\mathbb{S}^{d-1}}\mathrm{d}{\sigma}\,
                b\big(\sigma \cdot \frac{v - v_*}{|v - v_*|}\big) f(t, \mathrm{d}{v}) f(t, \mathrm{d}{v_*}) (
                \psi(t,v') + \psi(t,v_*') - \psi(t,v) - \psi(t,v_*)
                ).
            \end{aligned}
        \end{equation}
    \end{definition}

    Under the cutoff assumption, we can write the corresponding strong form of the inelastic collision operator as
    \begin{equation} \label{DefQe+}
        \begin{aligned}
            Q_e(f, f) &=  Q_e^+(f, f) - Q_e^-(f, f), \\
            Q_e^+(f, f)(v) &= {\int_{\mathbb{R}^d}^{}} 
                {\int_{\mathbb{S}^{d-1}}^{}} 
                    b^+(\cos\theta) {e_\mathrm{res}}^{-1} f(v'') f(v_*'')
                 \, {\mathrm{d}{\sigma}}
            {\mathrm{d}{v_*}}, \\
            Q_e^-(f, f)(v) &= {\int_{\mathbb{R}^d}^{}} 
                {\int_{\mathbb{S}^{d-1}}^{}} 
                    b(\cos\theta) f(v) f(v_*)
                 \, {\mathrm{d}{\sigma}}
            {\mathrm{d}{v_*}}
            = {b_0 f(v) {\int_{\mathbb{R}^d}^{}} f(v_*) \, {\mathrm{d}{v_*}},}
        \end{aligned}
    \end{equation}
    where
    \begin{equation} \label{Defstrongv}
        \begin{aligned}
            v'' &= \frac{v + v_*}{2} + \frac{1 - z''}{2} (v - v_*) + \frac{z''}{2} |v - v_*| \sigma ,\\
            v_*'' &= \frac{v + v_*}{2} - \frac{1 - z''}{2} (v - v_*) - \frac{z''}{2} |v - v_*| \sigma,
        \end{aligned}
    \end{equation}
    are {image of the inverse of the scattering mapping}, \( z'' = \frac{1}{2} (1 + {e_\mathrm{res}}^{-1}) = z / {e_\mathrm{res}} \), and
    \begin{equation*}
        b^+(c) = b(\frac{(1 + {e_\mathrm{res}}^2) c - (1 - {e_\mathrm{res}}^2)}{(1 + {e_\mathrm{res}}^2) - (1 - {e_\mathrm{res}}^2) c}) \sqrt{\frac{2}{(1 + {e_\mathrm{res}}^2) - (1 - {e_\mathrm{res}}^2) c}},
    \end{equation*}
    for $-1\leq c\leq 1$. The derivation of the strong form \eqref{DefQe+} from the weak form \eqref{weakform} is strongly based on the inelastic reflection map method as in the appendix of \cite{CarlenEtAl09SCT}, where the $\omega$ representation of the post-collision velocities plays a crucial role. If one defines
    $u:=v-v_*$ and $\omega:=\frac{u-|u|\sigma}{|u-|u|\sigma|}$
    such that
    $$
    v''=v-\frac{1+{e_\mathrm{res}}}{2{e_\mathrm{res}}}((v-v_*)\cdot\omega)\omega,\quad v''_*=v_*+\frac{1+{e_\mathrm{res}}}{2{e_\mathrm{res}}}((v-v_*)\cdot\omega)\omega,
    $$
    and also denotes $\tilde{b}$ in the sense that $b(c)=\tilde{b}(\sqrt{(1-c)/2})(2\sqrt{(1-c)/2})^{-1}$, then it holds that
    \begin{align}
        b(\cos\theta)\mathrm{d}{\sigma}&=2\tilde{b}(\omega\cdot\frac{v - v_*}{|v - v_*|})\mathrm{d}{\omega},\label{sigmaomega1}\\
        \tilde{b}(\omega\cdot\frac{v'' - v''_*}{|v'' - v''_*|})\mathrm{d}{\omega}&=\frac{1}{2}b(\frac{(1 + {e_\mathrm{res}}^2) \cos\theta - (1 - {e_\mathrm{res}}^2)}{(1 + {e_\mathrm{res}}^2) - (1 - {e_\mathrm{res}}^2) \cos\theta}) \sqrt{\frac{2}{(1 + {e_\mathrm{res}}^2) - (1 - {e_\mathrm{res}}^2) \cos\theta}}\mathrm{d}{\sigma}. \label{sigmaomega2}
    \end{align}
    One may refer to \cite[A.3]{CarlenEtAl09SCT} for detailed derivation, see also \cite{GambaEtAl04BED, CarrilloEtAl21RDK}.
    With the help of the strong form, we give another notion of solutions.

    \begin{definition}[Mild solution]
	{Let ${\mathit{L}^{1}}([0, T], \mathcal{M}_+)$ denote the space of all real-valued functions $t\rightarrow\|f(t)\|_\mathcal{M}$ that
		are integrable on $[0,T]$.} For \( g, h_0 \in {\mathit{L}^{1}}([0, T], \mathcal{M}_+) \) and \( 0 \leq t_0 \leq t \leq T < \infty \), we define
	\begin{equation} \label{DefS}
		S_g(t, t_0) h_0(v)
		\coloneqq e^{(t - t_0) \mathrm{tr}(A)} e^{-b_0 {\int_{t_0}^{t}} \left\|g(\tau)\right\|_{\mathcal{M}} \, {\mathrm{d}{\tau}}} h_0(e^{(t - t_0) A} v),
	\end{equation}
	such that \( h(t, v) = S_g(t, t_0) h_0(v) \) solves
	\begin{equation*}
		{\partial_{t}} h - { \nabla_{v} } \cdot \!(A v h) = -b_0 h {\int_{\mathbb{R}^d}^{}} g(t,v) \, {\mathrm{d}{v}} = -b_0 h \left\|g(t)\right\|_{\mathcal{M}},
	\end{equation*}
	{in the sense of measure} on \( t \in [t_0, T] \) with initial condition \( h(t_0, v) = h_0(v)\).  Then, {we say that} \( f \in \mathit{C}([0, \infty), \mathcal{M}_+(\mathbb{R}_c^{d})) \) is a mild solution of \eqref{equ:inelasticBoltzmann} {with initial condition} \( f_0 \in \mathcal{M}_+(\mathbb{R}_c^{d}) \) if
	\begin{equation*}
		f(t, v) = S_f(t, 0) f_0(v) + {\int_{0}^{t}} 
		S_f(t, \tau) Q_e^+(f,f)(\tau,v)
		\, {\mathrm{d}{\tau}}.
	\end{equation*}
\end{definition}

{
	\begin{remark}	
		We point out $Q_e^\pm(f,f)$ defines measures in $\mathcal{M}_+$ by the continuity and integrability of $b(\cos \theta).$ Here the sense of measure is understood integrating by parts and passing the derivative to the test function. Note that if \( f \) is a mild solution, then \eqref{equ:inelasticBoltzmann} holds {in the sense of measure,} which implies that \eqref{equ:inelasticBoltzmann_weakForm} is satisfied on appropriate test functions, and thus \( f \) is also a weak solution (see also \cite{JamesEtAl17SSP}).
	\end{remark}
}

    We first prove some useful properties of \( Q_e^+ \) and \( S \).
    
    \begin{lemma} \label{lem:propertyOfGainPartAndLinearPartSolver}
        If \( f, g \in \mathcal{M}_+(\mathbb{R}_c^{d}) \), then
        \begin{align}
            \left\|Q_e^+(f, f)\right\|_{\mathcal{M}}
            &= b_0 \left\|f\right\|_{\mathcal{M}}^2, \label{boundQ}\\
            \left\|Q_e^+(f, f) - Q_e^+(g, g)\right\|_{\mathcal{M}}
            &\leq b_0 (\left\|f\right\|_{\mathcal{M}}  + \left\|g\right\|_{\mathcal{M}}) \left\|f - g\right\|_{\mathcal{M}}. \label{boundQ-Q}
        \end{align}
        Also, if \( h_1, h_2 \in {\mathit{L}^{1}}([0, T], \mathcal{M}_+(\mathbb{R}_c^{d})) \), then for \( T \geq t \geq t' \geq 0 \),
        it holds that
        \begin{align}
            \left\|S_{h_1}(t, t')f\right\|_{\mathcal{M}}
            &= \exp(
            -b_0 {\int_{t'}^{t}} \left\|h_1(\tau)\right\|_{\mathcal{M}} \! \, {\mathrm{d}{\tau}}
            ) {\int_{\mathbb{R}^d}^{}} f(v) \, {\mathrm{d}{v}}
            \leq \left\|f\right\|_{\mathcal{M}},  \label{boundSf}\\
            \left\|S_{h_1}(t, t')f - S_{h_2}(t, t') g\right\|_{\mathcal{M}}
            &\leq \left\|f - g\right\|_{\mathcal{M}} + \frac{b_0}{2} (\left\|f\right\|_{\mathcal{M}}  + \left\|g\right\|_{\mathcal{M}}) {\int_{t'}^{t}} \left\|h_1(\tau) - h_2(\tau)\right\|_{\mathcal{M}} \, {\mathrm{d}{\tau}}.\label{boundSfg}
        \end{align}
        Furthermore, if \( \|f\|_s, \|g\|_s < \infty \) for some \( s > 0 \), then one has
        \begin{align}
            \|{Q_e^+(f,f)}\|_{s}
            &\leq C      \left\|f\right\|_{\mathcal{M}}  \|f\|_{s}, \label{boundQs}\\
            \|Q_e^+(f,f) - Q_e^+(g,g)\|_{s}
            &\leq C (\|f\|_s + \|g\|_s) \|f - g\|_{s}, \label{boundQ-Qs}\\
            \|S_{h_1}(t, t')f\|_{s}
            &\leq C_T\|f\|_s, \label{boundSfs}
        \end{align}
        and
        \begin{equation} \label{boundSfgs}
            \|S_{h_1}(t, t')f - S_{h_2}(t, t') g\|_{s}
            \leq C \|f - g\|_{s} + C(\|f\|_s + \|g\|_s) {\int_{t'}^{t}} \left\|h_1(\tau) - h_2(\tau)\right\|_{\mathcal{M}} \, {\mathrm{d}{\tau}},
        \end{equation}
        where the constant $C$ depends only on \( A, s, T, b_0 \).
    \end{lemma}
    
    \begin{proof}
        Noticing that \( Q_e^+(f,f) \geq 0 \) and \( S_g(t, t')f \geq 0 \) by the definitions of $Q_e^+$ and $S$ in \eqref{DefQe+} and \eqref{DefS}, by the fact that $1$ is a collision invariant, we have \eqref{boundQ} by
        \begin{align*}
            \left\|Q_e^+(f, f)\right\|_{\mathcal{M}}
            &= {\int_{\mathbb{R}^d}^{}} Q_e(f, f)
             \, {\mathrm{d}{v}} + {\int_{\mathbb{R}^d}^{}}    Q_e^-(f,f)
             \, {\mathrm{d}{v}} \\
            &= {\int_{\mathbb{R}^d}^{}} 
                {\int_{\mathbb{R}^d}^{}} 
                    {\int_{\mathbb{S}^{d-1}}^{}} 
                        b(\cos\theta) f(v) f(v_*)
                     \, {\mathrm{d}{\sigma}}
                {\mathrm{d}{v_*}}
            {\mathrm{d}{v}}\\
            &= b_0 \left\|f\right\|_{\mathcal{M}}^2.
        \end{align*}
        Next, it is direct to get
        \begin{align*}
            \left\|Q_e^+(f, f) - Q_e^+(g, g)\right\|_{\mathcal{M}}
            &\leq {\int_{\mathbb{R}^d}^{}} 
                {\int_{\mathbb{R}^d}^{}} 
                    {\int_{\mathbb{S}^{d-1}}^{}} 
                        b^+(\cos\theta) {e_\mathrm{res}}^{-1}
                        |f(v'') f(v_*'') - g(v'') g(v_*'')|
                     \, {\mathrm{d}{\sigma}}
                {\mathrm{d}{v_*}}
            {\mathrm{d}{v}}.
        \end{align*}
        From the definition of $v'',v_*''$ in \eqref{Defstrongv}, a straight-forward calculation gives that the Jacobian matrix for the transformation $(v,v_*)\mapsto(v'',v_*'')$ satisfies
        \begin{align*}
            |J|={e_\mathrm{res}}^{-1}.
        \end{align*}
        Then we use changes of variables, \eqref{sigmaomega1} and \eqref{sigmaomega2} to get
        \begin{align*}
            \left\|Q_e^+(f, f) - Q_e^+(g, g)\right\|_{\mathcal{M}}
            &\leq
            {\int_{\mathbb{R}^d}^{}} 
                {\int_{\mathbb{R}^d}^{}} 
                    {\int_{\mathbb{S}^{d-1}}^{}} 
                        b(\cos\theta)
                        |f(v) f(v_*) - g(v) g(v_*)|
                     \, {\mathrm{d}{\sigma}}
                {\mathrm{d}{v_*}}
            {\mathrm{d}{v}} \\
            &\leq b_0 {\int_{\mathbb{R}^d}^{}} 
                {\int_{\mathbb{R}^d}^{}} 
                    (|f(v)| |f(v_*) - g(v_*)| + |g(v_*)| |f(v) - g(v)|)
                 \, {\mathrm{d}{v_*}}
            {\mathrm{d}{v}} \\
            &= b_0 (\left\|f\right\|_{\mathcal{M}} + \left\|g\right\|_{\mathcal{M}}) \left\|f - g\right\|_{\mathcal{M}},
        \end{align*}
        which implies \eqref{boundQ-Q}.
        Moreover, \eqref{boundSf} follows from
        \begin{align*}
            \left\|S_{h_1}(t, t')f\right\|_{\mathcal{M}}
            = {\int_{\mathbb{R}^d}^{}} S_{h_1}(t, t')f(v) \, {\mathrm{d}{v}}
            &= \exp\big(
            (t - t') \mathrm{tr}(A)
            - b_0 {\int_{t'}^{t}} \left\|h_1(\tau)\right\|_{\mathcal{M}} \! \, {\mathrm{d}{\tau}}
            \big) {\int_{\mathbb{R}^d}^{}} f(e^{(t - t') A} v) \, {\mathrm{d}{v}} \\
            &= \exp\big(
            -b_0 {\int_{t'}^{t}} \left\|h_1(\tau)\right\|_{\mathcal{M}} \! \, {\mathrm{d}{\tau}}
            \big) {\int_{\mathbb{R}^d}^{}} f(v) \, {\mathrm{d}{v}} \\
            &\leq \left\|f\right\|_{\mathcal{M}}.
        \end{align*}
        Using changes of variables and the fact that \( |e^{-x_1} - e^{-x_2}| \leq |x_1 - x_2| \) on \( x_1, x_2 \geq 0 \), we get \eqref{boundSfg} by
        \begin{align*}
            &\quad \left\|S_{h_1}(t, t')f - S_{h_2}(t, t') g\right\|_{\mathcal{M}} \\
            &= e^{(t - t') \mathrm{tr}(A)} {\int_{\mathbb{R}^d}^{}} \Big|
                e^{-b_0 {\int_{t'}^{t}} \left\|h_1(\tau)\right\|_{\mathcal{M}} \! \, {\mathrm{d}{\tau}}} f(e^{(t - t') A} v)
                - e^{-b_0 {\int_{t'}^{t}} \left\|h_2(\tau)\right\|_{\mathcal{M}}\! \, {\mathrm{d}{\tau}}} g(e^{(t - t') A} v)
                \Big| \, {\mathrm{d}{v}} \\
            &\leq {\int_{\mathbb{R}^d}^{}} (
                \frac{f(v) + g(v)}{2}
                b_0 {\int_{t'}^{t}} 
                    \big|\left\|h_1(\tau)\right\|_{\mathcal{M}} - \left\|h_2(\tau)\right\|_{\mathcal{M}}\big|
                 \, {\mathrm{d}{\tau}}
                + |f(v) - g(v)|
                ) \, {\mathrm{d}{v}} \\
            &\leq \frac{b_0}{2} (\left\|f\right\|_{\mathcal{M}} + \left\|g\right\|_{\mathcal{M}}) {\int_{t'}^{t}} \left\|h_1(\tau) - h_2(\tau)\right\|_{\mathcal{M}} \, {\mathrm{d}{\tau}} + \left\|f - g\right\|_{\mathcal{M}}.
        \end{align*}

        Noticing \( |v'|^{s} \leq (|v| + |v_*|)^{s} \leq C( |v|^{s} + |v_*|^{s}) \) and \( |v|^{s} \leq C( |v''|^{s} + |v_*''|^{s}) \) for $s\geq 0$ and some constant $C$ which depends only on $s$, by changes of variables, \eqref{sigmaomega1} and \eqref{sigmaomega2}, then \eqref{boundQs} holds from
        \begin{align*}
            \|Q_e^+(f,f)\|_{s}
            &\leq C{\int_{\mathbb{R}^d}^{}} 
                {\int_{\mathbb{R}^d}^{}} 
                    {\int_{\mathbb{S}^{d-1}}^{}} 
                        (2 + |v''|^s + |v_*''|^s) b^+(\cos\theta) {e_\mathrm{res}}^{-1} f(v'') f(v_*'')
                     \, {\mathrm{d}{\sigma}}
                {\mathrm{d}{v_*}}
            {\mathrm{d}{v}} \\
            &= C{\int_{\mathbb{R}^d}^{}} 
                {\int_{\mathbb{R}^d}^{}} 
                    {\int_{\mathbb{S}^{d-1}}^{}} 
                        (2 + |v|^s + |v_*|^s) b(\cos\theta) f(v) f(v_*)
                     \, {\mathrm{d}{\sigma}}
                {\mathrm{d}{v_*}}
            {\mathrm{d}{v}} \\
            &\leq C \left\|f\right\|_{\mathcal{M}} \|f\|_s.
        \end{align*}
        Similarly, we yield \eqref{boundQ-Qs} by
        \begin{align*}
            &\quad\|Q_e^+(f,f) - Q_e^+(g,g)\|_{s} \\
            &\leq C {\int_{\mathbb{R}^d}^{}} 
                {\int_{\mathbb{R}^d}^{}} 
                    {\int_{\mathbb{S}^{d-1}}^{}} 
                        (2 + |v''|^{s} + |v_*''|^{s}) b^+(\cos\theta) {e_\mathrm{res}}^{-1} |f(v'') f(v_*'') - g(v'') g(v_*'')|
                     \, {\mathrm{d}{\sigma}}
                {\mathrm{d}{v_*}}
            {\mathrm{d}{v}} \\
            &=C {\int_{\mathbb{R}^d}^{}} 
                {\int_{\mathbb{R}^d}^{}} 
                    {\int_{\mathbb{S}^{d-1}}^{}} 
                        (2 + |v|^{s} + |v_*|^{s}) b(\cos\theta) |f(v) f(v_*) - g(v) g(v_*)|
                     \, {\mathrm{d}{\sigma}}
                {\mathrm{d}{v_*}}
            {\mathrm{d}{v}} \\
            &\leq C {\int_{\mathbb{R}^d}^{}} 
                {\int_{\mathbb{R}^d}^{}} 
                    (2 + |v|^{s} + |v_*|^{s})
                    (
                    f(v) |f(v_*) - g(v_*)|
                    + g(v_*) |f(v) - g(v)|
                    )
                 \, {\mathrm{d}{v_*}}
            {\mathrm{d}{v}} \\
            &\leq C (\|f\|_s + \|g\|_s) \|f - g\|_{s}.
        \end{align*}
        One can further prove \eqref{boundSfs} and \eqref{boundSfgs} by
        \begin{align*}
            \|S_{h_1}(t, t')f\|_{s}
            &= {\int_{\mathbb{R}^d}^{}} 
                (1 + |v|^{s})
                e^{(t - t') \mathrm{tr}(A)} e^{-b_0 {\int_{t'}^{t}} \left\|h_1(\tau)\right\|_{\mathcal{M}} \, {\mathrm{d}{\tau}}}
                f(e^{(t - t') A} v)
             \, {\mathrm{d}{v}} \\
            &\leq {\int_{\mathbb{R}^d}^{}} 
                (1 + \big|e^{-(t - t') A}v\big|^{s}) f(v)
             \, {\mathrm{d}{v}}\\
            &\leq C_T\|f\|_s,
        \end{align*}
        and
        \begin{align*}
            &\quad\|S_{h_1}(t, t')f - S_{h_2}(t, t') g\|_{s} \\
            &\leq C{\int_{\mathbb{R}^d}^{}} 
                (1 + |v|^{s})
                (
                \frac{f(v) + g(v)}{2}
                b_0 {\int_{t'}^{t}} 
                    |\left\|h_1(\tau)\right\|_{\mathcal{M}} - \left\|h_2(\tau)\right\|_{\mathcal{M}}|
                 \, {\mathrm{d}{\tau}}
                + |f(v) - g(v)|
                )
             \, {\mathrm{d}{v}} \\
            &\leq C(\|f\|_s + \|g\|_s) {\int_{t'}^{t}} \left\|h_1(\tau) - h_2(\tau)\right\|_{\mathcal{M}} \, {\mathrm{d}{\tau}}
            + C\|f - g\|_{s}.
        \end{align*}
        This then completes the proof of Lemma \ref{lem:propertyOfGainPartAndLinearPartSolver}.
    \end{proof}

    With all the above estimates, we now prove the existence of a mild solution by fixed point approach.

    \begin{theorem} \label{thm:existenceOfMildSol}
        If \( f_0 \in \mathcal{M}_+(\mathbb{R}_c^{d}) \), {then} for any constant shearing matrix $A$ and \( {e_\mathrm{res}} \in (0, 1] \), there exists a unique global-in-time mild solution $f$ to \eqref{equ:inelasticBoltzmann} satisfying
        \begin{equation}\label{masscon}
            \left\|f(t)\right\|_{\mathcal{M}} = \left\|f_0\right\|_{\mathcal{M}}.
        \end{equation}
    \end{theorem}

    \begin{proof}
        Let
        $$
        X_T \coloneqq \big\{
        f \in \mathit{C}([0, T], \mathcal{M}_+(\mathbb{R}_c^{d}))\,
        \big| \,
        \|f\|_{X_T} \coloneqq \sup\limits_{[0, T]} \left\|f(t)\right\|_{\mathcal{M}} \leq 2 \left\|f_0\right\|_{\mathcal{M}}
        \big\}
        $$
        with $T>0$. Consider the Picard iteration mapping \( P : X_T \to X_T \) by
        \begin{equation*}
            Pf(t, v)
            = S_f(t, 0) f_0(v) + {\int_{0}^{t}} S_f(t, \tau) Q_e^+(f,f)(\tau, v) \, {\mathrm{d}{\tau}}.
        \end{equation*}
        We first show the boundedness of $Pf$ for small time. By \eqref{boundSf} and \eqref{boundQ}, for \( f \in X_T \) and \( t \in [0, T] \), one has
        \begin{align*}
            \left\|Pf(t)\right\|_{\mathcal{M}}
            &\leq \left\|S_f(t, 0) f_0\right\|_{\mathcal{M}} + {\int_{0}^{t}} 
                \left\|S_f(t, \tau) Q_e^+(f, f)(\tau)\right\|_{\mathcal{M}}
             \, {\mathrm{d}{\tau}} \\
            &\leq \left\|f_0\right\|_{\mathcal{M}} + {\int_{0}^{t}} 
                \left\|Q_e^+(f, f)(\tau)\right\|_{\mathcal{M}}
             \, {\mathrm{d}{\tau}} \\
            &= \left\|f_0\right\|_{\mathcal{M}} + b_0 {\int_{0}^{t}} 
                \left\|f(\tau)\right\|_{\mathcal{M}}^2
             \, {\mathrm{d}{\tau}},
        \end{align*}
        which further by the definition of $X_T$ with norm $\|\cdot\|_{X_T}$ gives
        $$
        \left\|Pf(t)\right\|_{\mathcal{M}}\leq \left\|f_0\right\|_{\mathcal{M}} + b_0 T \|f\|_{X_T}^2\leq \left\|f_0\right\|_{\mathcal{M}} + 4 b_0 T \left\|f_0\right\|_{\mathcal{M}}^2.
        $$
        Thus, for
        \begin{align}\label{Tbound}
            T \leq \frac{4 b_0}{\left\|f_0\right\|_{\mathcal{M}} + 1} ,
        \end{align}
        it holds that
        \begin{align}\label{boundPf} 
        \|Pf\|_{X_T} \leq 2 \left\|f_0\right\|_{\mathcal{M}}.
        \end{align}

        Furthermore, since \( b \geq 0 \) and \( f \geq 0 \), \( Pf(t) \geq 0 \) for all \( t \in [0, T] \), which implies the positivity of $Pf$.

        Then we prove that $P$ is a contraction mapping on $X_T$ for small $T$. Let \( f, g \in X_T \), for \( 0 \leq t \leq T \), we take the difference to get
        \begin{align*}
            \left\|Pf(t) - Pg(t)\right\|_{\mathcal{M}}
            \leq& \left\|S_f(t, 0)f_0 - S_g(t, 0)f_0\right\|_{\mathcal{M}}
            + {\int_{0}^{t}} 
                \left\|S_f(t, \tau) Q_e^+(f,f)(\tau) - S_g(t, \tau) Q_e^+(g,g)(\tau)\right\|_{\mathcal{M}}
             \, {\mathrm{d}{\tau}}.
        \end{align*}
        It follows from \eqref{boundSfg}, \eqref{boundQ} and \eqref{boundQ-Q} that
        \begin{align*}
            &\left\|Pf(t) - Pg(t)\right\|_{\mathcal{M}} \\
            \leq &b_0 \left\|f_0\right\|_{\mathcal{M}} {\int_{0}^{t}} \left\|f(\tau) - g(\tau)\right\|_{\mathcal{M}} \, {\mathrm{d}{\tau}} \\
            &+{\int_{0}^{t}} \big(
                \left\|Q_e^+(f, f)(\tau) - Q_e^+(g, g)(\tau)\right\|_{\mathcal{M}}
                + \frac{b_0}{2} (\left\|Q_e^+(f, f)(\tau)\right\|_{\mathcal{M}} + \left\|Q_e^+(g,g)(\tau)\right\|_{\mathcal{M}}) {\int_{\tau}^{t}} \left\|f(\tau') - g(\tau')\right\|_{\mathcal{M}} \! \, {\mathrm{d}{\tau'}}
                \big)\! \, {\mathrm{d}{\tau}} \\
            \leq& b_0 T \left\|f_0\right\|_{\mathcal{M}} \|f - g\|_{X_T} \\
            &+ \|f - g\|_{X_T} {\int_{0}^{t}} \big(
                b_0
                (\left\|f(\tau)\right\|_{\mathcal{M}} + \left\|g(\tau)\right\|_{\mathcal{M}})
                + \frac{b_0^3}{2}
                (\left\|f(\tau)\right\|_{\mathcal{M}}^2 + \left\|g(\tau)\right\|_{\mathcal{M}}^2)
                T
                \big) \, {\mathrm{d}{\tau}}.
        \end{align*}
        By letting $T$ satisfy \eqref{Tbound}, we have from \eqref{boundPf} that
        \begin{align*}
            \left\|Pf(t) - Pg(t)\right\|_{\mathcal{M}}
            \leq& \|f - g\|_{X_T}
            b_0 T \left\|f_0\right\|_{\mathcal{M}}
            (
            5 + 4 b_0^2 T \left\|f_0\right\|_{\mathcal{M}}
            ).
        \end{align*}
        We further let
        \begin{equation*}
            {T = \min\Big\{\frac{1}{1 + 10 b_0\, \left\|f_0\right\|_{\mathcal{M}} + 8 b_0^3 \left\|f_0\right\|_{\mathcal{M}}^2},\frac{4 b_0}{\left\|f_0\right\|_{\mathcal{M}} + 1}\Big\},}
        \end{equation*}
        to get $\|Pf - Pg\|_{X_T} \leq \frac{1}{2}\|f - g\|_{X_T},$
        which implies \( P \) is a contraction on \( X_T \). By Banach contraction theorem, there exists a unique \( f \in X_T \) such that \( f = Pf \).

        {In order to get a global solution, we integrate in $v$ and $t$ on both sides of \eqref{equ:inelasticBoltzmann} to get
        \begin{equation}\label{massconserve}
            \left\|f(t)\right\|_{\mathcal{M}}
            =  \left\|f_0\right\|_{\mathcal{M}}.
        \end{equation}}
        Hence, the mass is conserved with \eqref{masscon}, so we can extend the local-in-time solution to be a unique global-in-time mild solution, which finishes the proof of Theorem \ref{thm:existenceOfMildSol}.
    \end{proof}

    Concerning the moments of a mild solution, we have the following result.

    \begin{theorem} \label{thm:mildSolAsWeakSolAndProperties}
        Suppose \( f \) is a mild solution of \eqref{equ:inelasticBoltzmann} {with initial condition} \( f_0(v) \in \mathcal{M}_+ \) with \( {\int_{\mathbb{R}^d}^{}} |v|^{s_0} f_0 \, {\mathrm{d}{v}} < \infty \) for some \( s_0 > 1 \), then  for any \( 0 < T < \infty \), it holds that
        \begin{align}\label{boundf}
            \sup_{0\leq t\leq T}\|f(t)\|_{s_0} \leq C_T\|f_0\|_{s_0}<\infty, \end{align} with some constant $C_T$ which depends on \( T \), and the weak form \eqref{equ:inelasticBoltzmann_weakForm} is satisfied on test function
        \( \psi(t,v) \in \mathit{C}^1([0, T], \mathit{C}^{1}(\mathbb{R}^d)) \)
        with
        \( |\psi(t,v)| + |v| |{ \nabla_{v} }\psi(t,v)| \leq C( 1 + |v|^{s'}) \)
        uniformly on \( t \in [0, T] \) for some \( 1 < s' < s_0 \).
        Furthermore, on each \( s > s_0 \) and each \( 0 < T < \infty \) there exists \( f_n \) with \( \sup_{[0, T]} \int |v|^s f_n(t) < \infty \) and \( f \) can be approximated by \( f_n(t,v) \in \mathit{C}^{0}([0,T], \mathcal{M}_+) \) in the sense that \( \sup_{[0, T]} \|f_n(t) - f(t)\|_{s_0} \to 0 \) as $n\to \infty$.
    \end{theorem}
    
    \begin{proof}
        By \eqref{boundSfs}, \eqref{boundQs} and \eqref{massconserve}, we have
        \begin{align*}
            \|f(t)\|_{s_0}
            &\leq \|S_f(t, 0)f_0\|_{s_0}
            + {\int_{0}^{t}} 
                \|S_f(t, \tau) Q_e^+(f,f)(\tau)\|_{s_0}
             \, {\mathrm{d}{\tau}} \\
            &\leq C_T( \|f_0\|_{s_0} + {\int_{0}^{t}} 
                \left\|f(\tau)\right\|_{\mathcal{M}} \|f(\tau)\|_{s_0}
             \, {\mathrm{d}{\tau}}) \\
            &= C_T(\|f_0\|_{s_0} + \left\|f_0\right\|_{\mathcal{M}} {\int_{0}^{t}} 
                \|f(\tau)\|_{s_0}
             \, {\mathrm{d}{\tau}}),
        \end{align*}
        which implies \eqref{boundf}.

        Let \( \psi(t,v) \in \mathit{C}^1([0, T], \mathit{C}^{1}(\mathbb{R}^d)) \) be {a test function} with $|\psi(t,v)| + |v| |{ \nabla_{v} }\psi(t,v)| \leq C( 1 + |v|^{s'})$. Then let \( \psi(t,v) \) be approximated by the sequence \( \{\psi_n(t,v)\}^{\infty}_{n=1} \subset \mathit{C}^1([0, T], \mathit{C}^{1}(\mathbb{R}_c^{d})) \) with \( |{ \nabla_{v} }\psi_n| \leq C( 1 +|v|^{s'-1}+ |{ \nabla_{v} } \psi|) \), then one gets
        \begin{equation} \label{weakapprox}
            \begin{aligned}[b]
                &\quad\;\,
                {\int_{\mathbb{R}^d}^{}} \psi_n(T,v) f(T,v) \, {\mathrm{d}{v}} - {\int_{\mathbb{R}^d}^{}} \psi_n(0,v) f_0(v) \, {\mathrm{d}{v}} \\
                &\qquad\qquad\qquad+ {\int_{0}^{T}} \int_{\mathbb{R}^d} f(t,v) {\partial_{t}}\psi_n(t,v) \mathrm{d}{v} \, {\mathrm{d}{t}}
                + {\int_{0}^{T}} \int_{\mathbb{R}^d} A v f(t,v) \cdot { \nabla_{v} }\psi_n(t,v) \mathrm{d}{v} \, {\mathrm{d}{t}} \\
                &= \frac{1}{2} {\int_{0}^{T}} {\int_{\mathbb{R}^d}^{}} {\int_{\mathbb{R}^d}^{}} {\int_{\mathbb{S}^{d-1}}^{}} 
                                b(\cos\theta) f(t,v) f(t,v_*) (\psi_n(t,v') + \psi_{n}(t,v'_*) - \psi_n(t,v) - \psi_{n}(t,v_*))
                 \, {\mathrm{d}{\sigma}}{\mathrm{d}{v_*}}{\mathrm{d}{v}}{\mathrm{d}{t}}.
            \end{aligned}
        \end{equation}
        Then an application of Taylor expansion gives
        \begin{equation}\label{Taylor}
            \begin{aligned}[b]
                &\quad
                \Big|{\int_{\mathbb{S}^{d-1}}^{}} 
                    b(\cos\theta) (\psi_n(t,v') + \psi_n(t,v'_*) - \psi_n(t,v) - \psi_n(t,v_*))
                 \, {\mathrm{d}{\sigma}}\Big| \\
                &\leq C\big|(v' - v) \cdot {\int_{0}^{1}} 
                    ({ \nabla_{v} } \psi_n(t,v + \xi(v' - v)) - { \nabla_{v} } \psi_n(t,v_* + \xi(v_*' - v_*)))
                 \, {\mathrm{d}{\xi}}\big|.
            \end{aligned}
        \end{equation}
        A direct calculation shows that {\begin{align*} |v' - v| &= |v_*' - v_*| = \big|\frac{z}{2}(v-v_*)-\frac{z}{2}|v-v_*|\sigma\big|=\big|\frac{z}{2}|v-v_*|(\frac{v-v_*}{|v-v_*|}-\sigma)\big|\leq z|v-v_*|,
        \end{align*}}
        which yields
        \begin{equation} \label{boundgra}
            \begin{aligned}[b]
                &\quad
                C\big|(v' - v) \cdot {\int_{0}^{1}} 
                    ({ \nabla_{v} } \psi_n(t,v + \lambda(v' - v)) - { \nabla_{v} } \psi_n(t,v_* + \lambda(v_*' - v_*)))
                 \, {\mathrm{d}{\lambda}}\big| \\
                &\leq C\sup_{|u - v| \leq z |v - v_*|} |v - v_*| |{ \nabla_{v} } \psi_n(t,u)|
                + \sup_{|u - v_*| \leq z |v - v_*|} |v - v_*| |{ \nabla_{v} } \psi_n(t,u)| \\
                &\leq C(|v| + |v_*|) \sup_{|u| \leq (1 + z) (|v| + |v_*|)} |{ \nabla_{v} } \psi_n(t,u)| \\
                &\leq C(|v| + |v_*|) \big(1 + (|v| + |v_*|)^{s'-1}+\sup_{R \leq |u| \leq (1 + z) (|v| + |v_*|)} |{ \nabla_{v} } \psi(t,u)|\big) \\
                &\leq C( 1 + (|v| + |v_*|)^{s'}),
            \end{aligned}
        \end{equation}
        where the constant $C$ is independent of \( n \). The third inequality above holds since \( \psi(t,v) \in \mathit{C}([0, T], \mathit{C}^{1}(\mathbb{R}^d))  \), which gives \( |{ \nabla_{v} }{\psi(t,v)}| \leq M \) on a neighborhood \( B(0, R) \) of the origin for some constant \( M = M(\psi) > 0 \).
        Collecting \eqref{weakapprox}, \eqref{Taylor} and \eqref{boundgra}, then {passing to limit \( n \to \infty \) on \( \psi_n \) in \eqref{weakapprox},} by dominated convergence theorem, one obtains \eqref{equ:inelasticBoltzmann_weakForm} {for \( \psi(t,v) \in \mathit{C}^1([0, T], \mathit{C}^{1}(\mathbb{R}^d)) \)  with $|\psi(t,v)| + |v| |{ \nabla_{v} }\psi(t,v)| \leq C( 1 + |v|^{s'})$. }

        Consider the higher moment approximation of $f$ now.  {Let $\chi$ denote the indicator functin such that $\chi_E(x)=1$ for $x\in E$.}  Defining the initial data \( f_{n0} = f_0 {\chi_{|v| \leq n}} \) and the corresponding mild solutions \( f_n \), we have
        \begin{align*}
            &\|f_{n0} - f_0\|_{s_0} = \int_{\mathbb{R}^d} (1 + |v|^{s_0}) f_0(v) {\chi_{|v| > n}} \mathrm{d}{v}
            \to 0,
        \end{align*}
        as $n\to\infty$.
        From \eqref{boundf}, it holds that
        \begin{align*}
            &\sup_{t\in[0, T]} \|f_n(t)\|_{s} \leq C_T \|f_{n0}\|_{s} < \infty, \\
            &\sup_{t\in[0, T]} \|f_n(t)\|_{s_0} \leq C_T \|f_{n0}\|_{s_0} \leq C_T\|f_0\|_{s_0} < \infty,
        \end{align*}
        which, together with \eqref{boundSfs}, \eqref{boundSfgs} and \eqref{boundQ-Qs}, yield
        \begin{align*}
            \|f_n(t) - f(t)\|_{s_0}
            &= \|S_{f_n}(t, 0)f_{n0} - S_f(t, 0)f_0\|_{s_0}
            + {\int_{0}^{t}} 
                \|S_{f_n}(t, \tau)Q_e^+(f_n,f_n)(\tau) - S_f(t, \tau)Q_e^+(f,f)(\tau)\|_{s_0}
             \, {\mathrm{d}{\tau}}\\
            &\leq C\Big(\|f_{n0} - f_0\|_{s_0} + \|f_0\|_{s_0} {\int_{0}^{t}} \left\|f_n(\tau) - f(\tau)\right\|_{\mathcal{M}} \, {\mathrm{d}{\tau}} \\
            &\quad\quad\;\,
            + {\int_{0}^{t}} 
                (\|f_n(\tau)\|_{s_0} + \|f(\tau)\|_{s_0}) \|f_n(\tau) - f(\tau)\|_{s_0}
             \, {\mathrm{d}{\tau}} \\
            &\quad\quad\;\,
            + {\int_{0}^{t}} 
                (
                \|f_n(\tau)\|_{s_0}^2
                + \|f(\tau)\|_{s_0}^2
                )
                {\int_{\tau}^{t}} \left\|f_n(\tau') - f(\tau')\right\|_{\mathcal{M}} \, {\mathrm{d}{\tau'}}
             \, {\mathrm{d}{\tau}}\Big) \\
            &\leq C\|f_{n0} - f_0\|_{s_0}
            + C_T\|f_0\|_{s_0} (\|f_0\|_{s_0} + T) {\int_{0}^{t}} \|f_n(\tau) - f(\tau)\|_{s_0} \, {\mathrm{d}{\tau}}.
        \end{align*}
        It then follows from Gronwall inequality that \( \sup_{t\in[0, T]} \|f_n(t) - f(t)\|_{s_0} \leq C_T\|f_{n0} - f_0\|_{s_0} \to 0 \) as $n\to \infty$.
    \end{proof}
{Theorem \ref{thm:existenceOfMildSol} and Theorem \ref{thm:mildSolAsWeakSolAndProperties} will be very useful later in the proof of Theorem \ref{thm:existenceOfStatProfileWithSmallness} for the existence of stationary self-similar profile $G$.}

    \subsection{Second moment equation} \label{sec:secondMomentEqu}

    We now compute the evolution equation for the second moment tensor \(\mathcal{B} =\mathcal{B}(t):= {\int_{\mathbb{R}^d}^{}} {{v}^{\otimes 2}} f \, {\mathrm{d}{v}} \) on a {mild} solution \( f \) of \eqref{equ:inelasticBoltzmannSelfSim} that has unit mass and zero momentum \( {\int_{\mathbb{R}^d}^{}} (1, v) f(t,v) \, {\mathrm{d}{v}} = (1, 0) \). Note that by definition, \( \mathcal{B} \) is symmetric positive semidefinite.

    By multiplying \( {{v}^{\otimes 2}} \) on both sides of \eqref{equ:inelasticBoltzmannSelfSim} and integrating over \( v \), we have
    \begin{align*}
        \frac{\mathrm{d}}{{\mathrm{d} t}} \mathcal{B}
        &=
        \frac{\mathrm{d}}{{\mathrm{d} t}} {\int_{\mathbb{R}^d}^{}} {{v}^{\otimes 2}} f \, {\mathrm{d}{v}} \\
        &= {\int_{\mathbb{R}^d}^{}} 
            {{v }^{\otimes 2}}{ \nabla_{v} } \cdot \!(A_\beta v f)
         \, {\mathrm{d}{v}}
        + {\int_{\mathbb{R}^d}^{}} {{v}^{\otimes 2}} Q_e(f, f) \, {\mathrm{d}{v}} \\
        &= -{\int_{\mathbb{R}^d}^{}} 
            A_\beta v f \cdot { \nabla_{v} }({{v}^{\otimes 2}})
         \, {\mathrm{d}{v}}
        + \frac{1}{2} {\int_{\mathbb{R}^d}^{}} 
            {\int_{\mathbb{R}^d}^{}} 
                {\int_{\mathbb{S}^{d-1}}^{}} 
                    b(\sigma \cdot \hat{e})
                    f(t,v) f(t,v_*)
                    \big({{v'}^{\otimes 2}} + {{v_*'}^{\otimes 2}} - {{v}^{\otimes 2}} - {{v_*}^{\otimes 2}}\big)
                 \, {\mathrm{d}{\sigma}}
            {\mathrm{d}{v_*}}
        {\mathrm{d}{v}},
    \end{align*}
    where we denote \( \hat{e} = \frac{v - v_*}{|v - v_*|} \). {Recalling \eqref{b0}, \eqref{b1} and \eqref{bsigma2}, a direct calculation gives}
    \begin{align*}
        \frac{\mathrm{d}}{{\mathrm{d} t}} \mathcal{B}&= -{\int_{\mathbb{R}^d}^{}} 
            (v \otimes (A_\beta v) + (A_\beta v) \otimes v) f(t,v)
         \, {\mathrm{d}{v}} \\
        &\quad
        - \frac{1}{4} {\int_{\mathbb{R}^d}^{}} {\int_{\mathbb{R}^d}^{}} {\int_{\mathbb{S}^{d-1}}^{}} 
                    b(\sigma \cdot \hat{e}) f(t,v) f(t,v_*) |v - v_*|^2 \\
                    &\qquad\qquad\qquad\times
                    (
                    z (2 - z) {\hat{e}}^{\otimes 2} - z^2 {{\sigma}^{\otimes 2}} - z (1 - z) (\hat{e} \otimes \sigma + \sigma \otimes \hat{e})
                    )
         \, {\mathrm{d}{\sigma}}{\mathrm{d}{v_*}}{\mathrm{d}{v}} \\
        &= -\Big(({\int_{\mathbb{R}^d}^{}} {{v}^{\otimes 2}}f(t,v) \, {\mathrm{d}{v}}) A_\beta^\mathsf{T} + A_\beta {\int_{\mathbb{R}^d}^{}} {{v}^{\otimes 2}}f(t,v) \, {\mathrm{d}{v}}\Big) \\
        &\quad
        -\frac{1}{4} {\int_{\mathbb{R}^d}^{}} {\int_{\mathbb{R}^d}^{}} 
                f(t,v) f(t,v_*)|v - v_*|^2
                (
                (2 \zeta + z^2 d c_{11}) {\hat{e}}^{\otimes 2} - z^2 c_{11} I
                )
         \, {\mathrm{d}{v_*}}{\mathrm{d}{v}},
    \end{align*}
    {where $\zeta$ and $c_{11}$ is defined in  \eqref{Defzeta} and \eqref{Defc11} respectively.} Then
    \begin{align*}
        \frac{\mathrm{d}}{{\mathrm{d} t}} \mathcal{B}&= -(\mathcal{B} A_\beta^\mathsf{T} + A_\beta \mathcal{B}) - \frac{1}{4} (
        (2 \zeta + z^2 d c_{11}) \cdot 2 \mathcal{B} - z^2 c_{11} \cdot 2 \mathrm{tr}(\mathcal{B}) I
        ) \\
        &= -\big(
        A_\beta \mathcal{B} + \mathcal{B} A_\beta^\mathsf{T}
        + \zeta \mathcal{B} + z^2 \frac{d c_{11}}{2} (\mathcal{B} - \frac{\mathrm{tr}(\mathcal{B})}{d} I)
        \big),
    \end{align*}
    which yields
    \begin{equation} \label{equ:secondMomentEquWithSelfSim}
        \begin{aligned}
            0
            &=
            \frac{\mathrm{d}}{{\mathrm{d} t}} \mathcal{B} + A_\beta \mathcal{B} + \mathcal{B} A_\beta^\mathsf{T} + \zeta \mathcal{B} + z^2 \frac{d c_{11}}{2} (\mathcal{B} - \frac{\mathrm{tr}(\mathcal{B})}{d} I) \\
            &=
            \frac{\mathrm{d}}{{\mathrm{d} t}} \mathcal{B} + A \mathcal{B} + \mathcal{B} A^\mathsf{T} + 2 (\beta + \frac{\zeta}{2}) \mathcal{B} + z^2 \frac{d c_{11}}{2} (\mathcal{B} - \frac{\mathrm{tr}(\mathcal{B})}{d} I) \\
            &= \frac{\mathrm{d}}{{\mathrm{d} t}} \mathcal{B} + (2 \beta + \zeta + \tilde{c}) \mathcal{B} + L \mathcal{B},
        \end{aligned}
    \end{equation}
    where $A_\beta$ is defined in \eqref{DefAbeta}, \( L\mathcal{B} = A \mathcal{B} + \mathcal{B} A^\mathsf{T} - \frac{\tilde{c}}{d} \mathrm{tr}(\mathcal{B}) I \) and \( \tilde{c} = z^2 \frac{d c_{11}}{2} > 0 \). Notice that our equation contains the classical elastic case where \( {e_\mathrm{res}} = 1 \), which leads to \( z = 1 \) and \( \zeta = 0 \), see \cite{JamesEtAl17SSP,BobylevEtAl20SSA}.

In terms of \eqref{equ:secondMomentEquWithSelfSim}, we denote the steady second moment equation {for symmetric B} by
    \begin{equation} \label{stableeq}
        (2 \tilde{\beta} + \tilde{c}) B + L B
        = (2 \beta + \zeta + \tilde{c}) B + L B
        = 0,
    \end{equation}
    where \( \tilde{\beta} = \beta + \zeta / 2 \). Here and in the sequel, for the steady case we have used the constant matrix $B$ instead of the time-dependent matrix $\mathcal{B}$.

    \subsection{Second moment for uniform shear flow} \label{sec:discussionOnUSF}
    We now consider the second moment equation \eqref{equ:secondMomentEquWithSelfSim} in the case where \( A = \alpha E_{12} \) {, defined in \eqref{DefUSF},} is the uniform shear flow matrix, look for a symmetric stationary solution and study the eigensystem of \( -L \). In such case, we can have a precise description on whether the temperature is increasing or decreasing in terms of the comparison of effects between shear heating and inelastic cooling.

    \begin{theorem}  \label{thm.ss.comsign}
        On a given non-degenerate cutoff collision kernel {satisfying \eqref{Defnondegenerate},} and inelastic parameter \( z \in (1 / 2, 1) \), for the inelastic Boltzmann equation under uniform simple shear heating
        \begin{equation*}
            {\partial_{t}} f - \alpha v_2 {\partial_{v_1}}f = Q_e(f, f),
        \end{equation*}
        there exists a critical shearing parameter \( \alpha_0 = \alpha_0(z) > 0 \) defined by \eqref{equ:balanceUSFParaWithoutSelfSim} in the proof such that for the {mild} solution \( f(t,v) \) with temperature \( T(t) = \int_{\mathbb{R}^d} |v|^2 f(t,v) \mathrm{d}{v} \) on an initial probability measure with zero momentum and finite energy, the following holds:
        \begin{itemize}
            \item
            If \( \alpha > \alpha_0 \), \( T(t) \) grows exponentially with rate \( 2 \gamma - \zeta > 0 \);

            \item
            If \( 0 \leq \alpha < \alpha_0 \), \( T(t) \) decays to zero exponentially with rate \( \zeta - 2 \gamma > 0 \);

            \item
            If \( \alpha = \alpha_0 \), \( T(t) \) converges exponentially to some finite positive value with rate \( \zeta - 2 \sigma \geq \tilde{c} > 0 \).

        \end{itemize}
        Here $\zeta$ is defined in \eqref{Defzeta}, and \( \gamma = \gamma(\alpha) \), \( \sigma = \sigma(\alpha) \) are defined in \eqref{equ:usfParaRelationRealSol} and \eqref{equ:usfParaRelationComplexSol} in the proof, respectively.
    \end{theorem}

    \begin{proof}
        Rewrite \eqref{stableeq} {for symmetric B} as
        \begin{equation} \label{equ:usfSecondMomentEquWithSelfSim}
            (2 \tilde{\beta} + \tilde{c}) B + \alpha E_{12} B + \alpha B E_{21} - \frac{\tilde{c}}{d} \mathrm{tr}(B) I = 0.
        \end{equation} By enumerating the indices, one has
        \begin{align}\label{2momentumequation}
            \displaystyle
            \left\{
            \begin{array}{rl}
                (2 \tilde{\beta} + \tilde{c}) B_{11} + 2 \alpha B_{12} - \frac{\tilde{c}}{d} \mathrm{tr}(B)                              = 0 & \  \mathrm{if}\; i = j = 1,                   \\
                (2 \tilde{\beta} + \tilde{c}) B_{1j} + \alpha B_{2j}                                                                                            = 0&  \  \mathrm{if}\; i = 1,\, j \neq 1,                    \\
                (2 \tilde{\beta} + \tilde{c}) B_{ii} - \frac{\tilde{c}}{d} \mathrm{tr}(B)                                                                            = 0 & \  \mathrm{if}\; i = j \neq 1,                \\
                (2 \tilde{\beta} + \tilde{c}) B_{ij}                                                                                                           = 0 & \  \mathrm{if}\ i, j, 1\ \mathrm{are\ distinct\ to\ each\ other,}\  d \geq 3.
            \end{array} \right.
        \end{align}

        Consider first the degenerate case \( \alpha = 0 \), or equivalently \( A = 0 \). Then \eqref{equ:usfSecondMomentEquWithSelfSim} becomes
        \begin{equation*}
            (2 \tilde{\beta} + \tilde{c}) B - \frac{\tilde{c}}{d} \mathrm{tr}(B) I = 0.
        \end{equation*}
        It is easy to see that there are only two solutions {$(\tilde{\beta},B)$ to the above equation}, the first one is \( \tilde{\beta} = -\tilde{c} / 2 \), \( \mathrm{tr}(B) = 0 \), which corresponds to the case that the eigenspace is of codimension \( 1 \) {in the space of symmetric matrices}, and the second one is \( \tilde{\beta} = 0 \), \( B =c I \) for any constant $c$, which is the eigenspace of dimension \( 1 \).

        Assume now that \( \alpha > 0 \). In the case where \( \tilde{\beta} = -\tilde{c} / 2 \) (corresponding to eigenvalue \( 0 \) of \( -L \)), we have \( \mathrm{tr}(B) = 0 \), which implies that the only symmetric positive semidefinite solution is \( B = 0 \). More precisely, the solution space consists exactly of symmetric matrices \( B \) with \( \mathrm{tr}(B) = 0 \) and \( B_{2j} = 0 \) for all \( j \), so it has codimension \( d + 1 \).

        Before proceeding to other solutions of \eqref{equ:usfSecondMomentEquWithSelfSim}, let us consider the behavior of \( (-L)^2 \) for \( d \geq 3 \). It is straightforward to get that
        \begin{equation*}
            L^2 B
            = 2 \alpha^2 E_{12} B E_{21} - \frac{\alpha \tilde{c}}{d} \mathrm{tr}(B) (E_{12} + E_{21}) - \frac{\tilde{c}}{d} (2 \alpha B_{12} - \tilde{c} \mathrm{tr}(B)) I.
        \end{equation*}
        Comparing each element, it holds that the kernel of \( (-L)^2 \) is defined by
        \begin{align*}
            \displaystyle
            \left\{
            \begin{array}{rl}
                0&= (L^2 B)_{11}
                = - 2 \frac{\tilde{c} \alpha}{d} B_{12} + 2 \alpha^2 B_{22} + \frac{\tilde{c}^2}{d} \mathrm{tr}(B),  \\
                0&= (L^2 B)_{12}
                = -\frac{\alpha \tilde{c}}{d} \mathrm{tr}(B),  \\
                0&= (L^2 B)_{ii}
                = - 2 \frac{\tilde{c} \alpha}{d} B_{12} + \frac{\tilde{c}^2}{d} \mathrm{tr}(B),
                \quad\mathrm{if}\; i \geq 2,
            \end{array} \right.
        \end{align*}
        or equivalently \( \mathrm{tr}(B) = B_{12} = B_{22} = 0 \), giving a nullspace of codimension \( 3 \).

        We now assume that \( 2 \tilde{\beta} + \tilde{c} \neq 0 \). It follows from \eqref{2momentumequation} that
        \begin{align*}
            &B_{1j}
            = 0 \quad \mathrm{if} \; j \geq 3, \\
            &B_{ii}
            = \frac{\tilde{c}}{d (2 \tilde{\beta} + \tilde{c})} \mathrm{tr}(B) \quad \mathrm{if}\; i \geq 2, \\
            &B_{12}
            = -\frac{\alpha}{2 \tilde{\beta} + \tilde{c}} B_{22}
            = -\frac{\alpha \tilde{c}}{d (2 \tilde{\beta} + \tilde{c})^2} \mathrm{tr}(B), \\
            &B_{11}
            = -\frac{2 \alpha}{2 \tilde{\beta} + \tilde{c}} B_{12} + \frac{\tilde{c}}{d (2 \tilde{\beta} + \tilde{c})} \mathrm{tr}(B)
            = \frac{\tilde{c}}{d (2 \tilde{\beta} + \tilde{c})} (1 + \frac{2 \alpha^2}{(2 \tilde{\beta} + \tilde{c})^2}) \mathrm{tr}(B).
        \end{align*}
        Hence, we have
        \begin{align}\label{trB}
            \mathrm{tr}(B)
            &= B_{11} + \sum\limits_{i = 2}^{d} B_{ii}
            = \frac{\tilde{c}}{d (2 \tilde{\beta} + \tilde{c})} (\frac{2 \alpha^2}{(2 \tilde{\beta} + \tilde{c})^2} + d) \mathrm{tr}(B).
        \end{align}
        If \( \mathrm{tr}(B) = 0 \), then \( B_{ii} = 0 \) for all \( i \) and \( B_{12} = 0 \), which implies \( B = 0 \). If \( \mathrm{tr}(B) \neq 0 \), then it follows from \eqref{trB} that $d \tilde{\beta} (2 \tilde{\beta} + \tilde{c})^2 = \tilde{c} \alpha^2$, or equivalently,
        \begin{equation} \label{equ:usfParaRelation}
            \begin{aligned}
                \tilde{\beta} (4 \tilde{\beta} + z^2 d c_{11})^2 &= 2 z^2 c_{11} \alpha^2.
            \end{aligned}
        \end{equation}
Direct calculation shows that the derivative of the cubic has two distinct
       	zeros, and the cubic is negative at the smallest zero, which indicates that \eqref{equ:usfParaRelation} has only one real root \( \tilde{\beta} = \gamma \)
    \begin{equation} \label{equ:usfParaRelationRealSol}
            \gamma=\frac{\tilde{c}}{3}(-1+\cosh Y)
            = \frac{2 \tilde{c}}{3} \sinh^2\frac{Y}{2}
            = \frac{z^2 d c_{11}}{3} \sinh^2(\frac{1}{6} \mathrm{arcosh}\big(1 + \frac{108}{d^3} (\frac{\alpha}{z^2 c_{11}})^2\big))
            \in (0, \frac{\alpha}{2 \sqrt{2 d}}],
        \end{equation}
        with \(Y = \frac{1}{3} \mathrm{arcosh}(1 + \frac{27}{d} (\alpha / \tilde{c})^2) = \frac{1}{3} \mathrm{arcosh}(1 + \frac{108}{d^3} (\frac{\alpha}{z^2 c_{11}})^2) > 0 \). {One may refer to page 473–477 in \cite{Holmes} for the formula.} This solution \( \alpha \mapsto \gamma \) is bijective on \( \mathbb{R}^+ \), strictly increasing, and has asymptotic \( \gamma \sim \alpha^2 \) on \( \alpha \ll 1 \) and \( \gamma \sim \alpha^{2 / 3} \) on \( \alpha \gg 1 \).

        The remaining two complex roots are
        \begin{equation} \label{equ:usfParaRelationComplexSol}
            \gamma_\pm
            = \frac{\tilde{c}}{3} (-1 - \frac{1}{2} \cosh Y \pm i \frac{\sqrt{3}}{2} \sinh Y)
            = \sigma_{R} \pm i \omega,
        \end{equation}
        with {real part \( \sigma_{R} = \mathbb{R}e\gamma_\pm = -\frac{\tilde{c}}{6} (2 + \cosh Y) < -\tilde{c} / 2 < 0 \), and imaginary part} \( \omega = \frac{\tilde{c}}{2 \sqrt{3}} \sinh Y > 0 \), and {\( 2 \gamma - 2 \sigma_{R} = \frac{\tilde{c}}{3} \cosh Y > 0 \), \( \gamma + 2 \sigma_{R} = -\tilde{c} <0 \).}

        For a root \( \tilde{\beta} \) of \eqref{equ:usfParaRelation}, the solution space of \eqref{equ:usfSecondMomentEquWithSelfSim} is of dimension \( 1 \), and each solution takes the form
        \begin{equation} \label{equ:usfSecondMomentSol}
            B
            = \mathrm{tr}(B) \Big(
            \frac{1}{d (1 + 2 \tilde{\beta} / \tilde{c})} I
            - \frac{\tilde{\beta}}{\alpha} (E_{12} + E_{21})
            + (1 - \frac{1}{1 + 2 \tilde{\beta} / \tilde{c}}) E_{11}
            \Big).
        \end{equation}
        Particularly, for the real root \( \gamma \),
        \begin{equation} \label{equ:usfSecondMomentRealSol}
            B = \frac{\mathrm{tr}(B)}{d (1 + 2 \gamma / \tilde{c})} \Big(
            I
            - \sqrt{d \frac{\gamma}{\tilde{c}}} (E_{12} + E_{21})
            + 2 d \frac{\gamma}{\tilde{c}} E_{11}
            \Big).
        \end{equation}

        Thus, assuming \( \mathrm{tr}(B) = d (1 + 2 \gamma / \tilde{c}) > 0 \), all eigenvalues of such \( B \) are, in non-increasing order,
        \begin{align*}
            \Lambda_1 &= 1 + d \frac{\gamma}{\tilde{c}} + \sqrt{d \frac{\gamma}{\tilde{c}} (1 + d \frac{\gamma}{\tilde{c}})}, \\
            \Lambda_2& = \ldots = \Lambda_{d - 1} = 1, \\
            \Lambda_d &= 1 + d \frac{\gamma}{\tilde{c}} - \sqrt{d \frac{\gamma}{\tilde{c}} (1 + d \frac{\gamma}{\tilde{c}})},
        \end{align*}
        which are all positive. Thus, \( B \) is positive definite, with the property that \( \Lambda_1 \to \infty \), \( \Lambda_d \to 1 / 2 \) when \( \alpha \to \infty \), and \( \Lambda_1^{-1} + \Lambda_d^{-1} = 2 \).

        Thus, on a given self-similar parameter \( \beta \), the operator \( B \mapsto (2 \beta + \zeta + \tilde{c}) B + L B \) has eigenvalues \( 2 \beta + \zeta - 2 \gamma, 2 \beta + \zeta - 2 \gamma_+, 2 \beta + \zeta - 2 \gamma_-, 2 \beta + \zeta + \tilde{c} \) where all except the last are simple eigenvalues, and  \( 2 \beta + \zeta + \tilde{c} \) has \( D - d - 1 \) linearly independent eigenvectors and \( d - 2 \) generalized eigenvectors of order \( 2 \) with \( D = \frac{d (d + 1)}{2} \) being the dimension of the space of symmetric \( d \times d \) complex matrices. Combined, the solution \( \mathcal{B}(t) \) of \eqref{equ:secondMomentEquWithSelfSim} on self-similar parameter \( \beta \) takes the form
        \begin{align*}
            \mathcal{B}(t)
            &= c_0 e^{-(2 \beta + \zeta - 2 \gamma) t} B_0
            + c_+ e^{-(2 \beta + \zeta - 2 \sigma_{R} - 2 i \omega) t} B_+
            + c_- e^{-(2 \beta + \zeta - 2 \sigma_{R} + 2 i \omega) t} B_- \\
            &\quad + e^{-(2 \beta + \zeta + \tilde{c}) t} (c_1 B_1 + \ldots + c_{D - d - 1} B_{D - d - 1})
            + t e^{-(2 \beta + \zeta + \tilde{c}) t} (c_1' B_1' + \ldots + c_{d - 2}' B_{d - 2}')
        \end{align*}
        for some constants \( c_0, c_+, c_-, c_1, \ldots, c_{D - d - 1}, c_1', \ldots, c_{d - 2}' \in \mathbb{C} \), nonzero matrices \( B_0, B_+, B_- \) of the form of \eqref{equ:usfSecondMomentRealSol} and \eqref{equ:usfSecondMomentSol}, nonzero symmetric matrices \( B_1, \ldots, B_{D - d - 1} \) which have trace and \( (2, j) \)-entries zero for all \( j \), and nonzero symmetric matrices \( B_1', \ldots, B_{d - 2}' \) which have trace, \( (1, 2) \)-entry, and \( (2, 2) \)-entry zero. Assuming \( B_0, B_+, B_- \) are normalized to have unit trace and noting that all except these three terms have zero trace, taking trace on \( \mathcal{B} \) yields the {temperature} \( T \coloneqq \mathrm{tr}(\mathcal{B}) \) of the form
        \begin{equation*}
            T(t) = c_0 e^{-(2 \beta + \zeta - 2 \gamma) t} + e^{-(2 \beta + \zeta - 2 \sigma_{R}) t} (c_+ e^{2 i \omega t} + c_- e^{-2 i \omega t}).
        \end{equation*}
        As \( T(t) \) should be a real number for all \( t \), we have that \( c_0 \in \mathbb{R} \) and \( c_+ = \overline{c_-} \). On \( c_+ = r + i m \) with \( r, m \in \mathbb{R} \), we have
        \begin{equation*}
            e^{2 \beta t} T(t)
            = c_0 e^{(2 \gamma - \zeta) t} + e^{(2 \sigma_{R} - \zeta) t} (2 r \cos(2 \omega t) - 2 m \sin(2 \omega t))
        \end{equation*}
        for some constants \( c_0, r, m \in \mathbb{R} \). By similarly considering the equation for \( \mathcal{B}_{12} \) and \( \mathcal{B}_{22} \), the three coefficients satisfy the relation
        \begin{align*}
            T(0) &= c_0 + 2 r, \\
            \mathcal{B}_{12}(0) &= -\frac{\gamma}{\alpha} c_0 - \frac{2 \sigma_{R}}{\alpha} r + \frac{2 \omega}{\alpha} m, \\
            \mathcal{B}_{22}(0) &= \frac{\tilde{c}}{d (2 \gamma + \tilde{c})} c_0 + \frac{2 \tilde{c} (2 \sigma_{R} + \tilde{c})}{d ((2 \sigma_{R} + \tilde{c})^2 + 4 \omega^2)} r + \frac{4 \tilde{c} \omega}{d ((2 \sigma_{R} + \tilde{c})^2 + 4 \omega^2)} m.
        \end{align*}
        Solving the system gives
        \begin{equation*}
            \begin{pmatrix}
                c_0 \\ r \\ m
            \end{pmatrix}
            = \frac{d}{\tilde{c}^2 \alpha^2} \left( 27 + 2 d \frac{\tilde{c}^2}{\alpha^2} \right)^{-\frac{1}{2}}
            \begin{pmatrix}
                4 \tilde{c} \omega (\sigma_{R}^2 + \omega^2) & 4 \tilde{c} \alpha \sigma_{R} \omega & 2 \tilde{c} \alpha^2 \omega \\
                -\frac{2}{3} \tilde{c} \omega (\tilde{c} \gamma - 8 \omega^2) & -2 \tilde{c} \alpha \sigma_{R} \omega & -\tilde{c} \alpha^2 \omega \\
                \frac{1}{6} \tilde{c}^2 (\tilde{c} \gamma - 8 \omega^2) & -\frac{1}{2} \tilde{c} \alpha (\tilde{c} \gamma - 8 \omega^2) & \frac{1}{2} \tilde{c} \alpha^2 (\tilde{c} + 3 \gamma)
            \end{pmatrix}
            \begin{pmatrix}
                T(0) \\ \mathcal{B}_{12}(0) \\ \mathcal{B}_{22}(0)
            \end{pmatrix}.
        \end{equation*}
        Particularly, we can see that \( c_0, r, m \), and so the precise behavior of \( T(t) \), are determined by the initial second moment matrix \( \mathcal{B}_0 = \mathcal{B}(0) \) only via \( \mathrm{tr}(\mathcal{B}_0) = \int |v|^2 f_0 \), \( (\mathcal{B}_0)_{12} = \int v_1 v_2 f_0 \) and \( (\mathcal{B}_0)_{22} = \int v_2^2 f_0 \), and the oscillating {terms} in \( T(t) \) vanish if these quantities {satisfy} \eqref{equ:usfSecondMomentRealSol}.

        Furthermore, we can see that in the long-time asymptotic, the energy
        \begin{itemize}
            \item
            grows exponentially if \( 2 \gamma > \zeta + 2 \beta \) (shearing dominant),

            \item
            decays exponentially if \( 2 \gamma < \zeta + 2 \beta \) (cooling dominant), and

            \item
            converges to some constant \( c_0 > 0 \) with rate \( \zeta - 2 \sigma_{R} > 0 \) if \( 2 \gamma = \zeta + 2 \beta \),
        \end{itemize}
    {where $\gamma$ depending on $z,\alpha,c_{11}$ and $d$ is defined in \eqref{equ:usfParaRelationRealSol},  and $\zeta$ depending only on $z$ is given in \eqref{Defzeta}.}
        Without self-similar scaling,
        the last case happens if \( \beta = 0 \) solves \eqref{equ:usfParaRelation}, or equivalently
        \begin{equation} \label{equ:balanceUSFParaWithoutSelfSim}
            \alpha=\alpha_0
            := (\zeta + z^2 \frac{d c_{11}}{2}) \sqrt{\frac{\zeta}{z^2 c_{11}}}
            = ((1 - z) (b_0 - b_1) + z \frac{d c_{11}}{2}) \sqrt{z (1 - z) \frac{b_0 - b_1}{c_{11}}}
            > 0,
        \end{equation}
        for $1/2<z<1$. It's also direct to see $\alpha_0\to 0$ as $z\to 1-$.
        This then completes the proof of Theorem \ref{thm.ss.comsign}.
    \end{proof}

    Note that \( \alpha_0 \) as in \eqref{equ:balanceUSFParaWithoutSelfSim} has a finite upper bound \( \alpha_0 < \frac{1}{4} (b_0 - b_1 + d c_{11} / 2) \sqrt{(b_0 - b_1) / c_{11}} < \infty \).

    \subsection{Stationary self-similar Radon profile} \label{sec:statRadonProfile}
    In \cite{JamesEtAl17SSP}, the existence of a stationary self-similar profile for elastic collision can be shown under the assumption that the shear strength $\|A\|$ is sufficiently small.
    However, due to the appearance of cooling effect for inelastic collision, we need to further impose the smallness of $\zeta$ as in \eqref{Defzeta} or equivalently the condition that the inelastic parameter $z$ is close to $1$, to establish the existence. {Here, the shearing matrix $A$ is not necessarily associated to a uniform shear.}

    \begin{theorem} \label{thm:existenceOfStatProfileWithSmallness}
        There exists \( \epsilon > 0 \) such that if \( \|A\| < \epsilon \) and \( 1 - z < \epsilon \), then there is some \( \beta \in \mathbb{R} \) and {a symmetric} positive definite matrix \( B \) with \( \mathrm{tr}(B) =  1 \) such that for all \( \Theta > 0 \), a self-similar profile \( G \in \mathcal{M}_+ \) exists where \eqref{equ:inelasticBoltzmannSelfSimStationary} is satisfied in the sense of measure with \( {\int_{\mathbb{R}^d}^{}} (1, v, {{v}^{\otimes 2}}) G(v) \, {\mathrm{d}{v}} = (1, 0, \Theta B) \). In particular, \( {\int_{\mathbb{R}^d}^{}} (1, v, |v|^2) G(v) \, {\mathrm{d}{v}} = (1, 0, \Theta) \).
    \end{theorem}

    \begin{proof}
        First, we need to {show the existence of} $B$ for the stationary second moment equation \eqref{stableeq} with the scaling parameter $\beta$. {By similar perturbation arguments as} in \cite[Lemma 4.16]{JamesEtAl17SSP}, for sufficiently small \( \|A\| \ll 1 \), there exist a \( \beta \in \mathbb{R} \) and a nontrivial symmetric positive definite stationary solution \( B \) to \eqref{stableeq}.
        Moreover, $\beta$ is chosen to be the complex number with the
        largest real part for which \eqref{stableeq} holds, and  \( B \) is normalized in this theorem with \( \mathrm{tr}(B) = 1 \).
        
        Take \( C_* > 0 \) {large enough} and \( s \in (2, 4] \). Let \( K \subseteq \mathcal{M}_+ \) be the set of measures \( G \) that satisfies \( \int_{\mathbb{R}^d} (1, v, {{v}^{\otimes 2}}) G(v) \mathrm{d}{v} = (1, 0, \Theta B) \) and \( \int |v|^s G \leq C_* \). {Note that $K$ depends on $C_*$, $s$ and $\Theta$.} It is easy to see that \( K \) is nonempty, bounded in \( \mathcal{M}_+\) and convex. Furthermore, \( K \) is weakly closed, then by Banach--Alaoglu theorem it is weak-* compact.

        Let \( S : [0, \infty) \times \mathcal{M}_+(\mathbb{R}_c^{d}) \to \mathcal{M}_+(\mathbb{R}_c^{d}) \) be the mild solution operator such that \( G(t) = S(t)G_0 \) is the mild solution from Theorem \ref{thm:existenceOfMildSol} with initial condition \( G_0 \). From Theorem \ref{thm:mildSolAsWeakSolAndProperties} we can see that \( t \mapsto S(t) \) is a semigroup {on $\mathcal{M}_+(\mathbb{R}_c^{d})$}.

        We first show that \( S \) is uniformly weakly continuous on \( [0, T] \times K \) for each \( T < \infty \).
  Let \( 0 \leq t_1 \leq t_2 \leq T < \infty \), and \( G_0(v) \in \mathcal{M}_+ \) with \( \int_{\mathbb{R}^d} |v|^s G_0(v) \mathrm{d}{v} < \infty \). By Theorem \ref{thm:mildSolAsWeakSolAndProperties}, \( G(t) = S(t)G_0 \) also has finite \( s \)-moment on \( [0, T] \). Let \( \varphi \in \mathit{C}^{0}(\mathbb{R}_c^{d}) \). Then for a cutoff function \( \varphi_n \in \mathit{C}^{2}(\mathbb{R}_c^{d}) \) with
        \( |v| |{ \nabla_{v} }\varphi_n| + |v|^2 |{ \nabla_{v}^2 }\varphi_n| \leq C( 1 + |v|^s) \) for some constant $C$ uniformly in $n$, we have
        \begin{align*}
            &\quad\;\,
            \Big|\int_{\mathbb{R}^d} \varphi_n(v) G(t_1,v) \mathrm{d}{v}- \int_{\mathbb{R}^d} \varphi_n(v) G(t_2,v) \mathrm{d}{v}\Big| \\
            &\leq {\int_{t_1}^{t_2}} {\int_{\mathbb{R}^d}^{}} 
                    |A v \cdot { \nabla_{v} }\varphi_n| G(t,v)
             \, {\mathrm{d}{v}}{\mathrm{d}{t}} \\
            &\quad+ \frac{1}{2} {\int_{t_1}^{t_2}} 
                {\int_{\mathbb{R}^d}^{}} 
                    {\int_{\mathbb{R}^d}^{}} 
                        {\int_{\mathbb{S}^{d-1}}^{}} 
                            b(\cos\theta) G(t,v) G(t,v_*)
                            |\varphi_n(v') + \varphi_n(v_*') - \varphi_n(v) - \varphi_n(v_*)|
                 \, {\mathrm{d}{\sigma}}{\mathrm{d}{v_*}}{\mathrm{d}{v}}
            {\mathrm{d}{t}},
        \end{align*}
        which, together with \eqref{boundf}, yields
        \begin{align*}
            &\quad\;\,
            \Big|{\int_{\mathbb{R}^d}^{}} \varphi_n(v) G(t_1, v) \, {\mathrm{d}{v}} - {\int_{\mathbb{R}^d}^{}} \varphi_n(v) G(t_2, v) \, {\mathrm{d}{v}}\Big| \\&\leq C{\int_{t_1}^{t_2}} {\int_{\mathbb{R}^d}^{}} 
                    |v| |{ \nabla_{v} }\varphi_n| G(t,v)
             \, {\mathrm{d}{v}}{\mathrm{d}{t}}
            + C{\int_{t_1}^{t_2}} 
                {\int_{\mathbb{R}^d}^{}} 
                    {\int_{\mathbb{R}^d}^{}} 
                        G(t,v) G(t,v_*) |v - v_*|^2 |{ \nabla_{v} ^2}\varphi_n|
                 \, {\mathrm{d}{v_*}}{\mathrm{d}{v}}
            {\mathrm{d}{t}} \\
            &\leq C{\int_{t_1}^{t_2}} (1 + |v|^s) G(t,v) \, {\mathrm{d}{t}}
            + C{\int_{t_1}^{t_2}} ({\int_{\mathbb{R}^d}^{}} (1 + |v|^s)G(t,v) \, {\mathrm{d}{v}})^2{\mathrm{d}{t}} \\
            &\leq C_T|t_2 - t_1| (1 + \|G_0\|_{s})^2,
        \end{align*}
        for some constant $C_T$ depending on $T$, \(|{ \nabla_{v} }\varphi_n| \) and \( |{ \nabla^2_{v} }\varphi_n| \). {By density of our test functions in $\mathit{C}^{0}(\mathbb{R}_c^{d})$,  the continuity holds for all test functions \( \varphi \in \mathit{C}^{0}(\mathbb{R}_c^{d}) \).}

        We now show the weak continuity of \( S(t) \) on \( K \) with fixed \( t \leq T < \infty \). Motivated by \cite{JamesEtAl17SSP}, this will be done via the 1-Wasserstein distance on \( \mathcal{M} \):
        \begin{equation*}
            \mathscr{W}_1(f, g) = \sup_{\mathrm{Lip}(\varphi) \leq 1} \int_{\mathbb{R}^d} \varphi(v) (f(v) - g(v)) \mathrm{d}{v},
        \end{equation*}
        where \( \mathrm{Lip}(\varphi) = \sup_{x \neq y} \frac{|\varphi(x) - \varphi(y)|}{|x - y|} \) is the Lipschitz constant of \( \varphi \).
        Let \( f(t) = S(t)f_0, g(t) = S(t)g_0 \) be mild solutions {with}  \( f_0, g_0 \in \mathcal{M}_+ \) {and} \( {\int_{\mathbb{R}^d}^{}} (f_0(v), g_0(v)) \, {\mathrm{d}{v}} = (1, 1) \) and finite \( s \)-moments, and \( \varphi \in \mathit{C}(\mathbb{R}_c^{d}) \) with {finite Lipschitz constant that satisfies} \( \mathrm{Lip}(\varphi) \leq 1 \). Then for fixed \( 0 \leq t \leq T < \infty \) and \( \psi(t', v) = \varphi(e^{(t - t') A} v) \) on \( t' \in [0, t] \) which satisfies \( {\partial_{t'}} \psi + A v \cdot { \nabla_{v} }\psi = 0 \), we first write
        \begin{align*}
            {\int_{\mathbb{R}^d}^{}} \varphi(v) (f(t, v) - g(t, v)) \, {\mathrm{d}{v}} &- {\int_{\mathbb{R}^d}^{}} \varphi(e^{t A} v) (f_0(v) - g_0(v)) \, {\mathrm{d}{v}} \\
            &= {\int_{\mathbb{R}^d}^{}} \psi(t, v) (f(t, v) - g(t, v)) \, {\mathrm{d}{v}} - {\int_{\mathbb{R}^d}^{}} \psi(0, v) (f_0(v) - g_0(v)) \, {\mathrm{d}{v}},
        \end{align*}
        and it further follows from the weak formulation \eqref{equ:inelasticBoltzmann_weakForm} that
        \begin{equation} \label{Wfg}
            \begin{aligned}[b]
                &\quad\;\,
                \int_{\mathbb{R}^d} \varphi (f(t,v) - g(t,v)) \mathrm{d}{v} - \int_{\mathbb{R}^d} \varphi(e^{t A} v) (f_0(v) - g_0(v)) \mathrm{d}{v} \\
                &= \frac{1}{2} {\int_{0}^{t}} 
                    {\int_{\mathbb{R}^d}^{}} 
                        {\int_{\mathbb{R}^d}^{}} 
                            {\int_{\mathbb{S}^{d - 1}}^{}} 
                                b(\frac{v - v_*}{|v - v_*|} \cdot \sigma) \\
                                &\qquad\qquad\qquad\qquad\times f(\tau,v) (f(\tau,v_*) - g(\tau,v_*)) (\psi(\tau,v') + \psi(\tau,v_*') - \psi(\tau,v) - \psi(\tau,v_*))
                             \, {\mathrm{d}{\sigma}}
                        {\mathrm{d}{v_*}}
                    {\mathrm{d}{v}}
                {\mathrm{d}{\tau}} \\
                &\quad
                + \frac{1}{2} {\int_{0}^{t}} 
                    {\int_{\mathbb{R}^d}^{}} 
                        {\int_{\mathbb{R}^d}^{}} 
                            {\int_{\mathbb{S}^{d - 1}}^{}} 
                                b(\frac{v - v_*}{|v - v_*|} \cdot \sigma) \\
                                &\qquad\qquad\qquad\qquad\times  (f(\tau,v) - g(\tau,v)) g(\tau,v_*) (\psi(\tau,v') + \psi(\tau,v_*') - \psi(\tau,v) - \psi(\tau,v_*))
                             \, {\mathrm{d}{\sigma}}
                        {\mathrm{d}{v_*}}
                    {\mathrm{d}{v}}
                {\mathrm{d}{\tau}} \\
                &\quad - {\int_{}^{}} \psi(0,v) (f_0(v) - g_0(v)) \, {\mathrm{d}{v}}.
            \end{aligned}
        \end{equation}
        Since \( \mathrm{Lip}(\varphi) \leq 1 \), on \( \tau \in [0, T] \) we have \( \mathrm{Lip}(\psi(\tau)) = \mathrm{Lip}(\varphi(e^{(t - \tau) A} v)) \leq \mathrm{Lip}(\varphi) \|e^{(t - \tau) A}\| \leq \sup_{\tau\in[0, T]} \|e^{\tau A}\| \) which is a constant that depends only on \( T \). Hence, one has
        \begin{equation*}
            {\int_{\mathbb{R}^d}^{}} \varphi(e^{t A} v) (f_0(v) - g_0(v)) \, {\mathrm{d}{v}}
            = \sup_{\tau\in[0, T]} \|e^{\tau A}\| {\int_{\mathbb{R}^d}^{}} \frac{\varphi(e^{t A} v)}{\sup_{[0, T]} \|e^{\tau' A}\|} (f_0(v) - g_0(v)) \, {\mathrm{d}{v}}
            \leq C\mathscr{W}_1(f_0, g_0).
        \end{equation*}
        Notice that
        \begin{align}
            {\mathrm{Lip}_{v_*}\Big(
            {\int_{\mathbb{R}^d}^{}} 
                {\int_{\mathbb{S}^{d - 1}}^{}} 
                    b(\frac{v - v_*}{|v - v_*|} \cdot \sigma)f(\tau,v) \psi(\tau,v_*)
                 \, {\mathrm{d}{\sigma}}
            {\mathrm{d}{v}}
            \Big)}
            &\leq C_T \mathrm{Lip}_{v_*}(\psi(\tau,v_*))
            \leq C_T, \label{Lip*}\\
            {\mathrm{Lip}_{v_*}\Big(
            -{\int_{\mathbb{R}^d}^{}} 
                {\int_{\mathbb{{S}^{d - 1}}}^{}} 
                    b(\frac{v - v_*}{|v - v_*|} \cdot \sigma)f(\tau,v) \psi(\tau,v)
                 \, {\mathrm{d}{\sigma}}
            {\mathrm{d}{v}}
            \Big)}
            &=0 \label{Lip} .
        \end{align}
        To estimate \( \mathrm{Lip}_{v_*}(\iint b f(\tau,v) \psi(\tau,v') \mathrm{d}{\sigma} \mathrm{d}{v}) \), let \( \hat{e} \in \mathbb{S}^{d-1} \) be a fixed unit vector, and \( R = R(v, v_*) \) be a rotation such that \( R^\mathsf{T}\frac{v - v_*}{|v - v_*|} = \hat{e} \). Then we use change of variables to get
        \begin{equation} \label{transR}
            \begin{aligned}[b]
                &\quad\;\,
                {\int_{\mathbb{R}^d}^{}} 
                    {\int_{\mathbb{S}^{d - 1}}^{}} 
                        b(\frac{v - v_*}{|v - v_*|} \cdot \sigma) f(\tau,v) \psi(\tau,v)
                     \, {\mathrm{d}{\sigma}}
                {\mathrm{d}{v}} \\
                &= {\int_{\mathbb{R}^d}^{}} 
                    {\int_{\mathbb{S}^{d - 1}}^{}} 
                        b(\frac{v - v_*}{|v - v_*|} \cdot R\sigma) {f(\tau,v)} \psi\Big(\tau, \frac{v + v_*}{2} + \frac{1 - z}{2} (v - v_*) + \frac{z}{2} |v - v_*| R \sigma\Big)
                     \, {\mathrm{d}{\sigma}}
                {\mathrm{d}{v}} \\
                &= {\int_{\mathbb{R}^d}^{}} 
                    {\int_{\mathbb{S}^{d - 1}}^{}} 
                        b(\hat{e} \cdot \sigma) {f(\tau,v)} \psi(\tau, \tilde{v}')
                     \, {\mathrm{d}{\sigma}}
                {\mathrm{d}{v}},
            \end{aligned}
        \end{equation}
        with \( \tilde{v}'(v, v_*, \sigma) = \frac{v + v_*}{2} + \frac{1 - z}{2} (v - v_*) + \frac{z}{2} |v - v_*| R(v, v_*) \sigma \).
        Since
        \begin{equation*}
            \left\|{ \nabla_{v_*} } (|v - v_*| R(v, v_*) \sigma)\right\|
            \leq \left\|(R(v, v_*) \sigma) \otimes \frac{v - v_*}{|v - v_*|}\right\| + |v - v_*| \left\|{\partial_{v_*}} \frac{v - v_*}{|v - v_*|}\right\|
            \leq C,
        \end{equation*}
        uniformly on \( v_* \), we have
        \begin{equation*}
            \mathrm{Lip}_{v_*}(\tilde{v}')
            \leq \frac{1}{2} + \frac{1 - z}{2} + \frac{z}{2} \mathrm{Lip}_{v_*}(|v - v_*| R(v, v_*) \sigma)
            \leq C + C\sup_{v_*} \|{ \nabla_{v_*} } (|v - v_*| R(v, v_*) \sigma)\|
            \leq C,
        \end{equation*}
        which, together with \eqref{transR}, implies
        \begin{equation} \label{Lip'}
            \begin{aligned}[b]
                &\mathrm{Lip}_{v_*}(
                {\int_{\mathbb{R}^d}^{}} 
                    {\int_{\mathbb{S}^{d - 1}}^{}} 
                        b(\frac{v - v_*}{|v - v_*|} \cdot \sigma) f(\tau,v) \psi(\tau,v')
                     \, {\mathrm{d}{\sigma}}
                {\mathrm{d}{v}}
                ) \\
                \leq& C_T {\int_{\mathbb{R}^d}^{}} 
                    {\int_{\mathbb{S}^{d - 1}}^{}} 
                        b(\hat{e} \cdot \sigma) f(\tau,v)\mathrm{Lip}_{\tilde{v}'}(\psi(\tau,\tilde{v}')) \mathrm{Lip}_{v_*}(\tilde{v}')
                     \, {\mathrm{d}{\sigma}}
                {\mathrm{d}{v}}
                \leq C_T.
            \end{aligned}
        \end{equation}
        Similar arguments give that
        \begin{align}\label{Lip'*}
            \mathrm{Lip}_{v_*}(
            {\int_{\mathbb{R}^d}^{}} 
                {\int_{\mathbb{S}^{d - 1}}^{}} 
                    b(\frac{v - v_*}{|v - v_*|} \cdot \sigma) f(\tau,v) \psi(\tau,v_*')
                 \, {\mathrm{d}{\sigma}}
            {\mathrm{d}{v}}
            ) \leq C_T.
        \end{align}
        Combining \eqref{Lip*}, \eqref{Lip}, \eqref{Lip'} and \eqref{Lip'*}, we have
        \begin{equation*}
            \mathrm{Lip}_{v_*}(
            {\int_{\mathbb{R}^d}^{}} 
                {\int_{\mathbb{S}^{d - 1}}^{}} 
                    b(\frac{v - v_*}{|v - v_*|} \cdot \sigma) f(\tau,v) (\psi'(\tau) + \psi_*'(\tau) - \psi(\tau) - \psi'(\tau))
                 \, {\mathrm{d}{\sigma}}
            {\mathrm{d}{v}}
            )
            \leq C_T,
        \end{equation*}
        which gives
        \begin{align*}
            &{\int_{\mathbb{R}^d}^{}} 
                {\int_{\mathbb{R}^d}^{}} 
                    {\int_{\mathbb{S}^{d - 1}}^{}} 
                        b(\frac{v - v_*}{|v - v_*|} \cdot \sigma) f(\tau,v) (f(\tau,v_*) - g(\tau,v_*)) (\psi(\tau,v') + \psi(\tau,v_*') - \psi(\tau,v) - \psi(\tau,v_*))
                     \, {\mathrm{d}{\sigma}}
                {\mathrm{d}{v_*}}
            {\mathrm{d}{v}} \\
            \leq& C_T\mathscr{W}_1(f(\tau), g(\tau)),
        \end{align*}
        and
        \begin{align*}
            &{\int_{\mathbb{R}^d}^{}} 
                {\int_{\mathbb{R}^d}^{}} 
                    {\int_{\mathbb{S}^{d - 1}}^{}} 
                        b(\frac{v - v_*}{|v - v_*|} \cdot \sigma)(f(\tau,v) - g(\tau,v)) g(\tau,v_*) (\psi(\tau,v') + \psi(\tau,v_*') - \psi(\tau,v) - \psi(\tau,v_*))
                     \, {\mathrm{d}{\sigma}}
                {\mathrm{d}{v_*}}
            {\mathrm{d}{v}} \\
            \leq& C_T\mathscr{W}_1(f(\tau), g(\tau)).
        \end{align*}
        It holds from the above two inequalities and \eqref{Wfg} that
        \begin{equation*}
            \mathscr{W}_1(f(t), g(t)) \leq C_T ( \mathscr{W}_1(f_0, g_0) + {\int_{0}^{t}} \mathscr{W}_1(f(\tau), g(\tau)) \, {\mathrm{d}{\tau}}),
        \end{equation*}
        which further yields by Gronwall inequality that \( \mathscr{W}_1(f(t), g(t)) \leq C_T \mathscr{W}_1(f_0, g_0) \) with the constant $C_T$ depending only on \( T \). Hence, for each \( t > 0 \), \( S(t) : K \to \mathcal{M}_+ \) is weakly continuous.

        Recalling \( K \subseteq \mathcal{M}_+ \) is the set of measures \( G \) with \( \int (1, v, {{v}^{\otimes 2}}) G = (1, 0, \Theta B) \) and \( \int |v|^s G \leq C_* \), we now show that \( S(t)K \subseteq K \) if \( C_* \) is sufficiently large. By property of mild solution and the selection of \( \beta \), it suffices to show that \( \int_{\mathbb{R}^d} |v|^s G(v) \mathrm{d}{v} \leq C_* \) on \( G(t) = S(t)G_0 \) for \( G_0 \in K \). Let \( G_{n0} = G_0 {\chi_{|v| \leq n}} \) and \( G_n(t) = S(t)G_{n0} \). As in Theorem \ref{thm:mildSolAsWeakSolAndProperties}, we have \( \|G_n(t) - G(t)\|_{2} \to 0 \) {when $n\rightarrow\infty$}. By the selection of \( B \) and the fact that \( \int_{\mathbb{R}^d} (1 + |v|^2) G(t,v) \mathrm{d}{v} = 1 + \Theta > 0 \) is constant, then for sufficiently large \( n \) we have $$ \int_{\mathbb{R}^d} (1 + |v|^2) G_n(t,v) \mathrm{d}{v} \leq C \int_{\mathbb{R}^d} (1 + |v|^2) G(t,v) \mathrm{d}{v} \leq C_T $$ at the given time \( t\leq T < \infty \) with constant independent of \( n \). 

        For \( s > 2 \), \cite[Lem.~3.3]{GambaEtAl04BED} gives the following Povzner inequality
        \begin{align}\label{Povzner}
            &{\int_{\mathbb{S}^{d-1}}^{}} b(\cos\theta) (|v'|^s + |v_*'|^s - |v|^s - |v_*|^s) \, {\mathrm{d}{\sigma}} = P - N,
        \end{align}
        with
        \begin{align}
            \quad
            &P \leq C(|v|^2 |v_*|^{s - 2} + |v_*|^2 |v|^{s - 2}),\label{P} \\
            &N \geq c(|v|^2 + |v_*|^2)^{s / 2},\label{N}
        \end{align}
        where the constants $C$ and $c$ depend only on $s$ and $b_0$.
        It follows from \eqref{equ:inelasticBoltzmann_weakForm}, \eqref{Povzner}, \eqref{P} and \eqref{N} that
        \begin{equation} \label{sfn}
            \begin{aligned}[b]
                \frac{\mathrm{d}}{{\mathrm{d} t}} {\int_{\mathbb{R}^d}^{}}  |v|^s G_n(t, v)  \, {\mathrm{d}{v}}
                &= -s {\int_{\mathbb{R}^d}^{}}  v^\mathsf{T} A_\beta v |v|^{s - 2} G_n(t, v)  \, {\mathrm{d}{v}} \\
                &\quad+ \frac{1}{2} {\int_{\mathbb{R}^d}^{}} 
                    {\int_{\mathbb{R}^d}^{}} 
                        G_n(t, v) G_n(t, v_*)
                        {\int_{\mathbb{S}^{d - 1}}^{}} 
                            b(\cos\theta) (|v'|^s + |v_*'|^s - |v|^s - |v_*|^s)
                         \, {\mathrm{d}{\sigma}}
                    {\mathrm{d}{v_*}}
                {\mathrm{d}{v}} \\
                &\leq s (\left\|A\right\| - \beta) {\int_{\mathbb{R}^d}^{}} 
                    |v|^s G_n(t, v)
                 \, {\mathrm{d}{v}} \\
                &\quad+ \frac{1}{2} {\int_{\mathbb{R}^d}^{}} 
                    {\int_{\mathbb{R}^d}^{}} 
                        G_n(t, v) G_n(t, v_*)
                        (
                        C (|v|^2 |v_*|^{s - 2} + |v_*|^2 |v|^{s - 2})
                        - c (|v|^s + |v_*|^s)
                        )
                    {\mathrm{d}{v_*}}
                {\mathrm{d}{v}} \\
                &= C ({\int_{\mathbb{R}^d}^{}}  |v|^2 G_n(t, v)  \, {\mathrm{d}{v}}) ({\int_{\mathbb{R}^d}^{}}  |v|^{s - 2} G_n(t, v)  \, {\mathrm{d}{v}})
                + (s \left\|A\right\| - s \beta - c) {\int_{\mathbb{R}^d}^{}}  |v|^s G_n(t, v)  \, {\mathrm{d}{v}}.
            \end{aligned}
        \end{equation}
        The first term on the right-hand side above is controlled by
        \begin{equation} \label{control1}
            \begin{aligned}[b]
                ({\int_{\mathbb{R}^d}^{}}  |v|^2 {G_n(v)}  \, {\mathrm{d}{v}}) ({\int_{\mathbb{R}^d}^{}}  |v|^{s - 2} G_n(v)  \, {\mathrm{d}{v}})
                &\leq ({\int_{\mathbb{R}^d}^{}}  (1 + |v|^2) G_n(v)  \, {\mathrm{d}{v}})^2
                \leq C_T ({\int_{\mathbb{R}^d}^{}}  (1 + |v|^2) G_{n0}(v)  \, {\mathrm{d}{v}})^2 \\
                &\leq C_T ({\int_{\mathbb{R}^d}^{}}  (1 + |v|^2) G_0(v)  \, {\mathrm{d}{v}})^2.
            \end{aligned}
        \end{equation}
        Recalling \( |\tilde{\beta}| = |\beta + \zeta / 2| \leq C \|A\| \) and $\zeta$ defined in \eqref{Defzeta}, we choose \( z, A \) such that \( 1 - z \) and \( \|A\|  \) are sufficiently small with
        \begin{equation}\label{A-beta}
            \|A\| - \beta
            \leq \|A\| + |\tilde{\beta}| + \zeta / 2
            \leq C( \|A\| + (1 - z))
            < c / s.
        \end{equation}
        Collecting \eqref{sfn}, \eqref{control1} and \eqref{A-beta}, we have
        \begin{equation*}
            \frac{\mathrm{d}}{{\mathrm{d} t}} {\int_{\mathbb{R}^d}^{}}  |v|^s G_n(t, v)  \, {\mathrm{d}{v}}
            \leq C_T ({\int_{\mathbb{R}^d}^{}}  (1 + |v|^2) G_0  \, {\mathrm{d}{v}})^2
            - \delta {\int_{\mathbb{R}^d}^{}}  |v|^s G_n(t, v)  \, {\mathrm{d}{v}}
        \end{equation*}
        for some \( C_T, \delta > 0 \) and thus by Gronwall inequality \( \int_{\mathbb{R}^d} |v|^s G_n(t,v) \leq C_* \) if \( \int_{\mathbb{R}^d} |v|^s G_0(t) \leq C_* \) with \( C_* > 0 \) sufficiently large. {Passing to the limit}, we have \( \int_{\mathbb{R}^d} |v|^s G(t,v) \leq C_* \) on each finite \( t \in [0, T] \). This implies that for each \( t > 0 \), \( S(t) K \subseteq K \).

        By Schauder fixed point theorem, for each \( 0 < t < T < \infty \), there exists a fixed point \( G_t \in K \) such that \( S(t) G_t = G_t \). By the semigroup property of \( S(t) \), \( S(n t)G_t = G_t \) for all \( n \in \mathbb{Z}^+ \). Let \( t_n > 0 \) be a sequence that \( t_n \to 0 \) {when $n\rightarrow\infty$}. Since \( K \) is closed and compact, up to a subsequence, we have that \( G_{t_n} \to G \) {when $n\rightarrow\infty$} for some \( G \in K \).

        Let \( t > 0 \). Then there exists \( n_k \in \mathbb{Z}^+ \) such that \( n_k t_k \to t \) {when $k\rightarrow\infty$}. Therefore, by convergence of \( G_{t_k} \) and continuity of \( S \) in time, it holds that
        \begin{equation*}
            G_{t_k}
            = S(n_k t_k) G_{t_k}
            = (S(n_k t_k) - S(t)) G_{t_k} + S(t) G_{t_k}
            \to 0 + S(t) G, {\quad k\rightarrow\infty,}
        \end{equation*}
        which yields \( S(t)G = G \) for all \( t > 0 \). Hence, the proof of Theorem \ref{thm:existenceOfStatProfileWithSmallness} is finished.
    \end{proof}

    {The above theorem} gives us the existence of stationary Radon profile noted in Theorem \ref{thm:mainResult-statRadonProfile}. Note also that by the construction, \( \beta \) is chosen such that the temperature of stationary profile is a finite positive number. For uniform shear flow matrix, the discussion in Theorem \ref{thm.ss.comsign} implies the sign relation in Theorem \ref{thm:mainResult-signOfUSFBeta}.

    From the proof, we can also see that the profile has finite moment of order \( s \in (2, 4] \), given that \( A \) and \( 1 - z \) are sufficiently small.

    We remark that the smallness assumption posed on \( \|A\|\) and \( 1 - z \) are needed only in two places:
    \begin{itemize}
        \item the existence of stationary solution to the second moment equation \eqref{equ:secondMomentEquWithSelfSim},
        \item the control \( \|A\| - \beta < c / s \) for Gronwall inequality in \eqref{A-beta}.
    \end{itemize}

    \subsubsection{Sufficient condition to self-similar solutions for uniform shear flow}
    In \cite{JamesEtAl17SSP}, a variant of Povzner inequality is proposed to tackle the existence problem of self-similar profile for elastic USF. The same variant also works for inelastic collision and a more generic class of shearing matrix, assuming {that} additional conditions on the kernel and the parameters are satisfied.
    
    \begin{lemma} \label{thm:jamesPovznerIneq}
        Let \( \beta \in \mathbb{R} \) be chosen as in Theorem \ref{thm:existenceOfStatProfileWithSmallness}. Suppose the kernel \( b \) is continuous in the sense that
        \begin{equation*}
            (v, w)
            \mapsto {\int_{\mathbb{S}^{d-1}}^{}} 
                b\big(\sigma \cdot \frac{v - w}{|v - w|}\big) g(v, w, \sigma)
             \, {\mathrm{d}{\sigma}}
        \end{equation*}
        is continuous for all continuous functions \( g : \mathbb{R}^d \times \mathbb{R}^d \times \mathbb{S}^{d-1} \to \mathbb{R} \), \eqref{stableeq} holds for  \( A^\mathsf{T} \), and self-similar parameter \( \beta \) has a symmetric positive definite stationary solution \( \tilde{B} \in \mathbb{R}^{d \times d} \). If
        \begin{equation} \label{equ:jamesCondition}
            \max_{|v| = 1} \mathcal{W}(v; W) + \mathcal{H}(v) < 0,
        \end{equation}
        for some matrix \( W \in \mathbb{R}^{d \times d} \), where
        \begin{equation*}
            \mathcal{W}(v; W)
            \coloneqq {\int_{\mathbb{S}^{d-1}}^{}} 
                b\big(\sigma \cdot \frac{v}{|v|}\big) (W^+ + W^- - W)
             \, {\mathrm{d}{\sigma}} - A_\beta v \cdot { \nabla_{v} }W ,
        \end{equation*}
        and
        \begin{equation*}
            \mathcal{H}(v)
            \coloneqq {\int_{\mathbb{S}^{d-1}}^{}} 
                b\big(\sigma \cdot \frac{v}{|v|}\big) (\tilde{B}^+ \ln \tilde{B}^+ + \tilde{B}^- \ln \tilde{B}^- - \tilde{B} \ln \tilde{B})
             \, {\mathrm{d}{\sigma}} - (1 + \ln \tilde{B}) A_\beta v \cdot { \nabla_{v} }\tilde{B},
        \end{equation*}
        with matrices acting as the induced quadratic form \( W(v) = v^\mathsf{T} W v \) and
        $$
        v^- = \frac{z}{2} (v - |v| \sigma) , \quad v^+ = v - v^-,
        $$
        then there exists \( \epsilon_0 > 0 \) such that for all \( \epsilon \in (0, \epsilon_0) \), \( \varPhi(v) = (\tilde{B}(v) + \epsilon W(v))^{1 + \epsilon} \) is homogeneous of degree \( s = 2 (1 + \epsilon) > 2 \) and satisfies
        \begin{align*}
            U\varPhi(v, v_*)
            &\coloneqq {\int_{\mathbb{S}^{d-1}}^{}} 
                b\big(\sigma \cdot \frac{v - v_*}{|v - v_*|}\big) (\varPhi(v') + \varPhi(v_*') - \varPhi(v) - \varPhi(v_*))
             \, {\mathrm{d}{\sigma}} - A_\beta v \cdot { \nabla_{v} }\varPhi \\
            &\leq C |v|^{s / 2} |v_*|^{s / 2} - \kappa \varPhi ,
        \end{align*}
        for some \( C, \kappa > 0 \) {and} \( |v_*| \leq |v| \), and $ \varPhi(v) \leq C |v|^s.$
    \end{lemma}
    Note that we can always choose \( \epsilon \) sufficiently small such that \( s \in (2, 3) \).
    \begin{proof}
        Since \( \tilde{B} \) is positive definite, on sufficiently small \( \epsilon > 0 \) we have \( \min_{|v| = 1} \tilde{B}(v) + \epsilon W(v) > 0 \) and so \( \min_{|v| = 1} \varPhi(v) > 0 \). Using Taylor expansion on \( \varPhi(v) = (\tilde{B}(v) + \epsilon W(v))^{1 + \epsilon} \) gives \( \varPhi = \tilde{B} + \epsilon (\tilde{B} \ln \tilde{B} + W) + O(\epsilon^2) \) and \( { \nabla_{v} }\varPhi = { \nabla_{v} }\tilde{B} + \epsilon ((1 + \ln \tilde{B}) { \nabla_{v} }\tilde{B} + { \nabla_{v} }W) + O(\epsilon^2) \). Then with \( |v| = 1 \), it holds that
        \begin{align*}
            V\varPhi
            \coloneqq U\varPhi(v, 0)
            &= {\int_{\mathbb{S}^{d-1}}^{}} b(\sigma \cdot \frac{v}{\left|{v}\right|}) (\varPhi^+ + \varPhi^- - \varPhi) \, {\mathrm{d}{\sigma}} - A_\beta v \cdot { \nabla_{v} }\varPhi \\
            &= {\int_{\mathbb{S}^{d-1}}^{}} 
                b(\sigma \cdot \frac{v}{\left|{v}\right|}) (\tilde{B}^+ + \tilde{B}^- - \tilde{B})
             \, {\mathrm{d}{\sigma}} - A_\beta v \cdot { \nabla_{v} }\tilde{B} \\
            &\quad\;\,
            + \epsilon \bigg(
            {\int_{\mathbb{S}^{d-1}}^{}} 
                b(\sigma \cdot \frac{v}{\left|{v}\right|})
                (
                \tilde{B}^+ \ln \tilde{B}^+ + \tilde{B}^- \ln \tilde{B}^- - \tilde{B} \ln \tilde{B}
                )
             \, {\mathrm{d}{\sigma}}
            - (1 + \ln \tilde{B}) A_\beta v \cdot { \nabla_{v} }\tilde{B} \\
            &\quad\;\,\quad\;\,
            + {\int_{\mathbb{S}^{d-1}}^{}} 
                b(\sigma \cdot \frac{v}{\left|{v}\right|})
                (
                W^+ + W^- - W
                )
             \, {\mathrm{d}{\sigma}}
            - A_\beta v \cdot { \nabla_{v} }W
            \bigg)
            + O(\epsilon^2) \\
            &= \mathcal{W}(v; \tilde{B}) + \epsilon (\mathcal{H}(v) + \mathcal{W}(v; W)) + O(\epsilon^2).
        \end{align*}
        Since \( \tilde{B} \) is a stationary solution to \eqref{equ:inelasticBobylevSelfSim} on shear \( A^\mathsf{T} \),
        \begin{align*}
            &\quad\;\,
            \mathcal{W}(v; \tilde{B}) \\
            &= {\int_{\mathbb{S}^{d-1}}^{}} 
                b(\sigma \cdot \frac{v}{\left|{v}\right|}) (\tilde{B}^+ + \tilde{B}^- - \tilde{B})
             \, {\mathrm{d}{\sigma}} - A_\beta v \cdot { \nabla_{v} }\tilde{B} \\
            &= \tilde{B} : {\int_{\mathbb{S}^{d-1}}^{}} 
                b(\sigma \cdot \frac{v}{\left|{v}\right|}) ({{(v^+)}^{\otimes 2}} + {{(v^-)}^{\otimes 2}} - {{v}^{\otimes 2}})
             \, {\mathrm{d}{\sigma}} - A_\beta v \cdot { \nabla_{v} }\tilde{B} \\
            &= \tilde{B} : \frac{1}{2}{\int_{\mathbb{S}^{d-1}}^{}} 
                b(\sigma \cdot \frac{v}{\left|{v}\right|}) (
                z (z - 2) {{v}^{\otimes 2}} + z^2 |v|^2 {{\sigma}^{\otimes 2}} + z (1 - z) |v| (v \otimes \sigma + \sigma \otimes v)
                )
             \, {\mathrm{d}{\sigma}} - A_\beta v \cdot (\tilde{B} + \tilde{B}^\mathsf{T})v \\
            &= \tilde{B} : \frac{1}{2} (
            z^2 d c_{11} (\frac{|v|^2}{d} I - {{v}^{\otimes 2}}) + 2 \zeta {{v}^{\otimes 2}}
            )
            - (A_\beta^\mathsf{T} \tilde{B} + \tilde{B} A_\beta) : {{v}^{\otimes 2}} \\
            &= -{{v}^{\otimes 2}} : (
            A^\mathsf{T} \tilde{B} + \tilde{B} A + (2 \beta + \zeta) \tilde{B} + z^2 \frac{d c_{11}}{2} (\tilde{B} - \frac{\mathrm{tr}(\tilde{B})}{d} I)
            )
            = 0
        \end{align*}
        with \( X : Y = \sum_{ij} X_{ij} Y_{ij} \). Furthermore, by {assumption \eqref{equ:jamesCondition}, taking} \( 0 < \epsilon \ll 1 \) sufficiently small, we have that \( \max_{|v| = 1} V\varPhi < 0 \).

        By the continuity of the kernel and \( \varPhi \), \( U\varPhi \) is continuous in \( v_* \), so on \( |v_*| \leq \delta \ll 1 \), we have \( \max_{|v| = 1} U\varPhi < 0 \) as well. Since \( \max_{|v| = 1} \varPhi > 0 \), we also have \( U\varPhi \leq -\kappa \varPhi \) for some \( \kappa > 0 \) sufficiently small.

        On \( \delta \leq |v_*| \leq |v| = 1 \), as \( |v'|, |v_*'| \leq |v| + |v_*| \leq 2 \), \( |\varPhi(v') + \varPhi(v_*') - \varPhi(v) - \varPhi(v_*)| \leq C \). Also, \( |{ \nabla_{v} }\varPhi| \leq (\tilde{B} + \epsilon W)^\epsilon (|{ \nabla_{v} }\tilde{B}| + \epsilon |{ \nabla_{v} }W|) \leq C  \) on \( |v| = 1 \). These imply that
        \begin{equation*}
            U\varPhi
            \leq {\int_{\mathbb{S}^{d-1}}^{}} 
                b\big(\sigma \cdot \frac{v - v_*}{|v - v_*|}\big)
                |\varPhi(v') + \varPhi(v_*') - \varPhi(v) - \varPhi(v_*)|
             \, {\mathrm{d}{\sigma}} + \|A_\beta\| |v| |{ \nabla_{v} }\varPhi|
            \leq C\delta^{s / 2}.
        \end{equation*}

        Combining the two cases and using the homogeneity of \( \varPhi \), we have that \( U\varPhi \leq C |v|^{s / 2} |v_*|^{s / 2} - \kappa \varPhi \) for some \( C, \kappa > 0 \) on \( |v_*| \leq |v| \). Hence, \( \varPhi \leq C|v|^s \) follows directly from the homogeneity of \( \varPhi \).
    \end{proof}

%

     {There are few results on the existence of matrices satisfying \eqref{equ:jamesCondition}. However, we note that,} assuming in Lemma \ref{thm:jamesPovznerIneq}, such \( \tilde{B} \) exists, then
    \begin{itemize}
        \item
        \( \mathcal{W}(v; W) \) is linear in \( W \)

        \item
        as \( |v^\pm| \leq C |v| \), \( \tilde{B}, \tilde{B}^\pm \leq C \). Since \( x \mapsto x \ln x \) is locally bounded, the integrand in \( \mathcal{H} \) is bounded. Furthermore, \( \tilde{B} \) is positive definite, so \( \min_{|v| = 1} \ln\tilde{B}(v) > 0 \). This implies that \( \max_{|v| = 1} \mathcal{H} < \infty \).
    \end{itemize}
    Hence, it would suffice to find \( W \) such that \( \max_{|v| = 1} \mathcal{W}(v; W) < 0 \), with \( \lambda > 0 \) being sufficiently large,
    \begin{equation*}
        \max_{|v| = 1} \mathcal{H}(v) + \mathcal{W}(v; \lambda W) \leq \max_{|v| = 1} \mathcal{H}(v) + \lambda \max_{|v| = 1}\mathcal{W}(v; W) < 0,
    \end{equation*}
    and thus \eqref{equ:jamesCondition} holds. As seen in the proof of Theorem \ref{thm:jamesPovznerIneq}, one has
    \begin{equation*}
        \mathcal{W}(v; W) = -{{v}^{\otimes 2}} : \big(
        A^\mathsf{T} W + W A + (2 \beta + \zeta) W
        + z^2 \frac{d c_{11}}{2} (W - \frac{\mathrm{tr}(W)}{d} I)
        \big).
    \end{equation*}
    Then \( \max_{|v| = 1} \mathcal{W}(v; W) < 0 \) is equivalent to
    \begin{equation} \label{equ:jamesConditionWeakForm}
        A^\mathsf{T} W + W A + (2 \beta + \zeta) W
        + z^2 \frac{d c_{11}}{2} (W - \frac{\mathrm{tr}(W)}{d} I)
        \succ 0,
    \end{equation}
    with \( \succ \) defined in the sense that $A\succ 0$ if and only if $v^\mathsf{T} A v>0$ for all $v\in\mathbb{R}^d$.

    So to use Theorem \ref{thm:jamesPovznerIneq}, it suffices to show that such \( W \) exists. For consistency with notations in \eqref{equ:usfParaRelation}, on \( \tilde{\beta} = \beta + \zeta / 2 \) and \( \tilde{c} = z^2 \frac{d c_{11}}{2} \), we rewrite \eqref{equ:jamesConditionWeakForm} as
    \begin{equation} \label{equ:jamesConditionWeakFormSimplified}
        R(W) = A_{\tilde{\beta}}^\mathsf{T} W + W A_{\tilde{\beta}} + \tilde{c} (W - \frac{\mathrm{tr}(W)}{d} I) \succ 0.
    \end{equation}

    In what follows we would point out that, while the approach above may yield desired results on some shearing matrices, it fails in case of the 2D USF.

    \begin{theorem}
        If \( d = 2 \) and \eqref{equ:usfParaRelation} holds, then no \( W \in \mathbb{R}^{2 \times 2} \) satisfies \eqref{equ:jamesConditionWeakFormSimplified} for uniform shear flow where \( A = \alpha E_{12} \).
    \end{theorem}
    \begin{proof}
        On \( x = (\cos\theta, \sin\theta)^\mathsf{T} \in S^1 \) where $S^1$ denotes the one dimensional sphere,
        \begin{equation*}
            x^\mathsf{T} R x = C + C_1 \cos(2 \theta) + C_2 \sin(2 \theta)
        \end{equation*}
        has minimum value \( \min_{x \in S^1} x^\mathsf{T} R x = C - \sqrt{C_1^2 + C_2^2} \) with
        \begin{align*}
            C &= (\tilde{\beta}, \frac{\alpha}{2}, \frac{\alpha}{2}, \tilde{\beta}) \cdot \vec{W}, \\
            C_1 &= (\tilde{\beta} + \frac{\tilde{c}}{2}, -\frac{\alpha}{2}, -\frac{\alpha}{2}, \tilde{\beta} + \frac{\tilde{c}}{2}) \cdot \vec{W}, \\
            C_2 &= (\alpha, \tilde{\beta} + \frac{\tilde{c}}{2}, \tilde{\beta} + \frac{\tilde{c}}{2}, 0) \cdot \vec{W},
        \end{align*}
        and \( \vec{W} = (W_{11}, W_{12}, W_{21}, W_{22})^\mathsf{T} \). To have \( \min_{x \in S^1} x^\mathsf{T} R x > 0 \), it is equivalent to require
        \begin{align*}
            C &> 0, \\
            C^2 - C_1^2 - C_2^2 &= \vec{W}^\mathsf{T} S \vec{W} > 0,
        \end{align*}
        to hold for some \( \vec{W} \) with

        \begin{equation*}
            S = \begin{pmatrix}
                -\alpha^2 - \frac{1}{4} \tilde{c} (4 \tilde{\beta} + \tilde{c})         & -\frac{1}{4} \alpha \tilde{c}                & -\frac{1}{4} \alpha \tilde{c}                & \frac{1}{4} \tilde{c} (4 \tilde{\beta} + \tilde{c})  \\
                -\frac{1}{4} \alpha \tilde{c}                                           & -\frac{1}{4} (2 \tilde{\beta} + \tilde{c})^2 & -\frac{1}{4} (2 \tilde{\beta} + \tilde{c})^2 & -\frac{1}{4} \alpha \tilde{c}                        \\
                -\frac{1}{4} \alpha \tilde{c}                                           & -\frac{1}{4} (2 \tilde{\beta} + \tilde{c})^2 & -\frac{1}{4} (2 \tilde{\beta} + \tilde{c})^2 & -\frac{1}{4} \alpha \tilde{c}                        \\
                2 \tilde{\beta}^2 + \frac{1}{4} \tilde{c} (4 \tilde{\beta} + \tilde{c}) & -\frac{1}{4} \alpha \tilde{c}                & -\frac{1}{4} \alpha \tilde{c}                & -\frac{1}{4} \tilde{c} (4 \tilde{\beta} + \tilde{c})
            \end{pmatrix}
        \end{equation*}
        and such \( \vec{W} \) exists if and only if \( -S \) is not positive semidefinite. However, we can enumerate all \( 15 \) principal minors of \( -S \) and note that
        \begin{align*}
            &
            M_1 = \alpha^2 + \frac{1}{4} \tilde{c} (4 \tilde{\beta} + \tilde{c})
            ,\quad
            M_2 = M_3 = \frac{1}{4} (2 \tilde{\beta} + \tilde{c})^2
            ,\quad
            M_4 = \frac{1}{4} \tilde{c} (4 \tilde{\beta} + \tilde{c}), \\
            &
            M_{12} = M_{13} = \frac{1}{16 \tilde{c}} (4 \tilde{\beta} + \tilde{c})^2 (2 \tilde{\beta} + \tilde{c})^3
            ,\quad
            M_{14} = \frac{1}{2} \tilde{\beta} (2 \tilde{\beta} + \tilde{c})^3
            ,\quad
            M_{23} = 0
            ,\quad
            M_{24} = M_{34} = \frac{1}{16} \tilde{c} (2 \tilde{\beta} + \tilde{c})^3, \\
            &
            M_{123} = M_{124} = M_{134} = M_{234} = 0
        \end{align*}
        are all nonnegative with $\det(-S) = 0$. Then \( -S \) is positive semidefinite, no such \( \vec{W} \) exists.
    \end{proof}

    This implies that, in order to have \eqref{equ:jamesCondition} on 2D USF, it is necessary to study the behavior of the more complicated part \( \mathcal{H} \) of the condition. If \( \mathcal{H}(v) \geq 0 \) on \( |v| = 1 \), \eqref{equ:jamesCondition} fails to hold.
    
 \medskip
{In the current section, following the same strategy as in \cite{JamesEtAl17SSP}, we have constructed the global in time solutions as nonnegative measures for the Cauchy problem \eqref{equ:inelasticBoltzmann} in Theorem \ref{thm:existenceOfMildSol}, and further obtained the propagation of moment bounds in Theorem \ref{thm:mildSolAsWeakSolAndProperties}. Moreover, we have provided a characterization of the large time behavior of temperature in case of the simple uniform  shear flow in Theorem \ref{thm.ss.comsign} and established the existence of stationary self-similar Radon profile in Theorem \ref{thm:existenceOfStatProfileWithSmallness}. In the next section, we will follow the approach in \cite{BobylevEtAl20SSA} for further obtaining the large time asymptotics toward the self-similar profile for global in time solutions of \eqref{equ:inelasticBoltzmann}.}

    \section{Fourier solution} \label{sec:fourierSol}
    In this part we consider the solutions in the Fourier space as characteristic functions.
    The natural solution space is the space \( \mathcal{K} \) of characteristic functions, the Fourier transform of probability measures on \( \mathbb{R}^d \). For two characteristic functions \( \varphi, \psi \in \mathcal{K} \), their \( p \)-Toscani distance for \( p > 0 \) is
    \begin{equation}
        \label{def.pTd}
        \|\varphi - \psi\|_{p} = \sup_{k \neq 0} \frac{|\varphi(k) - \psi(k)|}{|k|^p}.
    \end{equation}
    The Cannone--Karch space \( \mathcal{K}^{p} \subseteq \mathcal{K} \) {with} \( p \in (0, 2] \), the space of characteristic functions with finite \( p \)-Toscani distance with \( 1 \), that is \( \|\varphi - 1\|_{p} < \infty \),
    contains all characteristic functions of probability measures with finite absolute \( p \)-moment and (if \( p > 1 \)) zero mean \cite{CannoneKarch09IES}.

    Similar to the elastic case in \cite{BobylevEtAl20SSA}, we have the following lemma for properties of the operator $\widehat{Q^+_e}$.
    
    \begin{lemma}
        Let $\widehat{Q^+_e}$ be defined in \eqref{DefhatQ}, then \( b_0^{-1} \widehat{Q_e^+}(\mathcal{K},\mathcal{K}) \subseteq \mathcal{K} \), and the \( \mathcal{L}_e \)-Lipschitz property of \( \widehat{Q_e^+} \) holds:
        \begin{align}\label{LeLip}
            |\widehat{Q^+_e}(\varphi,\varphi)(k) - \widehat{Q_e^+}(\psi,\psi)(k)|\leq \mathcal{L}_e|\varphi(k) - \psi(k)|\leq 2 b_0 \|\varphi - \psi\|_{{\mathit{L}^{\infty}}},
        \end{align}
        {where $\mathcal{L}_e$ is defined in \eqref{DefLe}.}
    \end{lemma}
    
    \begin{proof}
        The proof is very similar to \cite[Lem.~3.1]{BobylevEtAl20SSA}, see also \cite{BobylevEtAl09SSA}{. For} brevity, we only prove the \( \mathcal{L}_e \)-Lipschitz property as follows:
        \begin{align*}
            |\widehat{Q^+_e}(\varphi,\varphi) - \widehat{Q_e^+}(\psi,\psi)|
            &\leq {\int_{\mathbb{S}^{d-1}}^{}} 
                b(\cos\vartheta) |\varphi(k^+) \varphi(k^-) - \psi(k^+) \psi(k^-)|
             \, {\mathrm{d}{\sigma}} \\
            &\leq {\int_{\mathbb{S}^{d-1}}^{}} b(\cos\vartheta) |\varphi(k^+)||\varphi(k^-) - \psi(k^-)| \, {\mathrm{d}{\sigma}}
            + {\int_{\mathbb{S}^{d-1}}^{}} b(\cos\vartheta) |\varphi(k^+) - \psi(k^+)||\varphi(k^-)| \, {\mathrm{d}{\sigma}} \\
            &\leq {\int_{\mathbb{S}^{d-1}}^{}} b(\cos\vartheta)  |\varphi(k^-) - \psi(k^-)| \, {\mathrm{d}{\sigma}}
            + {\int_{\mathbb{S}^{d-1}}^{}} b(\cos\vartheta) |\varphi(k^+) - \psi(k^+)| \, {\mathrm{d}{\sigma}} \\
            &=\mathcal{L}_e|\varphi - \psi|
            \leq 2 b_0 \|\varphi - \psi\|_{{\mathit{L}^{\infty}}},
        \end{align*}
        for \( \varphi, \psi \in \mathcal{K} \). {In the third inequality above, we use the fact that $|\phi(k)| \leq 1$ for any $k$.}
    \end{proof}
    
    Then we integrate the equation \eqref{equ:inelasticBobylev} to get the mild form of the solution defined as follows.
    
    \begin{definition}
        A function \( \varphi(t,k) \in \mathit{C}([0, \infty), \mathcal{K}) \) is a mild solution of \eqref{equ:inelasticBobylev} on initial condition \( \varphi_0(k) = \varphi(0,k) \in \mathcal{K} \) if
        \begin{equation} \label{equ:inelasticBobylev_mildForm}
            \varphi(t,k)
            = e^{-b_0 t} \varphi_0(e^{-t A^\mathsf{T}} k)
            + {\int_{0}^{t}} e^{-b_0 (t - \tau)} \widehat{Q_e^+}(\varphi,\varphi)(\tau, e^{-(t - \tau) A^\mathsf{T}} k) \, {\mathrm{d}{\tau}}.
        \end{equation}
    \end{definition}
    
    We can prove the existence of Cauchy problem solution of \eqref{equ:inelasticBobylev} in terms of mild solution of \eqref{equ:inelasticBobylev} in the frequency space.
    
    \begin{theorem}
        If \( \varphi_0(k) \in \mathcal{K} \), then there exists a unique mild solution \( \varphi(t,k) \in \mathit{C}([0, \infty), \mathcal{K}) \) {globally in time and} satisfying \eqref{equ:inelasticBobylev_mildForm} {with} initial condition \( \varphi(t,k)=\varphi_0(k) \).
    \end{theorem}
    
    \begin{proof}
    Consider the Picard map \( P : \mathcal{X}_T \to \mathcal{X}_T \) on \( \mathcal{X}_T \coloneqq \mathit{C}([0, T], \mathcal{K}) \) equipped with metric \( d(\varphi, \psi) = \sup_{[0, T]} \|\varphi(t) - \psi(t)\|_{{\mathit{L}^{\infty}}} \)  defined by
        \begin{equation*}
            P\varphi(t,k)
            = e^{-b_0 t} \varphi_0(e^{-t A^\mathsf{T}} k) + {\int_{0}^{t}} 
                e^{-b_0(t - \tau)} \widehat{Q_e^+}(\varphi,\varphi)(\tau, e^{-(t - \tau) A^\mathsf{T}} k)
             \, {\mathrm{d}{\tau}}.
        \end{equation*}
        {We first show that $P$ maps $\mathcal{X}_T$ into itself.} For \( \varphi_0, \varphi(\tau) \in \mathcal{K} \), it is straight-forward to see that \( \varphi_0(e^{-t A^\mathsf{T}} k) \) and \( \widehat{Q_e^+}(\varphi,\varphi)(\tau, e^{-(t - \tau) A^\mathsf{T}} k) \) are characteristic functions. \( P\varphi(t) \), which is a convex combination of characteristic functions, is also a characteristic function for each \( t \in [0, T] \). Furthermore, {since \( \mathcal{K} \) is a subset of all continuous functions,} \( t \mapsto e^{-b_0 t} \varphi_0(e^{-t A^\mathsf{T}} k) \) and \( (t, \tau) \mapsto e^{-b_0(t - \tau)} \widehat{Q_e^+}\varphi(\tau, e^{-(t - \tau) A^\mathsf{T}} k) \) are continuous, then \( t \mapsto P\varphi(t) \) is continuous. These imply that \( P \) is well-defined.

        For \( \varphi, \psi \in \mathcal{X}_T \) and \( t \in [0, T] \), it holds that
        \begin{align*}
            |P\varphi(t, k) - P\psi(t, k)|
            &\leq {\int_{0}^{t}} 
                e^{-b_0 (t - \tau)} |\widehat{Q_e^+}(\varphi,\varphi)(\tau) - \widehat{Q_e^+}(\psi,\psi)(\tau)|(e^{-(t - \tau) A^\mathsf{T}} k)
             \, {\mathrm{d}{\tau}} \\
            &\leq {\int_{0}^{t}} 
                2 b_0 \|\varphi(\tau) - \psi(\tau)\|_{{\mathit{L}^{\infty}}}
             \, {\mathrm{d}{\tau}} \\
            &\leq 2 b_0 T{\, d(\varphi,\psi)}.
        \end{align*}
        We choose \( T = \frac{1}{4 b_0} \), which is independent of \( \varphi_0 \). Hence, \( P \) is a Banach contraction, there exists \( \varphi \in \mathcal{X}_T \) such that \( P\varphi = \varphi \). By a standard extension argument, there exists \( \varphi \in \mathit{C}([0, \infty), \mathcal{K}) \) such that \( P\varphi = \varphi \).
    \end{proof}
    We show that the distance of two mild solutions \( \varphi, \psi \) for the elastic equation on initial conditions \( \varphi_0, \psi_0 \in \mathcal{K} \) is controlled by a corresponding solution on {the linearized equation \eqref{equ:inelasticBobylev}}. The proof is similar to that in  \cite[Thm.~4.8]{BobylevEtAl20SSA}, noting that the key points are using the bound \( |\widehat{Q_e^+}(\varphi,\varphi) - \widehat{Q_e^+}(\psi,\psi)| \leq \mathcal{L}_e|\varphi - \psi| \) and {the} comparison principle.

    \begin{theorem} \label{thm:controlOfBobMildSolDiffViaLinearizedMildSol}
        Let \( \varphi(t,k), \psi(t,k) \) be two mild solutions of \eqref{equ:inelasticBobylev} {with} initial conditions \( \varphi_0(k), \psi_0(k) \in \mathcal{K} \). Then if \( y(t,k) \) satisfies
        \begin{equation} \label{equ:inelasticBobylev_linearizedMildSol}
            y(t,k)
            = e^{-b_0 t} y_0(e^{-t A^\mathsf{T}} k)
            + {\int_{0}^{t}} e^{-b_0 (t - \tau)} \mathcal{L}_e y(\tau, e^{-(t - \tau) A^\mathsf{T}} k) \, {\mathrm{d}{\tau}},
        \end{equation}
        which is the mild solution of the linearized equation
        \begin{equation*}
            {\partial_{t}} y + A^\mathsf{T} k \cdot { \nabla_{k} } y = \mathcal{L}_e y - b_0 y,
        \end{equation*}
        with initial condition \( y_0(k) = |\varphi_0(k) - \psi_0(k)| \), then it holds that
        \begin{align}\label{phi-psi}
            |\varphi(t,k) - \psi(t,k)| \leq y(t,k).
        \end{align}
    \end{theorem}
    
    \begin{proof}
        We have from \eqref{equ:inelasticBobylev_mildForm} and \eqref{LeLip} that for \( z(t,k) = |\varphi(t,k) - \psi(t,k)| \),
        \begin{align*}
            z(t,k)
            &= |\varphi(t,k) - \psi(t,k)| \\
            &\leq e^{-b_0 t} |\varphi_0(e^{-t A^\mathsf{T}} k) - \psi_0(e^{-t A^\mathsf{T}} k)|
            + {\int_{0}^{t}} 
                e^{-b_0 (t - \tau)}
                |\widehat{Q_e^+}(\varphi,\varphi) - \widehat{Q_e^+}(\psi,\psi)(\tau, e^{-(t - \tau) A^\mathsf{T}} k)
                | \, {\mathrm{d}{\tau}} \\
            &\leq e^{-b_0 t}
            y_0(e^{-t A^\mathsf{T}} k)
            + {\int_{0}^{t}} 
                e^{-b_0 (t - \tau)}
                \mathcal{L}_e|\varphi - \psi|(\tau, e^{-(t - \tau) A^\mathsf{T}} k)
             \, {\mathrm{d}{\tau}} \\
            &= e^{-b_0 t} y_0(e^{-t A^\mathsf{T}} k) + {\int_{0}^{t}} 
                e^{-b_0 (t - \tau)}
                \mathcal{L}_e z(\tau, e^{-(t - \tau) A^\mathsf{T}} k)
             \, {\mathrm{d}{\tau}}.
        \end{align*}
        Therefore, \eqref{phi-psi} holds by comparison principle.
    \end{proof}

    In order to control the growth of {the} \( p\)-Toscani distance in time, we need to define the kernel constant
    \begin{equation} \label{Deflap}
        \begin{aligned}[b]
            \lambda_p
            &= {\int_{\mathbb{S}^{d-1}}^{}} 
                b\big(\sigma \cdot \frac{k}{|k|}\big) \big(
                1 - \frac{|k^+|^p + |k^-|^p}{|k|^p}
                \big)
             \, {\mathrm{d}{\sigma}} \\
            &= {\int_{\mathbb{S}^{d-1}}^{}} 
                b(\cos\vartheta) \Big(
                1 - (z^2 \sin^2\frac{\vartheta}{2})^{p / 2} - (1 - z (2 - z) \sin^2\frac{\vartheta}{2})^{p / 2}
                \Big)
             \, {\mathrm{d}{\sigma}},
        \end{aligned}
    \end{equation}
    for \( p \geq 0 \) as in \cite{BobylevGamba06BEM, BobylevEtAl20SSA}. Note that \( \lambda_0 = -b_0 < 0 \), \( \lambda_2 = \zeta \geq 0 \),
    and \( \lambda_p \leq b_0 \) for all \( p \). Furthermore, a direct calculation shows that
    \begin{equation}\label{monola}
        \frac{\partial \lambda_p}{{\partial p}}
        = {\int_{\mathbb{S}^{d-1}}^{}} 
            b(\cos\vartheta) (
            -(z^2 \sin^2\frac{\vartheta}{2})^{p / 2} \ln(z^2 \sin^2\frac{\vartheta}{2})
            - (1 - z (2 - z) \sin^2\frac{\vartheta}{2})^{p / 2} \ln(1 - z (2 - z) \sin^2\frac{\vartheta}{2})
            )
         \, {\mathrm{d}{\sigma}}
        > 0
    \end{equation}
    by the facts that \( z^2 \sin^2\frac{\vartheta}{2} \leq 1 \) and \( 1 - z (2 - z) \sin^2\frac{\vartheta}{2} \leq 1 \), which implies that \( \lambda_p \) is strictly increasing in \( p \). This gives that the unique root \( p_0 > 0 \) that \( \lambda_{p_0} = 0 \) is strictly smaller than \( 2 \) {in the case of inelastic} collision \( z < 1 \), and \( \lambda_p > 0 \) {with} \( p > 2 \) for all \( z \in (1 / 2, 1] \).

    Now we can justify the existence and show some properties of such \( y(t) \) for appropriate initial condition \( y_0 \), with the framework developed in \cite{BobylevEtAl20SSA} in {the} elastic case.

    \begin{theorem} \label{thm:existenceAndPropOfLinearizedInelasticBob}
        Let \( \mathcal{C}_p = \left\{f \in \mathit{C}(\mathbb{R}^d, \mathbb{R}) |\ \|f\|_p<\infty  \right\} \) equipped with \( p \)-Toscani norm for some \( p > 0 \), it holds that
        \begin{itemize}
            \item
            {For every \( y_0(k) \in \mathcal{C}_p \), there exists a unique $y(t,k)$ to \eqref{equ:inelasticBobylev_linearizedMildSol}} {denoted by}
            \[ y(t,k) = \exp(t (-b_0 + \mathcal{L}_e - A^\mathsf{T} k \cdot { \nabla_{k} }))  y_0 \in \mathit{C}([0, \infty), \mathcal{C}_p).
            \]

            \item
            If \( y_0(k) \in \mathcal{C}_p \), \( u(t,k) \in \mathit{C}([0, \infty), \mathcal{C}_p) \) are both nonnegative, and
            \begin{equation}\label{u>y0}
                u(t,k)
                \geq e^{-b_0 t} y_0(e^{-t A^\mathsf{T}} k)
                + {\int_{0}^{t}} e^{-b_0 (t - \tau)} \mathcal{L}_e u(\tau, e^{-(t - \tau) A^\mathsf{T}} k) \, {\mathrm{d}{\tau}},
            \end{equation}
            then one has
            \begin{align}\label{comparey}
                u(t,k) \geq y(t,k)= \exp(t (-b_0 + \mathcal{L}_e - A^\mathsf{T} k \cdot { \nabla_{k} }))  y_0(k).
            \end{align}

            \item
            If \( u_0=u_0(k), v_0=v_0(k) \in \mathcal{C}_p \) and \( 0 \leq u_0 \leq v_0 \), then it holds
            \begin{align}\label{compareuv}0 \leq \exp(t (-b_0 + \mathcal{L}_e - A^\mathsf{T} k \cdot { \nabla_{k} }))  u_0 \leq \exp(t (-b_0 + \mathcal{L}_e - A^\mathsf{T} k \cdot { \nabla_{k} }))  v_0. \end{align}
        \end{itemize}
    \end{theorem}
    
    \begin{proof}
        Consider first the Picard iteration
        \begin{align*}
            z_0(t, k) &= 0, \\
            z_{n + 1}(t, k)
            &= e^{-b_0 t} y_0(e^{-t A^\mathsf{T}} k)
            + {\int_{0}^{t}} e^{-b_0 (t - \tau)} \mathcal{L}_e z_n(\tau, e^{-(t - \tau) A^\mathsf{T}} k) \, {\mathrm{d}{\tau}},\quad n=0,1,\cdots,
        \end{align*}
        with \( t \in [0, T] \) for some \( T > 0 \) to be determined later. Then one gets
        \begin{align*}
            \frac{|z_{n + 1}(t, k)|}{|k|^p}
            &\leq e^{-b_0 t} \|e^{-t A^\mathsf{T}}\|^p \|y_0\|_{p}
            + {\int_{0}^{t}} 
                e^{-b_0 (t - \tau)}
                \|e^{-(t - \tau) A^\mathsf{T}}\|^p
                \frac{|\mathcal{L}_e z_n(\tau,k)|}{|k|^p}
             \, {\mathrm{d}{\tau}} \\
            &\leq \|y_0\|_{p} \sup_{[0, T]} \|e^{-\tau A^\mathsf{T}}\|^p
            + \sup_{[0, T]} \|e^{-\tau A^\mathsf{T}}\|^p
            {\int_{0}^{t}} 
                e^{-b_0 (t - \tau)}
                \|z_n(\tau)\|_{p} {\int_{\mathbb{S}^{d-1}}^{}} 
                    b(\cos\vartheta)
                    \frac{|k^+|^p +|k^-|^p}{|k|^p}
                 \, {\mathrm{d}{\sigma}}
            {\mathrm{d}{\tau}} \\
            &\leq \|y_0\|_{p} \sup_{[0, T]} \|e^{-\tau A^\mathsf{T}}\|^p
            + T \sup_{[0, T]} \|e^{-\tau A^\mathsf{T}}\|^p
            (b_0 - \lambda_p)
            \sup_{[0, T]} \|z_n(\tau)\|_{p},
        \end{align*}
        which yields
        \begin{align*}
            \sup_{[0, T]} \|z_{n + 1}(\tau)\|_{p}
            \leq \|y_0\|_{p} \sup_{[0, T]} \|e^{-\tau A^\mathsf{T}}\|^p
            + (b_0 - \lambda_p) T \sup_{[0, T]} \|e^{-\tau A^\mathsf{T}}\|^p
            \sup_{[0, T]} \|z_n(\tau)\|_{p}.
        \end{align*}
        By induction, \( z_n \in \mathit{C}([0, T], \mathcal{C}_p) \) for each \( n>0 \) and \( T > 0 \).

        Furthermore, direct calculations show that \(
        \sup_{[0, T]} \|z_1(t) - z_0(t)\|_{p}
        = \sup_{[0, T]} \|z_1(t)\|_{p}
        < \infty
        \)
        and
        \begin{align*}
            \|z_{n + 1}(t) - z_n(t)\|_{p}
            &\leq {\int_{0}^{t}} 
                e^{-b_0 (t - \tau)}
                \|(\mathcal{L}_e z_n - \mathcal{L}_e z_{n - 1})(\tau, e^{-(t - \tau) A^\mathsf{T}} k)\|_{p}
             \, {\mathrm{d}{\tau}} \\
            &\leq {\int_{0}^{t}} 
                e^{-b_0 (t - \tau)}
                \|e^{-(t - \tau) A^\mathsf{T}}\|^p \cdot \|\mathcal{L}_e (z_n(\tau) - z_{n - 1}(\tau))\|_{p}
             \, {\mathrm{d}{\tau}} \\
            &\leq \sup_{[0, T]} \|e^{-\tau A^\mathsf{T}}\|^p {\int_{0}^{t}} 
                e^{-b_0 (t - \tau)}
                {\int_{\mathbb{S}^{d-1}}^{}} 
                    b(\cos\vartheta)
                    \|z_n(\tau) - z_{n - 1}(\tau)\|_{p}
                    \frac{|k^+|^p + |k^-|^p}{|k|^p}
                 \, {\mathrm{d}{\sigma}}
            {\mathrm{d}{\tau}} \\
            &\leq (b_0 - \lambda_p) T e^{p T \|A\|}
            \sup_{[0, T]} \|z_n(\tau) - z_{n - 1}(\tau)\|_{p}.
        \end{align*}
        We choose $T>0$ {small enough such that} \( (b_0 - \lambda_p) T e^{p T \|A\|} < 1 \). This implies that \( z_n \to z \) {as $\rightarrow\infty$} in \( \sup_{[0, T]} \|\cdot\|_{p} \)-norm for some \( z \in \mathit{C}([0, T], \mathcal{C}_p) \) satisfying  \eqref{equ:inelasticBobylev_linearizedMildSol}. As the selection of \( T \) is independent {from} \( y_0 \), by standard extension argument we obtain {that there exists} some \( y \in \mathit{C}([0, \infty), \mathcal{C}_p) \) that solves \eqref{equ:inelasticBobylev_linearizedMildSol}.

        To prove the uniqueness {of the solution obtained in Theorem \ref{thm:existenceAndPropOfLinearizedInelasticBob}}, for two solutions \( \varphi_1(t,k), \varphi_2(t,k) \) of \eqref{equ:inelasticBobylev_linearizedMildSol} with the same initial condition, we have
        \begin{align*}
            \frac{|\varphi_1(t,k) - \varphi_2(t,k)|}{|k|^p}
            &\leq {\int_{0}^{t}} 
                e^{-b_0 (t - \tau)}
                |k|^{-p}
                |\mathcal{L}_e (\varphi_1 - \varphi_2)(\tau, e^{-(t - \tau) A^\mathsf{T}} k)|
             \, {\mathrm{d}{\tau}} \\
            &\leq (b_0 - \lambda_p) {\int_{0}^{t}} 
                e^{-b_0 (t - \tau)}
                \|e^{-(t - \tau) A^\mathsf{T}}\|^p
                \|\varphi_1(\tau) - \varphi_2(\tau)\|_{p}
             \, {\mathrm{d}{\tau}} \\
            &\leq (b_0 - \lambda_p) {\int_{0}^{t}} 
                e^{-b_0 (t - \tau)}
                e^{(t - \tau) p \|A\|}
                \|\varphi_1(\tau) - \varphi_2(\tau)\|_{p}
             \, {\mathrm{d}{\tau}},
        \end{align*}
        which gives
        $$
        e^{b_0 t} e^{-t p \|A\|} \|\varphi_1(t) - \varphi_2(t)\|_{p}
        \leq (b_0 - \lambda_p) {\int_{0}^{t}} 
            e^{b_0 \tau} e^{-\tau p \|A\|} \|\varphi_1(\tau) - \varphi_2(\tau)\|_{p}
         \, {\mathrm{d}{\tau}}.
        $$
        Then it follows from Gronwall inequality that \( \|\varphi_1(t) - \varphi_2(t)\|_{p} = 0 \) for all \( t \).

        To show the second part of the proposition, we note that \( z_0(t,k) = 0 \leq u(t,k) \) and
        \begin{align*}
            z_{n + 1}(t,k)
            &= e^{-b_0 t} y_0(e^{-t A^\mathsf{T}} k)
            + {\int_{0}^{t}} 
                e^{-b_0 (t - \tau)} \mathcal{L}_e z_n(\tau, e^{-(t - \tau) A^\mathsf{T}} k)
             \, {\mathrm{d}{\tau}} \\
            &\leq u(t,k)
            - {\int_{0}^{t}} 
                e^{-b_0 (t - \tau)} \mathcal{L}_e (u - z_n)(\tau, e^{-(t - \tau) A^\mathsf{T}} k)
             \, {\mathrm{d}{\tau}}
        \end{align*}
        by \eqref{u>y0} and the fact that \( \mathcal{L}_e \) is a positive operator in the sense that \( \mathcal{L}_e f \geq 0 \) if \( f \geq 0 \). Hence, by {an} induction argument one gets \( z_n(t,k) \leq u(t,k) \) for all \( n>0 \). Taking the limit {$n\rightarrow\infty$,} we have \eqref{comparey}.

        We turn to the third part. Noticing that with \( v(t,k) = \exp(t (-b_0 + \mathcal{L}_e - A^\mathsf{T} k \cdot { \nabla_{k} }))  v_0 \), \eqref{equ:inelasticBobylev_linearizedMildSol} and \( 0 \leq u_0 \leq v_0 \) imply that
        \begin{align*}
            v(t,k)
            &= e^{-b_0 t} v_0(e^{-t A^\mathsf{T}} k)
            + {\int_{0}^{t}} e^{-b_0 (t - \tau)} \mathcal{L}_e v(\tau, e^{-(t - \tau) A^\mathsf{T}} k) \, {\mathrm{d}{\tau}} \\
            &\geq e^{-b_0 t} u_0(e^{-t A^\mathsf{T}} k)
            + {\int_{0}^{t}} e^{-b_0 (t - \tau)} \mathcal{L}_e v(\tau, e^{-(t - \tau) A^\mathsf{T}} k) \, {\mathrm{d}{\tau}}.
        \end{align*}
        It follows from the above inequality and applying the conclusion of the second part that \( v(t,k) \geq \exp(t (-b_0 + \mathcal{L}_e - A^\mathsf{T} k \cdot { \nabla_{k} }))  u_0(k) \), which yields \eqref{compareuv}. Nonnegativity comes directly from the fact that \( z_n = 0 \) for all $n>0$ {in the case when} \( y_0 = 0 \). Hence, the proof of Theorem \ref{thm:existenceAndPropOfLinearizedInelasticBob} is complete.
    \end{proof}

    We have the following bound  of \( p\)-Toscani distance, which will be used in the next section when we consider the stationary profile.
    \begin{lemma} \label{thm:bobControlOfPolynomialDiffViaUp}
        For \( p \geq 0 \), define
        \begin{equation}\label{Defup}
            u_p(t, k) := |k|^p \exp(-t (\lambda_p - p \|A\|)).
        \end{equation}
        Then the solution \( \exp(t (-b_0 + \mathcal{L}_e - A^\mathsf{T} k \cdot { \nabla_{k} }))  |k|^p \) of \eqref{equ:inelasticBobylev_linearizedMildSol} with initial data \( |k|^p \) satisfies
        \begin{equation}\label{controlTos}
            \exp(t (-b_0 + \mathcal{L}_e - A^\mathsf{T} k \cdot { \nabla_{k} }))  |k|^p
            \leq u_p(t,k).
        \end{equation}
        In particular, for $A_\beta$ defined in \eqref{DefAbeta}, it holds that
        \begin{align}\label{controlTosbe}
            \exp(t (-b_0 + \mathcal{L}_e - A_\beta^\mathsf{T} k \cdot { \nabla_{k} })) |k|^p \leq e^{-p \beta t} u_p(t,k).
        \end{align}
    \end{lemma}
    \begin{proof}
        By direct computation, one gets
        \begin{align*}
            \mathcal{L}_e|k|^p
            &= {\int_{\mathbb{S}^{d-1}}^{}} 
                b(\cos\vartheta) (|k^+|^p + |k^-|^p)
             \, {\mathrm{d}{\sigma}} \\
            &= |k|^p
            {\int_{\mathbb{S}^{d-1}}^{}} 
                b(\cos\vartheta) \frac{|k^+|^p + |k^-|^p}{|k|^p}
             \, {\mathrm{d}{\sigma}} \\
            &= |k|^p (b_0 - \lambda_p),
        \end{align*}
        which yields \( \mathcal{L}_e u_p = (b_0 - \lambda_p) u_p \). Then we have
        \begin{align*}
            &\quad\;\,
            e^{-b_0 t} |e^{-t A^\mathsf{T}} k|^p
            + {\int_{0}^{t}} 
                e^{-b_0 (t - \tau)} \mathcal{L}_e u_p(\tau, e^{-(t - \tau) A^\mathsf{T}} k)
             \, {\mathrm{d}{\tau}} \\
            &\leq e^{-b_0 t} \|e^{-t A^\mathsf{T}}\|^p |k|^p
            + (b_0 - \lambda_p) {\int_{0}^{t}} 
                e^{-b_0 (t - \tau)}
                u_p(\tau, e^{-(t - \tau) A^\mathsf{T}} k)
             \, {\mathrm{d}{\tau}} \\
            &\leq e^{-b_0 t} \|e^{-t A^\mathsf{T}}\|^p |k|^p
            + (b_0 - \lambda_p) {\int_{0}^{t}} 
                e^{-b_0 (t - \tau)}
                e^{-\tau (\lambda_p - p \|A\|)} \|e^{-(t - \tau) A^\mathsf{T}}\|^p |k|^p
             \, {\mathrm{d}{\tau}} \\
            &\leq |k|^p (
            e^{-b_0 t} e^{p t \|A\|}
            + (b_0 - \lambda_p) {\int_{0}^{t}} 
                e^{-b_0 (t - \tau)}
                e^{-\tau (\lambda_p - p \|A\|)} e^{p (t - \tau) \|A\|}
             \, {\mathrm{d}{\tau}}
            ) \\
            &= |k|^p e^{-t (\lambda_p - p \|A\|)} = u_p,
        \end{align*}
        which, together with \eqref{comparey} in Theorem \ref{thm:existenceAndPropOfLinearizedInelasticBob} by substituting $y_0(k)=|k|^p$ and $u=u_p$, gives \eqref{controlTos}.
    \end{proof}

    With the help {of the results above}, we have the stability result on mild solution of \eqref{equ:inelasticBobylev}.

    \begin{theorem}
        Let \( \varphi(t,k), \psi(t,k) \) be mild solutions of \eqref{equ:inelasticBobylev} satisfying \eqref{equ:inelasticBobylev_mildForm} with initial conditions \( \varphi_0(k), \psi_0(k) \in \mathcal{K} \). Suppose \( |\varphi_0(k) - \psi_0(k)| \leq C |k|^p \) for some \( p > 0 \) and \( C > 0 \), then
        \begin{equation}\label{pstable.rj}
            |\varphi(t,k) - \psi(t,k)| \leq C u_p(t,k),
        \end{equation}
        where $u_p$ is defined in \eqref{Defup}.
        In particular,
        \begin{align}\label{stable}
            \|\varphi(t) - \psi(t)\|_{p} \leq C e^{-t (\lambda_p - p \|A\|)} \to 0 \quad \text{as} \ t\to \infty
        \end{align} if \( \lambda_p / p > \|A\|\).
    \end{theorem}

    \begin{proof}
        By our assumption, \( \varphi_0 - \psi_0, C |k|^p \in \mathcal{C}_p \), it holds by Theorem \ref{thm:controlOfBobMildSolDiffViaLinearizedMildSol} and Theorem \ref{thm:existenceAndPropOfLinearizedInelasticBob} that
        \begin{align*}
            |\varphi(t,k) - \psi(t,k)|
            &\leq \exp(t (-b_0 + \mathcal{L}_e - A^\mathsf{T} k \cdot { \nabla_{k} }))  |\varphi_0(k) - \psi_0(k)| \\
            &\leq C \exp(t (-b_0 + \mathcal{L}_e - A^\mathsf{T} k \cdot { \nabla_{k} })) |k|^p \\
            &\leq C u_p(t,k),
        \end{align*}
        which gives \eqref{pstable.rj}. Furthermore, \eqref{stable} is an immediate consequence of \eqref{pstable.rj} provided  \(  \|A\|<\lambda_p / p \).
    \end{proof}

    \begin{remark}
        In particular, since {the} Dirac mass at origin \( \varphi = 1 \) has unit mass, zero momentum and finite energy, and is a stationary solution to \eqref{equ:inelasticBobylev_mildForm}, for \( \|A\| <\lambda_2/2= \zeta / 2 \) where $\zeta$ is defined in \eqref{Defzeta}, every initial condition also with unit mass, zero momentum and finite energy weakly converges to {the Dirac mass $\varphi=1$}, and hence, in such case there is no nontrivial stationary solution of \eqref{equ:inelasticBobylev} with finite energy for \( z < 1 \).
    \end{remark}

    We summarize some differences between the inelastic and elastic cases. Let $f,g$ be solutions to \eqref{equ:inelasticBoltzmann} with the initial data $f_0$, $g_0$. On elastic collision, \( \lambda_2 = 0 \), then due to the monotonicity of $\lambda_p$ in \eqref{monola}, if we still want \eqref{stable} to hold for $(\varphi,\psi)=(\mathcal{F} f,\mathcal{F} g)$, it is necessary to consider \( p > 2 \) as in \cite{BobylevEtAl20SSA}, which requires the initial conditions $f_0(v)$ and $g_0(v)$ to have the same second moments. On the other hand, for inelastic collision, one can consider the case $p=2$ since we have \( \lambda_2 = \zeta > 0 \), which implies that, assuming \( \|A\| < \zeta / 2 \), on two initial probability measures \( f_0, g_0 \) with unit mass, zero momentum and finite energy, \eqref{stable} still holds. Notice that now we no longer require the second moments of $f_0$ and $g_0$ to be the same.

    \subsection{Stationary self-similar Fourier profile}\label{sec:statFourierProfile}
    The existence of {a} self-similar profile with finite second moments under small parameters can also be done in the Fourier framework. The stationary solution $\Phi=\Phi(k)$ for \eqref{equ:inelasticBobylevSelfSim} with self-similar parameter \( \beta \in \mathbb{R} \) {solves the equation}
    \begin{equation} \label{equ:inelasticBobylevSelfSimStationary}
        \Phi(k)
        = {\int_{0}^{\infty}} e^{-b_0 t} \widehat{Q_e^+}(\Phi,\Phi)(e^{-\beta t} e^{-t A^\mathsf{T}} k) \, {\mathrm{d}{t}}.
    \end{equation}
    Define
    $$
    B : {{k}^{\otimes 2}}=\sum_{i,j=1}^d B_{ij}k_ik_j,
    $$
    and
    \begin{equation}
        \label{def.Sp}
        \mathcal{S}_p \coloneqq \left\{
            \varphi \in \mathcal{K} 
            \mid 
            \|\varphi - (1 - \frac{1}{2} B : {{k}^{\otimes 2}})\|_{p} < \infty 
        \right\}.
    \end{equation}

    \begin{theorem} \label{thm:existenceOfBobStatProfileWithFiniteEnergy}
        Let \( p \in (2, 4] \). Then there exists \( \epsilon_p > 0 \) such that for \(\|A\|< \epsilon_p \) and
        \( 1 - z < \epsilon_p \), there exist \( \beta \in \mathbb{R} \), \( B \in \mathbb{R}^{d \times d} \) which is positive definite, and a unique \( \Phi \in \mathcal{S}_p\)
        that solves \eqref{equ:inelasticBobylevSelfSimStationary}.
    \end{theorem}

    \begin{proof}
        We still choose \( \beta \in \mathbb{R} \) and \( B \in \mathbb{R}^{d \times d} \) as in Theorem \ref{thm:existenceOfStatProfileWithSmallness} 
        such that \eqref{stableeq} is satisfied with $\beta$ having the largest real part. 
         Notice that we no longer require \( \mathrm{tr}(B) = 1 \) now.
        
        Since \( \Psi_0(k) := \exp(-\frac{1}{2} B : {{k}^{\otimes 2}}) \) is the characteristic function of a normal distribution, we then have \( \Psi_0 \in \mathcal{K} \), and thus \( \Psi_0 \in \mathcal{S}_p\)
        on \( 2<p \leq 4 \) with \( \mathcal{S}_p \) nonempty.
        
        Consider the Picard mapping \( P : \mathcal{S}_p \to \mathcal{K} \) by
        \begin{align*}
            P\Phi(t,k) = {\int_{0}^{\infty}} e^{-b_0 t} \widehat{Q_e^+}(\Phi,\Phi)(e^{-\beta t} e^{-t A^\mathsf{T}} k) \, {\mathrm{d}{t}}.
        \end{align*}
        The mapping is well-defined since \( P\Phi \) is a convex combination of characteristic functions
        $$
        k \mapsto b_0^{-1} \widehat{Q_e^+}(\Phi,\Phi)(e^{-\beta t} e^{-t A^\mathsf{T}} k).
        $$

        We first show that \( P\Psi_0 \in \mathcal{S}_p \).
        Since \( \Psi_0, P\Psi_0 \in \mathcal{K} \), we have \( |P\Psi_0 - \Psi_0| \leq C \).
        We expand $\Psi_0$ in $k$ such that \( \Psi_0(k) = 1 - \frac{1}{2} B : {{k}^{\otimes 2}} + O(|k|^4) \), using \( |k^\pm| \leq |k| \), it holds that
        \begin{align*}
            \widehat{Q_e^+}(\Psi_0,\Psi_0)(k)
            &= {\int_{\mathbb{S}^{d-1}}^{}} 
                b(\cos\theta)
                \big(1 - \frac{1}{2} B : {{(k^+)}^{\otimes 2}} + O(|k^+|^4)\big)
                \big(1 - \frac{1}{2} B : {{(k^-)}^{\otimes 2}} + O(|k^-|^4)\big)
             \, {\mathrm{d}{\sigma}} \\
            &= {\int_{\mathbb{S}^{d-1}}^{}} 
                b(\cos\theta)
                \big(1 - \frac{1}{2} B : ({{(k^+)}^{\otimes 2}} + {{(k^-)}^{\otimes 2}}) + O(|k|^4)\big)
             \, {\mathrm{d}{\sigma}} \\
            &= b_0 - \frac{1}{2} B :\big(
            (b_0 - \zeta) {{k}^{\otimes 2}}
            + z^2 \frac{d c_{11}}{2} (\frac{|k|^2}{d} I - {{k}^{\otimes 2}})
            \big) + O(|k|^4),
        \end{align*}
        which gives
        \begin{equation} \label{QPsi}
            \begin{aligned}[b]
                \widehat{Q_e^+}(\Psi_0,\Psi_0)(k) - b_0 \Psi_0(k)
                &= - \frac{1}{2} B : (
                - \zeta {{k}^{\otimes 2}} + z^2 \frac{d c_{11}}{2} (\frac{|k|^2}{d} I - {{k}^{\otimes 2}})
                ) + O(|k|^4) \\
                &= \frac{1}{2} {{k}^{\otimes 2}} : \big(
                \zeta B + \tilde{c} (B - \frac{\mathrm{tr}(B)}{d} I)
                \big) + O(|k|^4),
            \end{aligned}
        \end{equation}
        where \( \tilde{c} = z^2 \frac{d c_{11}}{2} > 0 \).
        We further have
        \begin{equation} \label{eqPhi0ek}
            \begin{aligned}[b]
                \Psi_0(e^{-\beta t} e^{-t A^\mathsf{T}} k)
                &= 1 - \frac{1}{2} e^{-2 \beta t} e^{-t A} B e^{-t A^\mathsf{T}} : {{k}^{\otimes 2}} + O(|e^{-\beta t} e^{-t A^\mathsf{T}} k|^4) \\
                &= \Psi_0(k) - \frac{1}{2} (e^{-2 \beta t} e^{-t A} B e^{-t A^\mathsf{T}} - B) : {{k}^{\otimes 2}} + O(|e^{-\beta t} e^{-t A^\mathsf{T}} k|^4) + O(|k|^4).
            \end{aligned}
        \end{equation}
        To bound the second term on the right-hand side above, by the selection of \( B \) such that \eqref{stableeq} holds, we have
        \begin{equation*}
            \zeta B + \tilde{c} (B - \frac{\mathrm{tr}(B)}{d} I) = -2 \beta B - A B - B A^\mathsf{T},
        \end{equation*}
        and a direct calculation shows that

        \begin{equation}\label{eqRemain}
            \frac{\mathrm{d}}{{\mathrm{d} t}} (e^{-c t} e^{-t A} B e^{-t A^\mathsf{T}})
            = -e^{-c t} e^{-t A} (c B + A B + B A^\mathsf{T}) e^{-t A^\mathsf{T}},
        \end{equation}
        for all \( c \in \mathbb{R} \). Thus, by choosing \( \|A\| \) and \( 1 - z \) to be sufficiently small such that \( 2 \|A\| + \zeta - 2 \tilde{\beta} \leq 2 \|A\| + 2|\tilde{\beta}| + \zeta < b_0 \), and then choosing $c=2\beta+b_0$ in the above equation, one gets from integrating in $t$ on both sides of \eqref{eqRemain} that
        \begin{equation}\label{control2}
            {\int_{0}^{\infty}} 
                e^{-(b_0 + 2 \beta) t} e^{-t A} (
                \zeta B + \tilde{c} (B - \frac{\mathrm{tr}(B)}{d} I) - b_0 B
                ) e^{-t A^\mathsf{T}}
             \, {\mathrm{d}{t}}
            = -B.
        \end{equation}
        Notice that to get the above equality, we use the fact that \( \|e^{-(2 \beta + b_0) t} e^{-t A} e^{-t A^\mathsf{T}}\| \leq e^{-(b_0 + 2 \tilde{\beta} - \zeta - 2 \|A\|) t} \to 0 \).
        For the last two terms on the right-hand side of \eqref{eqPhi0ek},  one has \( e^{-b_0 t} |e^{-\beta t} e^{-t A^\mathsf{T}} k|^4 \leq e^{-b_0 t} e^{-4 (\tilde{\beta} - \zeta / 2 - \|A\|) t} |k|^4 \). Then for sufficiently small \( \|A\| \) and \( 1 - z \), it holds that \( 4 (\tilde{\beta} - \zeta / 2 - \|A\|) < b_0 \), which yields
        \begin{align}\label{control3}
            {\int_{0}^{\infty}} e^{-b_0 t} (O(|k|^4) + O(|e^{-\beta t} e^{-t A^\mathsf{T}} k|^4)) \, {\mathrm{d}{t}} = O(|k|^4) .
        \end{align}
        It follows from \eqref{eqPhi0ek}, \eqref{control2} and \eqref{control3} that
        \begin{align*}
            P\Psi_0(k) - \Psi_0(k)
            &= {\int_{0}^{\infty}} 
                e^{-b_0 t} \big(
                (\widehat{Q_e^+}(\Psi_0,\Psi_0) - b_0 \Psi_0)(e^{-\beta t} e^{-t A^\mathsf{T}} k)
                + b_0 (\Psi_0(e^{-\beta t} e^{-t A^\mathsf{T}} k) - \Psi_0(k))
                \big)
             \, {\mathrm{d}{t}} \\
            &= \frac{1}{2}
            {\int_{0}^{\infty}} 
                e^{-b_0 t} \big(
                e^{-2 \beta t} e^{-t A} (
                \zeta B + \tilde{c} (B - \frac{\mathrm{tr}(B)}{d} I) - b_0 B
                ) e^{-t A^\mathsf{T}} : {{k}^{\otimes 2}} \\
                &\qquad\qquad\qquad\quad
                + b_0 B : {{k}^{\otimes 2}}
                + O(|k|^4) + O(|e^{-\beta t} e^{-t A^\mathsf{T}} k|^4)
                \big)
             \, {\mathrm{d}{t}} \\
            & = O(|k|^4).
        \end{align*}
        Hence, \( | P\Psi_0(k)-\Psi_0(k) | \leq C \min(1, |k|^4) \), and \( \|P\Psi_0 - \Psi_0\|_{p} \leq C\sup_{k \neq 0} \frac{\min(1, |k|^4)}{|k|^p} < \infty \) for \( 2<p \leq 4 \), which implies \( P\Psi_0 \in \mathcal{S}_p \).

        Now we can prove that  \( P : \mathcal{S}_p \to \mathcal{S}_p \) is a contraction mapping. Let \( \varphi(k), \psi(k) \in \mathcal{S}_p \), then {\( \|\varphi - \psi\|_{p} \leq \|\varphi - (1 - \frac{1}{2} B : {{k}^{\otimes 2}})\|_{p} + \|\psi - (1 - \frac{1}{2} B : {{k}^{\otimes 2}})\|_{p} < \infty \),} which, combining with \eqref{LeLip}, further yields
        \begin{equation} \label{contra}
            \begin{aligned}[b]
                |P\varphi(k) - P\psi(k)|
                &\leq {\int_{0}^{\infty}} e^{-b_0 t} |\widehat{Q_e^+}(\varphi,\varphi) - \widehat{Q_e^+}(\psi,\psi)(e^{-\beta t} e^{-t A^\mathsf{T}} k)| \, {\mathrm{d}{t}} \\
                &\leq {\int_{0}^{\infty}} e^{-b_0 t} \mathcal{L}_e|\varphi - \psi|(e^{-\beta t} e^{-t A^\mathsf{T}} k) \, {\mathrm{d}{t}} \\
                &\leq \|\varphi - \psi\|_{p} {\int_{0}^{\infty}} e^{-b_0 t} \mathcal{L}_e|k|^p(e^{-\beta t} e^{-t A^\mathsf{T}} k) \, {\mathrm{d}{t}} \\
                &\leq (b_0 - \lambda_p) \|\varphi - \psi\|_{p} {\int_{0}^{\infty}} e^{-b_0 t} e^{-p \beta t} \|e^{-t A^\mathsf{T}}\|^p |k|^p \, {\mathrm{d}{t}} \\
                &\leq (b_0 - \lambda_p) \|\varphi - \psi\|_{p} |k|^p {\int_{0}^{\infty}} 
                    e^{-p t (b_0 / p + \tilde{\beta} - \zeta / 2 - \|A\|)}
                 \, {\mathrm{d}{t}}.
            \end{aligned}
        \end{equation}

        Recalling the definitions of $\lambda_p$ and $\zeta$ in \eqref{Deflap} and \eqref{Defzeta}, for any \( p \in (2, 4] \), the mapping \( z \mapsto \lambda_p / p - \zeta / 2 \) is continuous, and at \( z = 1 \) it holds that \( (\lambda_p / p - \zeta / 2 )|_{z = 1} > 0 \). Therefore, for \( 1 - z > 0 \) sufficiently small, one has \( \lambda_p / p - \zeta / 2 > 0 \).

        Moreover, we choose \( \|A\| \) to be small enough such that $$ \|A\| - \tilde{\beta} \leq \|A\| + |\tilde{\beta}| < \lambda_p / p - \zeta / 2 \leq b_0 / p - \zeta / 2, $$ the last integral in \eqref{contra} is finite with \( (b_0 - \lambda_p) {\int_{0}^{\infty}} e^{-t (b_0 + p \tilde{\beta} - p \frac{\zeta}{2} - p \|A\|)} \, {\mathrm{d}{t}} = \frac{b_0 - \lambda_p}{b_0 + p (\tilde{\beta} - \zeta / 2 - \left\|A\right\|)} < 1 \),  which yields
        \begin{align}\label{contraP}
            \|P\varphi - P\psi\|_{p} \leq  \frac{b_0 - \lambda_p}{b_0 + p (\tilde{\beta} - \zeta / 2 - \|A\|)} \|\varphi - \psi\|_{p}.
        \end{align}
        For \( \varphi \in \mathcal{S}_p \), since \( P\Psi_0 \in \mathcal{S}_p \), we have from \eqref{contraP} that $$\|P\varphi - \Psi_0\|_{p} \leq \|P\varphi - P\Psi_0\|_{p} + \|P\Psi_0 - \Psi_0\|_{p} \leq C \|\varphi - \Psi_0\|_{p} + \|P\Psi_0 - \Psi_0\|_{p} < \infty,$$ then \( P\varphi \in \mathcal{S}_p \), which, together with \eqref{contraP}, implies that \( P : \mathcal{S}_p \to \mathcal{S}_p \) is a Banach contraction mapping, and has a unique fixed point \( \Phi \in \mathcal{S}_p \). The proof of Theorem \ref{thm:existenceOfBobStatProfileWithFiniteEnergy} is complete.
    \end{proof}

    \begin{remark}
        Recalling \eqref{def.Sp}, note that \( \mathcal{S}_p \supseteq \mathcal{S}_q \) on \( p \leq q \leq 4 \). This implies that, while the smallness conditions on \( \|A\| \) and \( 1 - z \) depend on \( p \), once \( \|A\| \) and \( 1 - z \) are fixed, the stationary solution \( \Phi_p \) obtained is the same for all \( p \in (2, 4] \) for which the required smallness conditions are satisfied.
    \end{remark}

    The convergence rate to a stationary profile can be improved from \( \lambda_p + 2 \beta - p \|A\| \) indicated by \eqref{stable} with the following theorem.

    \begin{theorem} \label{thm:convergenceToBobStatSol}
        Assume the same condition in Theorem \ref{thm:existenceOfBobStatProfileWithFiniteEnergy}, and let \( \varphi(t) \) be the mild solution of \eqref{equ:inelasticBobylevSelfSim} with self-similar scaling parameter \( \beta \) and initial condition \( \varphi_0 \in \mathcal{K}^{2} \) satisfying \( \|\varphi_0 - (1 - \frac{1}{2} C_0 : {{k}^{\otimes 2}})\|_p < \infty \) for some positive semidefinite matrix \( C_0 \).
        Then there exist \( \lambda > 0 \) depending on \( C_0 \), some constant $\nu>0$ and
        \begin{align}\label{Defeta}
            \eta = \frac{1}{2} \min\{\lambda_p + p (\beta - \|A\|), \nu + \zeta + 2 (\beta - \|A\|)\} > 0, \end{align}such that
        \begin{equation}\label{converFourier}
            |\varphi(t, k) - \Phi(\lambda k)| \leq C e^{-\eta t} (|k|^p + |k|^2).
        \end{equation}
    \end{theorem}

    \begin{proof}
        By assumption, \( C_0 \) is the second moment matrix of the probability measure corresponding to \( \varphi_0 \).
        Due to Theorem \ref{thm:mildSolAsWeakSolAndProperties}, \( \varphi(t,k) \) also has a corresponding finite second moment matrix \( C(t) \), which is positive semi-definite on \( t \geq 0 \), and satisfies \eqref{equ:secondMomentEquWithSelfSim} with initial condition \( C(0) = C_0 \).
        By the selection of \( \beta, B \), there exists \( \lambda,\nu,C \geq 0 \) such that \( \|C(t) - \lambda^2 B\| \leq Ce^{-\nu t} \).

        Let \( \phi(t, k) = \exp(-\frac{1}{2} C(t) : {{k}^{\otimes 2}}) \). By a similar calculation as in \eqref{QPsi}, one has
        \begin{equation*}
            \widehat{Q_e}(\phi,\phi)(t,k)
            = \widehat{Q_e^+}(\phi,\phi)(t,k) - b_0 \phi(t,k)
            = \frac{1}{2} {{k}^{\otimes 2}} : (\zeta C(t) + \tilde{c} (C(t) - \frac{\mathrm{tr}(C(t))}{d} I)) + O(|k|^4),
        \end{equation*}
        where the \( O(|k|^4) \) term is uniformly bounded in time since \( C(t)\) converges to \(\lambda^2 B \). Thus, it further holds that
        \begin{equation*}
            {\partial_{t}}\phi + A_\beta^\mathsf{T} k \cdot { \nabla_{k} }\phi - \widehat{Q_e}(\phi,\phi)
            = -\frac{1}{2} {{k}^{\otimes 2}} : \big(
            \frac{\mathrm{d}}{{\mathrm{d} t}} C(t) + A_\beta C(t) + C(t) A_\beta^\mathsf{T} + \zeta C(t)
            + \tilde{c} (C(t) - \frac{\mathrm{tr}(C(t))}{d} I)
            \big) + O(|k|^4),
        \end{equation*} which implies
        \begin{equation*}
            ({\partial_{t}} + A_\beta^\mathsf{T} \cdot { \nabla_{k} } + b_0)(\varphi - \phi)
            = (\widehat{Q_e^+}(\varphi,\varphi) - \widehat{Q_e^+}(\phi,\phi)) + \delta(t,k)
        \end{equation*}
        for some \( \delta(t,k) \) satisfying \( |\delta(t,k)| \leq C|k|^4 \) for some constant $C>0$ uniformly in time. With the \( \mathcal{L}_e \)-Lipschitz property \eqref{LeLip}, one gets
        \begin{align*}
            |\varphi(t,k) - \phi(t,k)|
            \leq& e^{-t (b_0 + A_\beta^\mathsf{T} k \cdot { \nabla_{k} })} |\varphi_0(k) - \phi_0(k)|+ {\int_{0}^{t}} 
                e^{-(t - \tau) (b_0 + A_\beta^\mathsf{T} k \cdot { \nabla_{k} })} |\widehat{Q_e^+}(\varphi,\varphi)(\tau,k) - \widehat{Q_e^+}(\phi,\phi)(\tau,k)|
             \, {\mathrm{d}{\tau}} \\
            &
            + {\int_{0}^{t}} 
                e^{-(t - \tau) (b_0 + A_\beta^\mathsf{T} k \cdot { \nabla_{k} })} |\delta(\tau,k)|
             \, {\mathrm{d}{\tau}} \\
            \leq& e^{-t (b_0 + A_\beta^\mathsf{T} k \cdot { \nabla_{k} })} |\varphi_0(k) - \phi_0(k)| \\
            &
            + {\int_{0}^{t}} 
                e^{-(t - \tau) (b_0 + A_\beta^\mathsf{T} k \cdot { \nabla_{k} })} \mathcal{L}_e|\varphi(\tau,k) - \phi(\tau,k)|
             \, {\mathrm{d}{\tau}}
            + {\int_{0}^{t}} 
                e^{-(t - \tau) (b_0 + A_\beta^\mathsf{T} k \cdot { \nabla_{k} })} |\delta(\tau,k)|
             \, {\mathrm{d}{\tau}}.
        \end{align*}
Letting 
\begin{equation*}
            Y(t,k) = e^{-t (b_0 + A_\beta^\mathsf{T} k \cdot { \nabla_{k} })} |\varphi_0(k) - \phi_0(k)|
            + {\int_{0}^{t}} 
                e^{-(t - \tau) (b_0 + A_\beta^\mathsf{T} k \cdot { \nabla_{k} })} \mathcal{L}_e Y(\tau,k)
             \, {\mathrm{d}{\tau}}
            + {\int_{0}^{t}} 
                e^{-(t - \tau) (b_0 + A_\beta^\mathsf{T} k \cdot { \nabla_{k} })} |\delta(\tau,k)|
             \, {\mathrm{d}{\tau}},
        \end{equation*}
it can be solved as
$$
Y(t,k) = e^{-t (A_\beta^\mathsf{T} k \cdot { \nabla_{k} } + b_0 - \mathcal{L}_e)}|\varphi_0(k) - \phi_0(k)| + {\int_{0}^{t}} e^{-(t - \tau) (A_\beta^\mathsf{T} k \cdot { \nabla_{k} } + b_0 - \mathcal{L}_e)}|\delta(\tau,k)| \, {\mathrm{d}{\tau}}.
$$
  By comparison principle,  with the assumption that \( |\varphi_0(k) - \phi_0(k)| = O(|k|^p) \), we have by the similar computation as in Theorem \ref{thm:bobControlOfPolynomialDiffViaUp} that
        \begin{align*}
            |\varphi(t,k) - \phi(t,k)|
            &\leq Y(t,k)
            = e^{-t (A_\beta^\mathsf{T} k \cdot { \nabla_{k} } + b_0 - \mathcal{L}_e)}|\varphi_0(k) - \phi_0(k)|
            + {\int_{0}^{t}} e^{-(t - \tau) (A_\beta^\mathsf{T} k \cdot { \nabla_{k} } + b_0 - \mathcal{L}_e)}|\delta(\tau,k)| \, {\mathrm{d}{\tau}} \\
            &\leq C e^{-t (A_\beta^\mathsf{T} k \cdot { \nabla_{k} } + b_0 - \mathcal{L}_e)}|k|^p
            + C{\int_{0}^{t}} 
                e^{-(t - \tau) (A_\beta^\mathsf{T} k \cdot { \nabla_{k} } + b_0 - \mathcal{L}_e)}|k|^4
             \, {\mathrm{d}{\tau}} \\
            &\leq Ce^{-p \beta t} u_p(t,k)
            + C{\int_{0}^{t}} 
                e^{-4 \beta (t - \tau)} u_4(t - \tau,k)
             \, {\mathrm{d}{\tau}} \\
            &= Ce^{-p t (\lambda_p / p + \beta - \|A\|)} |k|^p
            + C{\int_{0}^{t}} 
                e^{-4 \tau (\lambda_4 / 4 + \beta - \|A\|)} |k|^4
             \, {\mathrm{d}{\tau}},
        \end{align*}
        with constants independent of \( t \), where $u_p$ is defined in \eqref{Defup}. {By similar arguments as in the proof of} Theorem \ref{thm:existenceOfBobStatProfileWithFiniteEnergy}, for \( \|A\|, 1 - z \) sufficiently small, both exponents are negative, and so \( |\varphi(t,k) - \phi(t,k)| \leq C |k|^p +C |k|^4 \) for some constant $C$ independent of \( t \).

        If \( |k| \geq 1 \), then \( |\varphi(t,k) - \phi(t,k)| \leq C \leq C|k|^p \). If \( |k|< 1 \), then \( |\varphi(t,k) - \phi(t,k)| \leq C|k|^p + C|k|^4 \leq C|k|^p \). Then \( |\varphi(t,k) - \phi(t,k)| \leq C |k|^p \). Therefore, for \( T > 0 \),
        \begin{align*}
            |\varphi(T / 2, k) - \Phi(\lambda k)|
            &\leq |\varphi(T / 2, k) - \phi(T / 2, k)| + |\phi(T / 2, k) - \Phi(\lambda k)| \\
            &\leq C |k|^p + C\|C(T / 2) - \lambda^2 B\| |k|^2 \\
            &\leq C|k|^p + Ce^{-\nu T / 2} |k|^2,
        \end{align*}
        which, together with \eqref{phi-psi} for initial time \( t_0 = T / 2 \) and \eqref{controlTosbe}, yields
        \begin{align*}
            |\varphi(T, k) - \Phi(\lambda k)|
            &\leq \exp(-\frac{T}{2} (b_0 + A_\beta^\mathsf{T} k \cdot { \nabla_{k} } - \mathcal{L}_e)) |\varphi(T / 2, k) - \Phi(\lambda k)| \\
            &\leq C\exp(-\frac{T}{2} (b_0 + A_\beta^\mathsf{T} k \cdot { \nabla_{k} } - \mathcal{L}_e)) (|k|^p + e^{-\nu T / 2} |k|^2) \\
            &\leq Ce^{-p \beta T / 2} u_p(T / 2, k) + e^{-\nu T / 2} e^{-\beta T} u_2(T / 2, k) \\
            &= Ce^{-T \frac{p}{2} (\frac{\lambda_p}{p} + \beta - \|A\|)} |k|^p
            + e^{-T (\frac{\nu}{2} + \frac{\zeta}{2} + \beta - \|A\|)} |k|^2 \\
            &\leq Ce^{-\eta T} (|k|^p + |k|^2),
        \end{align*}
        with $\eta$ defined in \eqref{Defeta}, and \( \|A\|, 1 - z \) chosen to be sufficiently small. Hence, we get \eqref{converFourier} and the proof of Theorem \ref{thm:convergenceToBobStatSol} is complete.
    \end{proof}

    \begin{remark}
        Note that by splitting \( [0, T] \) into \( [0, r T] \cup [r T, T] \) with \( r \in (0, 1) \) instead of \( [0, T / 2] \cup [T / 2, T] \) in the proof above, we can obtain
        \begin{align*}
            \left|\varphi(T, k) - \Phi(\lambda k)\right|
            &\leq C e^{-(1 - r) T (\lambda_p + p(\beta - \|A\|))} |k|^p + Ce^{-\nu r T} e^{-(1 - r) T (\zeta + 2(\beta - \|A\|))} |k|^2 \\
            &\leq e^{-\eta(r) T} (|k|^p + |k|^2)
        \end{align*}
        with
        \begin{equation*}
            \eta(r)
            = \min(
            (1 - r) (\lambda_p + p (\beta - \|A\|)),
            r \nu + (1 - r) (\zeta + 2 (\beta - \|A\|))),
        \end{equation*}
        which is positive if \( \|A\| \) is {sufficiently} small. If \( \|A\| - \beta < \frac{\lambda_p - \zeta}{p - 2} \), then it holds $$ \max\eta(r) = \frac{\nu (\lambda_p + p (\beta - \|A\|))}{\nu + \lambda_p - \zeta + (p - 2) (\beta - \|A\|)} ,$$ otherwise by selecting \( r  \) sufficiently small, we have \( \eta(r) = \lambda_p + p (\beta - \|A\|) - \epsilon \) {for arbitrary $\epsilon$, provided that it is small enough}.
    \end{remark}


    Both Theorem \ref{thm:existenceOfBobStatProfileWithFiniteEnergy} and Theorem \ref{thm:convergenceToBobStatSol} are given in the sense of Fourier transform. We also have the corresponding results in terms of distribution functions.

    \begin{theorem} \label{DFss}
        Let \( p \in (2, 4] \). Then there exists \( \epsilon_p > 0 \) such that for \(\|A\|< \epsilon_p \) and
        \( 1 - z < \epsilon_p \), there exists a self-similar solution $f$ to \eqref{equ:inelasticBoltzmann} in the form:
        \begin{align*}
            f(t,v)=e^{-d \beta t} G(e^{-\beta t}v),
        \end{align*}
        with {$G$ being Radon probability measure and}
        \begin{align*}
            \int_{\mathbb{R}^d}G(v)|v|^p \mathrm{d}{v} <\infty.
        \end{align*}
    \end{theorem}
    
    The proof of Theorem \ref{DFss} is a direct consequence of Theorem \ref{thm:existenceOfBobStatProfileWithFiniteEnergy} and thus omitted here. The following theorem gives the convergence {for solutions of \eqref{equ:inelasticBoltzmann} to the self similar profile}.
    
    \begin{theorem} \label{thm:convToDFss}
        Suppose \( f \in \mathit{C}([0, \infty), \mathcal{M}_+) \) is a weak solution of \eqref{equ:inelasticBoltzmann} on initial condition \( f_0(v) \in \mathcal{M}_+ \) satisfying $\int_{\mathbb{R}^d}f_0(v)|v|^p \mathrm{d}{v} <\infty$ for any $p>2$, then for $\beta$ and $\lambda$ chosen in Theorem \ref{thm:convergenceToBobStatSol} and $G$ defined in Theorem \ref{DFss}, it holds that
        \begin{align}\label{convergence}
            e^{d \beta t}f(t,e^{\beta t} v)\rightarrow \lambda^{-d}G(\lambda^{-1}v),
        \end{align}
        in the weak topology of $\mathcal{M}_+$ as $t\rightarrow 0$.
    \end{theorem}
    
    \begin{proof}
        The weak convergence \eqref{convergence} follows from the combination of Theorem \ref{thm:convergenceToBobStatSol}, the properties of the Fourier transform that relates the group of translations with the multiplication by a phase, the
        scaling properties of the Fourier transform and the uniform convergence of characteristic functions in compact sets. See also \cite{BobylevEtAl20SSA,Feller09IPT2}.
    \end{proof}

    Combining Theorem \ref{DFss} and Theorem \ref{thm:convToDFss}, we obtain the result in Theorem \ref{thm:mainResult-convToProfile}.

    \medskip
    \noindent {\bf Acknowledgment:}\,
JAC was supported by the Advanced Grant Nonlocal-CPD (Nonlocal PDEs for Complex Particle Dynamics: Phase Transitions, Patterns and Synchronization) of the European Research Council Executive Agency (ERC) under the European Union’s Horizon 2020 research and innovation programme (grant agreement No. 883363). JAC was also partially supported by the EPSRC grant number EP/V051121/1. RJD was partially supported by the General Research Fund (Project No.~14303321) from RGC of Hong Kongcand also partially supported by the grant from the National Natural Science Foundation of China (Project No.~12425109). The authors would thank the anonymous referees for valuable and helpful comments on the manuscript. 

    \medskip

    \noindent{\bf Conflict of Interest:} The authors declare that they have no conflict of interest.

\end{document}